\documentclass[10pt,a4paper]{article}
\usepackage{amsmath}
\usepackage{amsfonts}
\usepackage{comment}
\usepackage{amssymb}
\usepackage{enumerate}
\usepackage{authblk}
\usepackage[left=2.5cm,right=2.5cm,top=2.5cm,bottom=2.5cm]{geometry}
\usepackage{amsthm}
\usepackage{tikz}
\usetikzlibrary{decorations.pathreplacing,calligraphy}

\newcommand{\Exterior}{\mathchoice{{\textstyle\bigwedge}}%
  {{\bigwedge}}%
    {{\textstyle\wedge}}%
   {{\scriptstyle\wedge}}}
\newcommand{\C}{\mathbb{C}}

\newcommand{\R}{\mathbb{R}}

\newcommand{\CP}{\mathbb{CP}}

\newcommand{\ra}{\rightarrow}

 \newcommand{\Real}{{\operatorname{Re}}}
 
\usepackage{graphicx}
\usepackage{tikz-cd} 
\theoremstyle{definition}

\theoremstyle{plain}
\newtheorem{theorem}{Theorem}[section]
\newtheorem{corollary}[theorem]{Corollary}
\newtheorem{lemma}[theorem]{Lemma}

\newtheorem{prop}[theorem]{Proposition}
\newtheorem{proposition}[theorem]{Proposition}

\newtheorem{remark}[theorem]{Remark}
\theoremstyle{plain}
\newtheorem{maintheorem}{Theorem}
 
\author{Jakob Stein}
\affil{University College London, jakob.stein.19@ucl.ac.uk}
\date{}
\title{$SU(2)^2$-invariant gauge theory on asymptotically conical Calabi-Yau 3-folds}

\begin{document}

\maketitle
\begin{abstract}
We give a complete description of the behaviour of Calabi-Yau instantons and monopoles with an $SU(2)^2$-symmetry, on Calabi-Yau 3-folds with asymptotically conical geometry and $SU(2)^2$ acting with co-homogeneity one. We consider gauge theory on the smoothing and small resolution of the conifold, and on the canonical bundle of $\CP^1 \times \CP^1$, with their known asymptotically conical co-homogeneity one Calabi-Yau metrics, and find new one-parameter families of invariant instantons. We also entirely classify the relevant moduli-spaces of instantons and monopoles satisfying a natural curvature decay condition, and show that the expected bubbling phenomena occur in  these families of instantons.    
\end{abstract}

\section{Introduction}

On Calabi-Yau (CY) 3-folds, Riemannian manifolds of real dimension six with holonomy contained in $SU(3)$, one can define analogues of the Bogomol'nyi monopole and anti-self-dual equations found in dimensions three and four. These analogues are referred to as the \textit{Calabi-Yau monopole equations}, and the \textit{Calabi-Yau instanton equations} respectively, and are defined with the additional data of a fixed principal bundle over the 3-fold. It has been conjectured in  \cite{donaldson1996gauge}, \cite{donaldson2009gauge} that one might be able to use the moduli-space of their solutions to construct invariants of the underlying 3-folds. 

However, the analysis of the resulting partial differential equations can be difficult in general: in this article, we focus on the $SU(2)^2$-symmetric setting for both the Calabi-Yau structure and the bundle, so that these equations can be written as ordinary differential equations of a single variable. We further restrict our investigations to bundles with structure group of rank one, and asymptotically conical (AC) geometries, i.e. the Calabi-Yau is diffeomorphic to a cone outside of a compact subset\footnote{the presence of continuous symmetries for full holonomy $SU(3)$ necessitates the manifold be non-compact.}, with a Riemannian metric converging in a suitable sense to the corresponding metric cone. 

In this setting, we are able to explicitly describe the aforementioned moduli-spaces and their structure, which will hopefully shed more light on the wider subject: for example, how the underlying geometry interacts with the gauge theory, and how to construct local models for solutions on compact three-folds with isolated conical singularities, and their de-singularisations. We also prove the relevant bubbling and compactness theorems for these moduli-spaces, in line with the general picture laid out in \cite{donaldson1996gauge}, \cite{donaldson2009gauge}.  

Another related motivation for studying $SU(2)^2$-invariant AC Calabi-Yau gauge theory are the potential applications to an analogous notion of gauge theory on Riemannian manifolds with  exceptional holonomy group $G_2$. In particular, Foscolo, Haskins, and Nordstr\"{o}m have recently constructed one-parameter families of $SU(2)^2$-invariant $G_2$-metrics in \cite{Foscolo2018} with asymptotically locally conical (ALC) geometry at infinity, i.e. outside of a compact subset, these metrics converge to a circle fibration over a Calabi-Yau cone, with fibres of some length $\ell >0$. These families collapse to invariant AC Calabi-Yau 3-folds in the limit as $\ell \rightarrow 0$, and one may be able to use the invariant Calabi-Yau gauge theory constructed in this article to construct invariant $G_2$ instantons near the collapsed limit. Far from the collapsed limit, the families of $G_2$-metrics have AC geometry when $\ell \rightarrow \infty$, see \cite{goncalog2} for partial results comparing invariant instantons on $G_2$-metrics with ALC and AC asymptotics.      
\subsection{Overview}
Let $\left( M^6 ,\omega, \Omega \right)$ be a Calabi-Yau $3$-fold, where $\omega$ denotes the K\"{a}hler form, and $\Omega$ denotes the holomorphic volume form on $M$ such that $\tfrac{1}{3!}\omega^3 = \tfrac{1}{2^3} \Omega \wedge \bar{\Omega}$, and fix a principal $G$-bundle $P \rightarrow M$ with a compact semi-simple Lie group $G$. The pair $(A, \Phi)$, for some connection $A$ on $P$ and non-trivial $\Phi\in \Omega^0 \left( \text{Ad}P \right)$, is called a (Calabi-Yau) \textit{monopole} if it satisfies the \textit{Calabi-Yau monopole equations}: 
\begin{align} \label{CYmonopoleI}
F_{A} \wedge \omega^2 = 0&  &F_{A} \wedge \mathrm{Re}\Omega = * d_A \Phi 
\end{align}
where $*$ is the Hodge star of the Riemannian metric defined by $\left( \omega, \Omega \right)$, $F_A \in \Omega^2 \left( \text{Ad}P \right)$ is the curvature of $A$, and $d_A : \Omega^0 \left( \text{Ad}P \right) \rightarrow \Omega^1 \left( \text{Ad}P \right)$ is the induced covariant derivative. We refer to the section $\Phi$ as the \textit{Higgs field} for this monopole. 

We obtain the \textit{Calabi-Yau instanton equations} for a connection $A$ on $P$ by setting $\Phi=0$ in \eqref{CYmonopoleI}: 
\begin{align} \label{CYinstantonI}
F_{A} \wedge \omega^2 = 0&  &F_{A} \wedge \mathrm{Re} \Omega= 0 
\end{align}
Note that if a monopole $(A, \Phi)$ has $d_A \Phi = 0$, then $A$ is also a (Calabi-Yau) \textit{instanton}, i.e. a solution of \eqref{CYinstantonI}, but the existence of a non-trivial parallel section $\Phi$ implies that the connection $A$ must be reducible in this case.  

In terms of the complex geometry, the first condition of \eqref{CYinstantonI} says that $F_A$ is a primitive Lie algebra-valued two-form, while the second condition says it is of type $\left(1,1 \right)$. Furthermore, it is not hard to prove that instantons minimize the Yang-Mills energy functional $\mathcal{YM}(A):= \int_M \lvert F_A \rvert^2$ on the space of connections on $P$, where we take point-wise norms with respect to some ad-invariant metric on the Lie algebra of $G$. Hence, on the special unitary frame bundle $SU(E)$ of some hermitian vector bundle $E$ over $M$ with trivial determinant bundle, a Calabi-Yau instanton is also referred to as a Hermitian Yang-Mills (HYM) connection in the literature. 

When $G$ is abelian, \eqref{CYmonopoleI} and  \eqref{CYinstantonI} are linear equations, and the moduli-space of their solutions are well-understood: if $G=U(1)$ for example, any two-form on $M$ which is an instanton in the sense of \eqref{CYinstantonI} is harmonic, with the converse holding when $\left(M, \omega, \Omega \right)$ is compact with full holonomy $SU(3)$. Even when $M$ is non-compact, every $U(1)$-bundle carries a unique Calabi-Yau instanton with decaying curvature when $\left( \omega, \Omega \right)$ is asymptotically conical with full holonomy $SU(3)$ by \cite[Theorem 5.12]{FoscoloALC}. For non-abelian gauge groups, one usually seeks a description of the gauge theory starting with the next simplest case of rank one groups: in particular, without loss of generality\footnote{gauge group $SO(3)$ always lifts to $SU(2)$ in our invariant setting, see Proposition \ref{so3prop} of the appendix.}, we will always take the gauge group to be $SU(2)$ in this article. 

The Calabi-Yau monopole equations were first studied in the $SU(2)^2$-invariant setting in \cite{Oliveira2015}, for the asymptotically conical metric of Stenzel \cite{stenzel1993} on the cotangent bundle of $S^3$. There is a one-parameter family of invariant monopoles for this metric, with a single explicit instanton \cite[Theorem 2]{Oliveira2015} appearing at the boundary of this family when the Higgs field vanishes.  In this article, we will independently verify this claim using new proofs, as well as proving that the explicit instanton actually lies in a one-parameter family of invariant instantons for this metric. We also describe the invariant gauge theory for all the other known examples of $SU(2)^2$-invariant AC Calabi-Yau metrics, namely the metric of Candelas and de la Ossa \cite{CANDELAS1990246} on the small resolution of the conifold $\mathcal{O}(-1)\oplus \mathcal{O}(-1)$ over $\mathbb{CP}^1$, and the metric of Calabi \cite{calabi}, later generalised to a one-parameter family by Pando-Zayas and Tseytlin \cite{PandoZayas}, on the canonical bundle $\mathcal{O}(-2,-2)$ of $\mathbb{CP}^1 \times \mathbb{CP}^1$. Here, the action of $SU(2)^2$ on these 3-folds is with \textit{co-homogeniety one} i.e. the generic $SU(2)^2$-orbit is co-dimension one. 

To understand the various components of these gauge-theoretic moduli-spaces, we must first discuss fixing the asymptotic behaviour of solutions. A natural condition on a solution of \eqref{CYmonopoleI} on an asymptotically conical metric is that it converges at the conical end to some model solution $\left(A_\infty, \Phi_\infty \right)$ on the cone, pulled back from the link. Concretely, up to double-cover, the metrics on $T^* S^3$, $\mathcal{O}(-1)\oplus \mathcal{O}(-1)$, and $\mathcal{O}(-2,-2)$ all share the same asymptotic cone with link $S^2 \times S^3$, and we have the following potential invariant model solutions: either we have a flat connection with a trivial Higgs field, or we have the unique non-flat invariant instanton pulled back from $S^2 \times S^3$, which we denote $A^\mathrm{can}$, with a possibly non-trivial parallel Higgs field\footnote{these monopoles $\left( A^\mathrm{can} , \Phi_m \right)$ pulled back from the link actually come in a one-parameter family, parametrised by the \textit{mass} $m = | \Phi_m |>0$. This is explained in more detail in \cite{Oliveira2015}.}. 

On these asymptotically conical metrics, we find four distinct possibilities for any invariant irreducible solution $\left(A, \Phi \right)$ to the monopole equations: (i) the curvature does not decay quadratically, i.e. $t^2 |F_A|$ is unbounded as $t\rightarrow \infty$, where $t$ is the radial parameter of the cone, and we take norms with respect to the cone metric, (ii)  $\left(A, \Phi \right)$ is an  invariant monopole which is asymptotic to $A^\mathrm{can}$ with a non-trivial Higgs field as $t\rightarrow \infty$, (iii)  $\Phi =0$, $A$ is an invariant instanton which is asymptotic to $A^\mathrm{can}$ as $t\rightarrow \infty$, (iv) $\Phi =0$, $A$ is invariant instanton which is asymptotic to a flat connection as $t\rightarrow \infty$\footnote{one can also show that for (ii), (iii), solutions have \textit{exactly} quadratic curvature decay, while for (iv), solutions have curvature decaying \textit{faster} than quadratically, and moreover, have curvature bounded in $L^2$-norm.}. 

We shall restrict to cases (ii)-(iv) by only considering invariant solutions with quadratic curvature decay. In general, this is a natural assumption to make for solutions on asymptotically conical metrics, since solutions on the cone converging to some model solution have curvature decaying (at least) as a two-form on link of the cone. As far as the author is aware, this article is the first situation for which we have a complete description of this moduli-space for the invariant co-homogeneity one gauge theory. Also, although we were unable to prove this in full generality, we conjecture that situation (i) does not actually arise, i.e. any invariant solution to the monopole equations on $T^* S^3$, $\mathcal{O}(-1)\oplus \mathcal{O}(-1)$, and $\mathcal{O}(-2,-2)$ without quadratic curvature decay must blow up in finite time. 

We now summarise our main results. For the metric of Stenzel, there is a single $SU(2)$-bundle admitting irreducible invariant connections, which we denote $P_\mathrm{Id}$, and we find a one-parameter family of instantons, and a one-parameter family of monopoles:  
\begin{maintheorem} \label{stenzelthm0} In a neighbourhood of $S^3 \subset T^* S^3$, up to gauge, invariant solutions to the monopole equations are in a two-parameter family $\left(S, \Phi\right)_{\xi,\chi}$, containing a one-parameter family of invariant instantons with $\chi =0$. Moreover $\left(S, \Phi\right)_{\xi,\chi}$ extends over all of $T^* S^3$ when: 
\begin{enumerate}[\normalfont(i)]
\item $\xi \in \left(-1,1 \right), \chi =0 $, as an irreducible instanton asymptotic to $A^\mathrm{can}$ at infinity, \item $\xi = \pm 1, \chi =0 $, as a flat connection, \item $\xi=0, \chi \in \left(0,\infty \right) $, as an irreducible monopole asymptotic to $A^\mathrm{can}$ with a non-trivial parallel Higgs field at infinity. 
\end{enumerate}
Otherwise, $\left(S, \Phi\right)_{\xi,\chi}$ cannot extend over $T^* S^3$ with quadratically decaying curvature.  
\end{maintheorem} 
See Proposition \ref{localmonopolesmoothing} for a proof of the local statement, Theorem \ref{stenzelthm} parts (i), (ii), and Proposition \ref{thmmonopoleglobalB} for (iii). The existence of the one-parameter family of monopoles $\left(S, \Phi\right)_{\chi}:= \left(S, \Phi\right)_{0,\chi}$, $\chi \in \left(0,\infty \right) $, and the instanton $S_0 := \left(S, \Phi\right)_{0,0}$ was already established in \cite{Oliveira2015}, which considered only local solutions $\left(S, \Phi\right)_{\xi,\chi}$ with $\xi=0$: we fix a gap in the proof of \cite[Theorem 1]{Oliveira2015} by showing these are all the invariant monopoles with quadratic curvature decay. We also note here that there is a (non-equivariant) isometric involution of $T^*S^3$, arising from the map exchanging the factors of $SU(2)$ in each $SU(2)^2$-orbit, which sends $\left(S, \Phi\right)_{\xi,\chi} \mapsto \left(S, \Phi\right)_{-\xi,\chi}$.

For the small resolution $\mathcal{O}(-1)\oplus \mathcal{O}(-1)$, there are two $SU(2)$-bundles admitting invariant irreducible connections, denoted $P_{0,\mathrm{Id}}$ and $P_{1, \mathbf{0}}$, and these are equivariantly isomorphic over the complement $\mathcal{O}(-1)\oplus \mathcal{O}(-1) \setminus \mathbb{CP}^1$. We find that each bundle carries a one-parameter family of instantons $R_\epsilon$, $R'_{\epsilon'}$, respectively, and the family $R'_{\epsilon'}$ contains an invariant abelian instanton $R'_0$: 
\begin{maintheorem} \label{candelasthm0} In a neighbourhood of $\mathbb{CP}^1 \subset \mathcal{O}(-1)\oplus \mathcal{O}(-1)$, invariant instantons are in two one-parameter families $R_\epsilon$ and $R'_{\epsilon'}$, $\epsilon' \in \left[ 0 ,\infty \right)$, up to gauge. Moreover $R_\epsilon$, $R'_{\epsilon'}$ extends over all of $\mathcal{O}(-1)\oplus \mathcal{O}(-1)$ when:
\begin{enumerate}[\normalfont(i)]
\item  $\epsilon \in \left(0, \infty \right)$, as an irreducible instanton asymptotic to $A^\mathrm{can}$ at infinity,
\item $\epsilon' \in \left[ 0, 1 \right)$, as an instanton asymptotic to $A^\mathrm{can}$ at infinity, which is abelian if $\epsilon'=0$ and irreducible otherwise,
\item $\epsilon=0$ or $\epsilon'=1$, as a flat connection.
\end{enumerate} 
Otherwise, $R_\epsilon$, $R'_{\epsilon'}$ cannot extend over $\mathcal{O}(-1)\oplus \mathcal{O}(-1)$ with quadratically decaying curvature.
\end{maintheorem}
See Propositions \ref{localresol1}, \ref{localresol6} for a proof of the local statement, Theorems \ref{candelasthm1}, \ref{candelasthm2} for parts (i)-(iii). 

We can also show that, in the limit $\epsilon \rightarrow \infty$, i.e. as the curvature of the invariant family $R_\epsilon$ blows up on the calibrated co-dimension four $\mathbb{CP}^1$, we get the expected bubbling and removable-singularity phenomena:
\begin{maintheorem} \label{candelasthm00} Let $R_\epsilon$ be the one-parameter family of invariant instantons and $R'_0$ the invariant abelian instanton extending over $\mathcal{O}(-1)\oplus \mathcal{O}(-1)$. Then, in the limit $\epsilon \rightarrow \infty$: 
\begin{enumerate}[\normalfont(i)] 
\item Up to an appropriate rescaling, $R_\epsilon$ bubbles off a family of anti-self-dual connections along $\mathbb{CP}^1 \subset \mathcal{O}(-1)\oplus \mathcal{O}(-1)$.
\item Without this rescaling, $R_\epsilon$ converges uniformly to $R'_0$ on compact subsets of $\mathcal{O}(-1)\oplus \mathcal{O}(-1) \setminus \mathbb{CP}^1$. 
\end{enumerate}
\end{maintheorem} 
See Theorem \ref{candelasthm1instanton} for proofs and a more precise statement of these results. The proof of Theorem \ref{candelasthm00} is more involved than for a similar co-homogeneity one bubbling theorem for instantons found in \cite[Theorem 2]{goncalog2}: everything was explicit in that case, whereas we must genuinely prove (i) here to obtain the relevant compactification result (ii).    

There are countably many bundles over $\mathcal{O}(-2,-2)$ admitting irreducible invariant connections, which we denote $P_{1-l,l}$ for $l \in \mathbb{Z}$. The number $l \in \mathbb{Z}$ can be understood topologically by associating a rank two complex vector bundle to $P_{1-l,l}$ via the standard representation: this associated bundle splits into a direct sum of line bundles pulled back from $\mathcal{O}(\pm(1-l),\pm l) \ra \CP^1 \times \CP^1$. Each bundle $P_{1-l,l}$ carries a one-parameter family of instantons $Q^l_{\alpha_l}$ similar to the family $R'_{\epsilon'}$ of Theorem \ref{candelasthm0}: $Q^l_{0}$ is abelian, $Q^l_{\alpha_l}$ is asymptotic to $A^\mathrm{can}$ at infinity when the parameter $\alpha_l\geq 0$ is less than some finite critical value $\alpha^\mathrm{crit}_l$, and the asymptotic behaviour of this family jumps to a flat connection at the critical value. However, there is a new phenomenon on $\mathcal{O}(-2,-2)$, as the instantons $Q^l_{\alpha^\mathrm{crit}_l}$ are not themselves flat when $l\neq 0,1$: they are rigid in the moduli-space of invariant, irreducible instantons with these asymptotics.
\begin{maintheorem} \label{pandozayasthm0} In a neighbourhood of $\CP^1 \times \CP^1 \subset \mathcal{O}(-2,-2)$, invariant instantons are in countably many one-parameter families $Q^l_{\alpha_l}$, $l\in \mathbb{Z}$, $\alpha_l \in \left[0, \infty \right)$, up to gauge. Moreover, $Q^l_{\alpha_l}$ extends over all of $\mathcal{O}(-2,-2)$ when: 
\begin{enumerate}[\normalfont(i)]
\item$\alpha_l \in \left[ 0, \alpha^\mathrm{crit}_l \right)$ for some $\alpha^\mathrm{crit}_l \in \left( 0, \infty \right)$, as an instanton asymptotic to $A^\mathrm{can}$ at infinity, which is abelian if $\alpha_l=0$ and irreducible otherwise, \item $l = 0,1$, $\alpha_l = \alpha^\mathrm{crit}_l$ as a flat connection, \item  $l \neq 0,1$, $\alpha_l = \alpha^\mathrm{crit}_l$ as an irreducible instanton asymptotic to a flat connection at infinity. 
\end{enumerate}
Otherwise, $Q^l_{\alpha_l}$ cannot extend over $\mathcal{O}(-2,-2)$ with quadratically decaying curvature.   
\end{maintheorem}
See Proposition \ref{localresol4} for a proof of the local statement, and Theorems \ref{pandozayasthm1}, \ref{pandozayasthm2} for parts (i)-(iii). See also the end of \S \ref{bubbling} for a further discussion of the behaviour of instantons on $\mathcal{O}(-2,-2)$. 

In the final result, the proof of which can be found in Proposition \ref{thmmonopoleglobalA}, we show that Theorems \ref{stenzelthm0} - \ref{pandozayasthm0} fully describe the moduli-space of the $SU(2)^2$-invariant Calabi-Yau gauge theory: 
\begin{maintheorem} \label{monopoletheorem0}
There are no irreducible, invariant monopoles on $\mathcal{O}(-2,-2)$ or $\mathcal{O}(-1) \oplus \mathcal{O}(-1)$ with quadratically decaying curvature.
\end{maintheorem} 

\subsection{Plan of paper}
For the rest of the introduction, we summarise the structure of this article. 

Throughout the following, if a manifold $M$ has a co-homogeneity one action by Lie group $K$, with exactly one exceptional isotropy subgroup $H'$, and generic isotropy subgroup $H$, we will denote the sequence $H \subset H' \subseteq K$ as the \textit{group diagram} of $M$. We will refer to the generic $K$-orbit $K/H$ as the \textit{principal} orbit, the orbit $K/H'$ as the \textit{singular} orbit, and the union of all generic $K$-orbits as the \textit{space of principal orbits}. In order to fix conventions, we start with a preliminary introduction to the geometry of co-homogeneity one Calabi Yau metrics in \S \ref{sectionprelim}, in the case $K=SU(2)^2$, and $H$ is either the diagonal subgroup $\triangle U(1)$ or $\triangle U(1) \times \mathbb{Z}_2$, i.e. we describe the Calabi-Yau metrics on $\mathcal{O}(-1)\oplus \mathcal{O}(-1)$, $T^* S^3$, and $\mathcal{O}(-2,-2)$. 
 
We proceed with the main goal of the article in \S \ref{sectiongauge}: we consider the space of connections, Higgs fields, and $SU(2)$-bundles over these manifolds that are invariant under the $SU(2)^2$-action, and write down the monopole equations in this invariant setting. We describe the gauge theory on the complement of the singular orbit by pulling back invariant bundles over the principal orbit in \S \ref{invmonopole}, giving us some ODE system for our connection and Higgs field. Invariant bundles over the principal orbit are classified by an integer, but only one of these bundles, denoted $P_1$, admits irreducible connections. We write down the ODE system for this bundle explicitly in Proposition \ref{propgen}. We also briefly mention the reducible solutions to these equations in \S \ref{abeliansection}, which are explicit.

We cannot generically expect to find explicit solutions in the irreducible case, but by imposing that the bundle data extends to the singular orbit, we can describe the space of solutions to the ODEs near the singular orbit using a power-series. In \S \ref{sectionlocalsolutions}, we will find that these local solutions to the monopole equations are always in a two-parameter family for each extension of the bundle $P_1$ to the singular orbit, and we can obtain a local one-parameter family of instantons by setting one of these two parameters to zero. 

The discussion of boundary conditions for extending the invariant bundle data to the singular orbit is relegated to Appendix \ref{sectioncohobundles}. Using the analysis of Eschenburg-Wang \cite{Eschenburg2000} on invariant tensors, which can be adapted to (adjoint-valued) forms, these are just representation-theoretic computations. 

We dedicate the remaining sections to finding a qualitative description of the behaviour of the local solutions in \S \ref{sectionlocalsolutions} as we move away from the singular orbit. In \S \ref{odessection1}, using the existence of invariant sets for these ODE systems, we determine the asymptotic behaviour of the local instanton solutions to obtain Theorems \ref{candelasthm0}, \ref{pandozayasthm0}, and parts (i),(ii) of Theorem \ref{stenzelthm0}. To prove existence of the critical value of the parameter $\alpha_l$ in Theorem \ref{pandozayasthm0} when $l \neq 0,1$, we must also employ a rescaling argument along the fibres of $\mathcal{O}(-2,-2)$, and we prove uniqueness via some comparison results allowing us to compare solutions away from the singular orbit for different values of $\alpha_l$.

We continue discussing rescaling arguments in \S \ref{bubbling}. To show Theorem \ref{candelasthm00}, we can consider an adiabatic limit in which we shrink the metric on $\mathcal{O}(-1)\oplus \mathcal{O}(-1)$ along the fibre. We prove that in this limit, as $\epsilon \rightarrow \infty$, a rescaling of the one-parameter family of solutions $R_\epsilon$ to the ODEs converges to the standard anti-self-dual connection on $\C^2$, and use this result to prove the convergence of the solution $R_\epsilon$ as $\epsilon \rightarrow \infty$. The general picture is that the solution curve $R_\epsilon$ breaks into two pieces in this limit, the first being the anti-self-dual connection, which is only traversed in non-zero time if we rescale, and the second being an abelian instanton $R'_0$. 

We also include an extended remark on bubbling phenomena for instantons on $\mathcal{O}(-2,-2)$: we consider a limit in which the metric is close to the simplest (non-trivial) example of an asymptotically locally Euclidean (ALE) fibration: a copy of the Eguchi-Hansion metric on the total space of the co-tangent bundle of $\CP^1$, fibred over the standard metric on $\CP^1$. As one might expect, in this limit, we find that the families of instantons on $\mathcal{O}(-2,-2)$, suitably rescaled, are close to corresponding families of anti-self-dual connections for the Eguchi-Hanson metric. Although this result was ultimately unnecessary for proving the main theorems of this article, they provide a way to understand the moduli-space over $\mathcal{O}(-2,-2)$ in terms of the moduli-spaces of anti-self-dual connections constructed by Nakajima in \cite{nakajima}.   

Finally, in \S \ref{odessection2}, we analyse the behaviour of the full system of the monopole equations away from the singular orbit to prove Theorem \ref{monopoletheorem0}, and the final part of Theorem \ref{stenzelthm0}. We show that, aside from the one-parameter family of monopoles already found in \cite{Oliveira2015} and the instantons described in the previous sections, any other member of the local two-parameter families of monopoles from \S \ref{sectionlocalsolutions} cannot have quadratically decaying curvature.  
\subsection*{Acknowledgements} The author would like to thank Jason Lotay, Gon\c calo Oliveira, and Matt Turner for all their helpful comments and discussions. A special thanks goes to the author's  supervisor Lorenzo Foscolo, without whom this work would not have been possible, and to the Royal Society, who funded this work through a studentship supported by the Research Fellows Enhancement Award 2017 RGF$\backslash$EA$\backslash$180171. 
\section{Preliminaries} \label{sectionprelim} 
\subsection{Calabi-Yau structures, cones, and co-homogeneity one manifolds} \label{sectiondefine}
Before considering Calabi-Yau structures in (real) dimension $6$, let us recall some general definitions from \cite{conti2005} in dimension $5$. We let $N$ be a real $5$-dimensional manifold equipped with an $SU(2)$-reduction of the frame bundle: this gives a unique Riemannian metric and orientation on $N$ compatible with such a reduction, via the inclusion $SU(2) \subset SO(5)$. An $SU(2)$-structure on $N$ is equivalent to a triple of 2-forms $\left( \omega_1 , \omega_2 , \omega_3 \right)$ and a nowhere-vanishing 1-form $\eta$ such that: 
\begin{itemize}
\item[1.] $\omega_i \wedge \omega_j = \delta_{ij} v$, with $v$ fixed 4-form s.t. $v \wedge \eta$ is nowhere-vanishing: i.e. $v$ is a volume form on the distribution $ \mathcal{H} := \ker \eta$.   
\item[2.]$X \lrcorner \omega_1 = Y\lrcorner \omega_2 \Rightarrow \omega_3 (X, Y) \geq 0$, i.e. $\left( \omega_1 , \omega_2 , \omega_3 \right)$ is an oriented basis of $\Exterior^+ \left( \mathcal{H} \right)$ in the splitting $\Exterior^2 \left( \mathcal{H} \right) = \Exterior^+ \left( \mathcal{H} \right) \oplus \Exterior^- \left( \mathcal{H} \right)$, with respect to the induced Riemannian metric on $ \mathcal{H}$, and volume form $v$.
\end{itemize}
We will take the quadruple $\left( \eta, \omega_i \right)$ satisfying the above as defining an $SU(2)$-structure. If we take $t \in I$ as parametrizing some interval $I \subset \mathbb{R}$, then $\left( \eta, \omega_i \right)$ can be used to define an $SU(3)$-structure $\left( \omega , \Omega \right)$ on $N \times I$:
\begin{align} \label{CYdefine}
\omega = dt \wedge \eta + \omega_1& &\Omega = \left( dt +i \eta \right) \wedge \left( \omega_2 + i \omega_3 \right)
   \end{align}
Requiring the $SU(3)$-structure be torsion-free, i.e. that $\left( \omega , \Omega \right)$ be closed on $N \times I$, gives (on $N$): 
\begin{align} \label{hypo}
d \omega_1 = 0& &d( \omega_3 \wedge \eta ) = 0& &d(\omega_2 \wedge \eta ) = 0 
\end{align}
Along with the evolution equations:
\begin{align} \label{dynamic}
d \eta = \partial_t \omega_1& &d\omega_2 = - \partial_t (\omega_3 \wedge \eta )& &d\omega_3 = \partial_t ( \omega_2 \wedge \eta ) 
\end{align}
The right-hand side of the evolution equations vanishes if the $SU(2)$-structure $\left( \eta, \omega_i \right)$ on $N$ is fixed, but if we instead allow it to vary with $t$, then a one-parameter family of $SU(2)$-structures $\left( \eta, \omega_i \right)_t$ satisfying \eqref{dynamic}, with $\left( \eta, \omega_i \right)_{t}$ initially satisfying \eqref{hypo}, will also define a torsion-free $SU(3)$-structure on $N \times I$. Conversely, if $N$ can be embedded as an oriented hyper-surface in a 6-manifold $M$, then any $SU(3)$-structure on $M$ gives rise to a one-parameter family of $SU(2)$-structures on $N$ for some tubular neighbourhood $M^* \cong N \times I$ of $N \subset M$ (see \cite{conti2005}), and requiring that the $SU(3)$-structure be torsion-free gives \eqref{hypo}, \eqref{dynamic}. 

We will refer to a torsion-free $SU(3)$-structure as a \textit{Calabi-Yau structure} on $M$, an $SU(2)$-structure satisfying \eqref{hypo} as a \textit{hypo-structure} on $N$, and equations \eqref{dynamic} as the \textit{hypo-evolution equations}.

Note that $N \times I$ is foliated into parallel hyper-surfaces when equipped with the metric $g$ compatible with this $SU(3)$-structure, i.e. we have $g = dt^2 + g_t$ for some $t$-dependent metric $g_t$ on $N$. Equivalently, there exists a geodesic on $N \times I$ that meets every hyper-surface in the foliation perpendicularly.

Putting aside completeness of the resulting metrics for a moment, an important example of the above procedure is the Riemannian cone $C(N)$ over $N$. As a smooth manifold, this is just $ \mathbb{R}_{>0} \times N$, and if we identify $N \subset C(N)$ with the hypersurface $ \lbrace 1 \rbrace \times N$ at $t=1$, a fixed $SU(2)$-structure $\left( \eta^{se}, \omega_i^{se}\right)$ on $N$ defines the following 1-parameter family $\left( \eta, \omega_i \right)_t$ of $SU(2)$-structures: 
\begin{align} \label{sesu2}
&\eta = t \eta^{se}  &\omega_i = t^2 \omega_i^{se} 
\end{align}
As in \eqref{CYdefine}, this family defines the \textit{conical $SU(3)$-structure}  $\left( \omega_C , \Omega_C \right)$ on $\mathbb{R}_{>0} \times N$: 
\begin{align} \label{CYcone}
\omega_C = t dt \wedge \eta^{se} + t^2 \omega^{se}_1& &\Omega_C = t^2 \left( \omega^{se}_2 + i \omega^{se}_3 \right) \wedge \left( dt +i t \eta^{se} \right)
   \end{align}
which is Calabi-Yau iff \eqref{sesu2} satisfies equations \eqref{hypo}, \eqref{dynamic}. In this case, satisfying \eqref{hypo}, \eqref{dynamic} is equivalent to the following structure equations on $N$:
\begin{align} \label{sasaki}
&d \eta^{se} = 2 \omega_1^{se}& &d\omega_2^{se} = -3 \omega_3^{se} \wedge \eta^{se}&  &d\omega_3^{se} = 3 \omega_2^{se} \wedge \eta^{se}    
\end{align}
We refer to an $SU(2)$-structure $\left( \eta^{se}, \omega_i^{se}\right)$ satisfying \eqref{sasaki} as being \textit{Sasaki-Einstein}: one can show that such an $SU(2)$-structure induces a Sasaki-Einstein metric $g^{se}$ on $N$, or other words, the Calabi-Yau metric $g_C$ compatible with  $\left( \omega_C , \Omega_C \right)$ on  $\mathbb{R}_{>0} \times N$ is a metric cone $g_C = dt^2 + t^2 g^{se}$ over $g^{se}$. 
  
Another class of examples for this construction arise when we have a smooth, isometric action by some compact Lie group $K$ on $M$, such that there is a $K$-orbit with co-dimension one. These are the \textit{co-homogeneity one} Riemannian manifolds, and it is not difficult to show that the $K$-orbits foliate (a dense open subset of) $M$ into parallel hypersurfaces, and that the quotient space $M / K$ is one-dimensional, c.f. \cite{alekseevsky1993}. These parallel hyper-surfaces can all be written as the homogeneous space $K/H$, where $H$ denotes the principal (i.e. generic) isotropy subgroup of the $K$-action, and the evolution equations \eqref{dynamic} for some $K$-invariant forms $\left( \eta, \omega_i \right)$ on $K/H$ become a finite-dimensional system of ODEs, which can be explicitly solved in some cases. 

Such a situation arises for the three known distinct examples of complete asymptotically conical co-homogeneity one Calabi-Yau 3-folds in the literature, each of which has the form $M = SU(2)^2 \times_{H'} V$ for some singular isotropy subgroup $H' \subset SU(2)^2$, and $H'$-representation $V$:
\begin{enumerate}
\item $\mathcal{O}(-1)\oplus \mathcal{O}(-1)$ over $\mathbb{CP}^1$, with a metric obtained by Candelas and de la Ossa in \cite{CANDELAS1990246}, also known as the \textit{small resolution of the conifold}. The metric is unique up to rescaling by a constant factor, and as a co-homogeneity one manifold we have the diagram $ \triangle U(1) \subset U(1) \times SU(2) \subset SU(2)^2$, where $\triangle U(1)$ is the diagonal $U(1)$ subgroup. The $U(1) \times SU(2)$ representation is given by the following: viewing $v \in V$ as a quaternion, and $q \in SU(2)$ as a unit quaternion, then $(e^{i\theta},q).v = qve^{-i\theta}$. By applying the outer automorphism exchanging the factors of $SU(2) \subset SU(2)^2$, we can get another co-homogeneity one metric from the small resolution, with singular isotropy group $U(1) \times SU(2) \subset SU(2)^2$, but this metric is distinct only up to equivariant isometries.   
\item $T^* S^3$ over $S^3$, with a metric also considered in \cite{CANDELAS1990246} and found independently by Stenzel in \cite{stenzel1993}. This is also referred to as the \textit{smoothing of the conifold} and again, this metric is unique up to overall scale. The group diagram is $ \triangle U(1) \subset \triangle SU(2) \subset SU(2)^2$, and we have as a $\triangle SU(2)$ representation $V \cong \mathfrak{su}(2)$, i.e. $SU(2)$ acts via the adjoint representation. As a smooth manifold, it is diffeomorphic to $\mathbb{R}^3 \times S^3$, the only rank $3$ vector-bundle over $S^3$ up to diffeomorphism. 
\item $\mathcal{O}(-2,-2)$, the total space of the canonical bundle over $\mathbb{CP}^1 \times \mathbb{CP}^1$, with a metric found by Calabi in \cite{calabi} (unique up to overall scaling), which was later generalised to a one-parameter family of metrics by Pando-Zayas and Tseytlin in \cite{PandoZayas}. This parameter represents the relative volume of each $\mathbb{CP}^1$ as the zero-section of $\mathcal{O}(-2,-2)$, and Calabi's construction considers the case when these two volumes are equal. The group diagram is $K_{2,-2} \subset U(1)^2 \subset SU(2)^2$, where $K_{2,-2}$ is the kernel of the map $U(1)^2 \rightarrow U(1)$ given by $(e^{i\theta_1}, e^{i\theta_2}) \mapsto e^{2i\theta_1 - 2i\theta_2}$,  and as a $U(1)^2$-representation we have $V \cong \mathbb{C}_{2,-2}$, i.e. for complex number $V \ni v$, $(e^{i\theta_1}, e^{i\theta_2}).v = e^{2i(\theta_1 -\theta_2)}v$. Note that there is a (non-unique) isomorphism $K_{2,-2} \cong \triangle U(1)\times \mathbb{Z}_2 \subset U(1)^2$, where we define $\triangle U(1)\times \mathbb{Z}_2 \subset U(1)^2$ as the (internal) direct product of the diagonal subgroup $\triangle U(1)$ and the $\mathbb{Z}_2$-subgroup generated by $(e^{2i \pi}, e^{i \pi})$, by sending $K_{2,-2} \ni (e^{i\theta_1}, e^{i\theta_2}) \mapsto (e^{i\theta_1}, e^{i\theta_1}).(e^{2i \pi}, e^{i (\theta_2 -\theta_1)}) \in  \triangle U(1)\times \mathbb{Z}_2$.  
\end{enumerate}
The asymptotic model for the geometry of these spaces (up to $\mathbb{Z}_2$-cover) is the unique co-homogeneity one Calabi-Yau metric cone over $SU(2)^2 / \triangle U(1)\cong S^2 \times S^3$, referred to as the \textit{conifold} in \cite{CANDELAS1990246}. In the co-homogeneity one setting, there is an obvious diffeomorphism identifying the space of principal orbits with the smooth manifold underlying the conifold, and pulling back any of these asymptotically conical metrics to a metric on the conifold via this diffeomorphism, by \cite{CANDELAS1990246}, \cite{PandoZayas}, we have $|  i^*g - g_C | \rightarrow 0$ as $t\rightarrow \infty$, where $t$ denotes the radial parameter on the cone, $i^*g$ denotes the pulled-back metric, and we take norms with respect to the conical metric $g_C$.

\subsection{Invariant Calabi-Yau structures on the space of principal orbits.}
\label{sectionCY}

In order to have a uniform set-up for the gauge theory in later sections, we will recall the construction of these co-homogeneity one Riemannian metrics on $\mathcal{O}(-1)\oplus \mathcal{O}(-1)$, $T^* S^3$, and $\mathcal{O}(-2,-2)$. They appear as solutions to the hypo-equations \eqref{hypo} and evolution equations \eqref{dynamic} on the space of principal orbits $S^2 \times S^3 = SU(2)^2 / \triangle U(1)$, which extend to the singular orbits in the complete cases. 

Let us begin by fixing an explicit basis $E_1, E_2, E_3$ for $\mathfrak{su}(2)$, given by the matrices: 
\begin{align*}
E_1 := \begin{pmatrix} 
i & 0 \\
0 & -i 
\end{pmatrix} \quad
E_2 := \begin{pmatrix} 
0 & 1 \\
-1 & 0 
\end{pmatrix} \quad
E_3 := \begin{pmatrix} 
0 & i \\
i & 0 
\end{pmatrix}
\end{align*}
so that $\left[ E_i, E_j\right] = 2E_k $ for cyclic permutations of $\left (1 2 3 \right)$, and the action of $U(1)$ on $SU(2)$ is generated by $E_1$. Clearly, we can identify the span of $E_2, E_3$ under the adjoint action of $U(1)$ with $\mathbb{C}_2$, where $\mathbb{C}_n$ denotes $n^{th}$ tensor power of the standard representation of $U(1)$ on $\mathbb{C}$. 
 
We will also fix a basis for the left-invariant vector fields of $SU(2)^2$: 
\begin{align*}
U^1 &:= (E_1, 0) &V^1 &:= (E_2, 0) &W^1 &:= (E_3, 0) \\
U^2 &:= (0, E_1) &V^2 &:= (0, E_2) &W^2 &:= (0, E_3)
\end{align*}
and denote $U^\pm := U^1 \pm U^2$, where $U^+$ generates the diagonal subgroup $\triangle U(1)$. 

Let $\mathfrak{m}$ denote the complement of $\mathfrak{u}(1)$ in $\mathfrak{su}(2) \oplus \mathfrak{su}(2)$, where $\mathfrak{u}(1)$ is the span of $U^+$. We have the ad-invariant splitting as $\triangle U(1)$-representations: 
\begin{align*}
\mathfrak{su}(2) \oplus \mathfrak{su}(2) = \mathfrak{u}(1) \oplus \mathfrak{m} := \langle U^+ \rangle \oplus \langle U^-, V^1, W^1, V^2, W^2 \rangle \cong \mathbb{R} \oplus \left( \mathbb{R} \oplus \mathbb{C}_{2} \oplus \mathbb{C}_{2} \right)
\end{align*}
\begin{remark} \label{remarkz2}
Assume $\mathbb{Z}_2 \subset SU(2)^2$ is a subgroup of the flow generated by the vector field $U^-$. The adjoint action of $\mathbb{Z}_2$ on $\mathfrak{m} $ is trivial, and so the results of this section will also apply to $ SU(2)^2 / \triangle U(1) \times \mathbb{Z}_2$.
\end{remark}  
With this notation, we define the \textit{standard} invariant Sasaki-Einstein structure $\left( \eta^{se}, \omega_i^{se} \right)$ on $SU(2)^2 / \triangle U(1)$ as:
\begin{gather} \label{standardse} 
 \begin{aligned} 
\eta^{se} &:= \tfrac{4}{3} u^- & \omega_1^{se} &:= - \tfrac{2}{3}(v^1 \wedge w^1 - v^2 \wedge w^2 )  \\
\omega_2^{se} &:= \tfrac{2}{3}(v^1 \wedge v^2 + w^1 \wedge w^2 ) & \omega_3^{se} &:=  \tfrac{2}{3} (v^1 \wedge w^2 - w^1 \wedge v^2 ) \\   
 \end{aligned} 
 \end{gather}
It is easy to check that $\left( \eta^{se}, \omega_i^{se}\right)$ satisfies the Sasaki-Einstein structure equations \eqref{sasaki}. The corresponding Calabi-Yau cone  $C \left( SU(2)^2 / \triangle U(1) \right)$ has the $SU(3)$-structure $\left( \omega_C , \Omega_C \right)$, as in \eqref{CYcone}, and we refer to this cone as the \textit{conifold}\footnote{Note that any invariant Sasaki-Einstein structure on $SU(2)^2 / \triangle U(1)$ can be obtained from \eqref{standardse} by rotating the plane spanned by $\left( \omega_2^{se}, \omega_3^{se} \right)$. However, since any of two of these structures induce the same Sasaki-Einstein metric $g^{se}$, we will make this particular choice without loss of generality.}.    

Furthermore, it is not hard to show that the space of invariant two-forms on $SU(2)^2 / \triangle U(1)$ is four-dimensional, and spanned by $\omega_0^{se}, \omega_1^{se},\omega_2^{se},\omega_3^{se}$, where we define: 
\begin{align}
\omega_0^{se} &:= \tfrac{2}{3}(v^1 \wedge w^1 + v^2 \wedge w^2 )
\end{align} 
By using this basis of invariant two-forms and the invariant one-form $\eta^{se}$, we have the following description of the space of hypo-structures:
\begin{prop}[\cite{FoscolonK}] Up to transformations by isometries with respect to the induced metric, any invariant family of hypo-structures $\left(\eta, \omega_1, \omega_2,\omega_3 \right)_t$ on $SU(2)^2 / \triangle U(1)$ can be written:
\begin{gather} \label{hypoGEN}
\begin{aligned}
&\eta = \lambda \eta^{se}& &\omega_1 = u_0 \omega_0^{se} + u_1 \omega_1^{se}& &\omega_2 =  \mu \omega_2^{se}&  &\omega_3 = v_0 \omega_0^{se} + v_3 \omega_3^{se} 
\end{aligned}
\end{gather}
where $\lambda, u_0, u_1, v_0, v_3$ are real-valued functions depending on $t \in \mathbb{R}_{>0}$, with $\mu^2 := - u_0^2 + u_1^2 = - v_0^2 + v_3^2$, and $v_0 u_0 = 0$.
\end{prop} 
Clearly, at least one of $v_0$ or $u_0$ must vanish: if  $v_0$ vanishes, we will refer to this family as a \textit{hypo-structure of type} $\mathcal{I}$, while if $u_0$ vanishes, we will refer to this family as a \textit{hypo-structure of type} $\mathcal{II}$. We will write these two situations explicitly below, along with corresponding hypo-evolution equations \eqref{dynamic}:
\begin{enumerate}
\item Type $\mathcal{I}$: 
\begin{align}  \label{hypoA}
&\eta = \lambda \eta^{se}& &\omega_1 = u_0 \omega_0^{se} + u_1 \omega_1^{se}&
 &\omega_2 =  \mu \omega_2^{se}&  &\omega_3 = \mu \omega_3^{se} &
\end{align}
The corresponding hypo-evolution equations are: 
\begin{align} \label{hypoAevolution}
\partial_t u_0 = 0&   &\partial_t u_1 = 2 \lambda&   &\partial_t (\lambda \mu ) = 3 \mu 
\end{align}
\item Type $\mathcal{II}$:
\begin{align} \label{hypoB}
&\eta = \lambda \eta^{se}& &\omega_1 = \mu \omega_1^{se}& 
 &\omega_2 =  \mu \omega_2^{se}&  &\omega_3 = v_0 \omega_0^{se} + v_3 \omega_3^{se} &
\end{align}
The corresponding hypo-evolution equations are: 
\begin{align} \label{hypoBevolution}
\partial_t \mu = 2 \lambda& &\partial_t (\mu \lambda) = 3 v_3& &\partial_t (\lambda v_3 ) = 3 \mu & &\partial_t (\lambda v_0 ) = 0 
\end{align}
\end{enumerate}
If both $v_0, u_0$ vanish, then $u_1 = v_3 = \mu$, and clearly $\lambda = t, \mu = t^2 $ is a solution to the resulting evolution equations:
\begin{align*}
\partial_t \mu = 2 \lambda& &\partial_t (\mu \lambda) = 3 \mu
\end{align*}
which gives rise to the conical Calabi-Yau structure $\left(\omega_C, \Omega_C \right)$ of the conifold. 
\begin{remark} $\left( \omega, \mathrm{Re} \Omega \right)$ represent cohomology classes of $M^* := \mathbb{R}_{>0} \times SU(2)^2 / \triangle U(1) $, and the conserved quantities $u_0$, $ -\lambda v_0$ appearing in \eqref{hypoAevolution}, \eqref{hypoBevolution} are the coefficients of $\left( \left[ \omega \right], \left[ \Real \Omega  \right] \right) \in H^2 \left( M^* \right) \times H^3 \left( M^* \right) \cong \mathbb{R}^2$ with respect to the basis $\omega_0^{se}$, $\omega_0^{se} \wedge \eta^{se}$.   
\end{remark}
For each of the families, one can write down the corresponding invariant Calabi-Yau metric $g= dt^2 + g_t$ explicitly on the space of principal orbits, cf. \cite[Prop.2.16]{FoscolonK}:    
\begin{enumerate} 
\item Type $\mathcal{I}$: 
\begin{align} \label{hypoAmetric}
g = dt^2 + \lambda^2 (\eta^{se})^2 + \tfrac{2}{3} (u_1 - u_0) \left( (v^1)^2 + (w^1)^2 \right) + \tfrac{2}{3} (u_1 + u_0) \left( (v^2)^2 + (w^2)^2 \right) 
\end{align} 
\item Type $\mathcal{II}$:
\begin{align} \label{hypoBmetric}
g = dt^2 + \lambda^2 (\eta^{se})^2 + \tfrac{4}{3} (v_3 - v_0) \left( (v^-)^2 + (w^-)^2 \right) + \tfrac{4}{3} (v_3 + v_0) \left( (v^+)^2 + (w^+)^2 \right) 
\end{align}
where $v^\pm$ is the 1-form dual to tangent vector $V_\pm = V_1 \pm V_2$, respectively $w^+$ is the dual to $W_\pm = W_1 \pm W_2$.
\end{enumerate}
 
With this description in hand, the problem of finding invariant Calabi-Yau metrics on the space of principal orbits is reduced to finding solutions to the evolution equations \eqref{hypoAevolution} or \eqref{hypoBevolution}. We can write the complete Calabi-Yau metrics as solutions extending to the singular orbits at $t=0$, c.f. \cite[Thm.2.27]{FoscolonK}:
\begin{lemma}[{\cite{PandoZayas}},{\cite{CANDELAS1990246}}] \label{CYstructurelemma} Up to transformations by isometries with respect to the induced metric, the space of $SU(2)^2$-invariant Calabi-Yau structures $\left(\omega, \Omega \right)$ on $M$ can be identified with: 
\begin{enumerate}[\normalfont(i)]
\item For $M = \mathcal{O}(-2,-2)$, the open convex cone $ \lbrace \left(U_0, U_1\right) \in \mathbb{R}^2 \mid U_1 > | U_0 | \geq 0 \rbrace$. 
\item For $M = \mathcal{O}(-1)\oplus \mathcal{O}(-1)$, the ray $ \lbrace \left(U_0, U_1\right) \in \mathbb{R}^2 \mid U_1 = - U_0   < 0 \rbrace$.
\end{enumerate}
These invariant Calabi-Yau structures induce a hypo-structure of type $\mathcal{I}$ on the principal orbits, with $ \left(U_0, U_1\right):= \left(u_0(0), u_1(0)\right)$, and:
\begin{align} \label{CYstructure}
&\mu^2 = u_1^2 - U_0^2 & &\lambda^2 = \frac{u_1^3-3 U_0^2 u_1 + U_1 (3U_0^2 - U_1^2)}{u_1^2 -U_0^2}&
\end{align}
\end{lemma}
There are a few comments to be made about the parameters $\left(U_0, U_1\right)$ appearing in Lemma \ref{CYstructurelemma}: firstly, the point $U_0=U_1=0$ is clearly identified with the conifold $u_1 = \mu = t^2$, $\lambda = t$. Secondly, the interior of the cone $ \lbrace \left(U_0, U_1\right) \in \mathbb{R}^2 \mid U_1 > | U_0 | \geq 0 \rbrace$ can be identified with the K\"{a}hler cone of $\mathcal{O}(-2,-2)$, i.e. the convex cone generated by the K\"{a}hler classes of the two copies of $\mathbb{CP}^1 \subset \mathcal{O}(-2,-2)$, and it is not hard to see that multiplicative rescalings of the cone are equivalent to constant rescalings of the metric. Furthermore, the diffeomorphism arising from exchanging the two copies of $\mathbb{CP}^1$ acts on this cone via reflection $U_0 \rightarrow -U_0$, and the Calabi construction in \cite{calabi} produces exactly the metrics in the subset fixed by this action.    

Calabi-Yau structures on the cone boundary $U_1 = \pm U_0$ (excluding the origin) are not quite the same as those found on the boundary of the K\"{a}hler cone of $\mathcal{O}\left( -2, -2 \right)$ however, which generically have $\mathbb{Z}_2$-quotient singularities. Rather, they are a (smooth) branched double-covering\footnote{these quotient singularities do not appear in our set-up, as we only define $\lambda, \mu, u_0, u_1$ at the identity coset on the principal orbit.}: up to exchanging the factors of $\mathbb{CP}^1$, this boundary gives the Calabi-Yau structure on $\mathcal{O}(-1)\oplus \mathcal{O}(-1)$ over $\mathbb{CP}^1 = SU(2)^2/U(1)\times SU(2)$. In the rest of this article, for ease of notation, we will fix the scaling convention for this metric to be $\left(U_0, U_1\right) = \left( -1, 1 \right)$.
 
Finally, for the Calabi-Yau structure on $T^* S^3$, we give the explicit solutions to \eqref{hypoBevolution} extending to the singular orbit $S^3$: 
\begin{lemma}[\cite{stenzel1993}] Up to scale, and transformations by isometries with respect to the induced metric, there is a unique $SU(2)^2$-invariant Calabi-Yau structure on $T^* S^3$. It induces a hypo-structure of type $\mathcal{II}$ on the principal orbits, with:
\begin{gather} \label{hypoBsol}
\begin{aligned}
&\lambda = \left(\frac{2}{3}\right)^{\frac{1}{3}} \frac{\sinh 3s }{ \left( \sinh 3s \cosh 3s -3 s \right)^{\frac{1}{3}}} & & \mu = \left(\frac{2}{3}\right)^{\frac{2}{3}}\left( \sinh 3s \cosh 3s -3 s \right)^{\frac{1}{3}}& \\
 & v_0 = - \left(\frac{2}{3}\right)^{\frac{2}{3}} \frac{ \left( \sinh 3s \cosh 3s -3 s \right)^{\frac{1}{3}}}{\sinh 3s } & & v_3 = \left(\frac{2}{3}\right)^{\frac{2}{3}} \frac{ \left( \sinh 3s \cosh 3s -3 s \right)^{\frac{1}{3}}}{\tanh 3s }&
\end{aligned}
\end{gather}
for $s \in \left[ 0, \infty \right) $, where $s(t) := \int_0^t \lambda^{-1}(\hat{t}) d\hat{t}$
\end{lemma} 
With this description of the underlying Calabi-Yau geometry out of the way, we now return to describing the gauge theory. 
\section{Calabi-Yau Gauge Theory} \label{sectiongauge}
Consider again the monopole equations \eqref{CYmonopoleI}, for a connection $A$ and $\Phi \in \Omega^0 \left( \mathrm{Ad} P \right)$ on some principal bundle $P$ over a Calabi-Yau $3$-fold $\left( M, \omega, \Omega \right)$. As it is more convenient for our purposes, we can rewrite \eqref{CYmonopoleI} as: 
\begin{subequations} \label{CYmonopole}
\begin{align} 
F_{A} \wedge \omega^2 &= 0  \label{CYmonopole1} \\ 
F_{A} \wedge \text{Im}\Omega &= - \frac{1}{2} d_A \Phi \wedge \omega^2 \label{CYmonopole2}
\end{align}
\end{subequations} 
Let us assume we are in the general set-up of \S \ref{sectiondefine}: we let $N \subset M$ be a (real) hyper-surface, and we suppose that $N$ foliates $M$ into parallel hyper-surfaces, up to working on a tubular neighbourhood $N \times I \subseteq M$ for some $I \subseteq \mathbb{R}$. As in \eqref{CYdefine}, we can write the Calabi-Yau structure $\left( \omega, \Omega \right)$ on $M$ in terms of a one-parameter family of hypo-structures $\left( \eta, \omega_i \right)_t$ on $N$. 

In this neighbourhood, we may always write $P \rightarrow M$ as the pull-back of some bundle on $N$, and we can view any $\Phi \in \Omega^0 \left( \mathrm{Ad} P \right)$ as a one-parameter family of sections $\Phi_t$ over $N$. We can also split the connection $A = A_t + \gamma_t dt$, where $A_t$ is a one-parameter family of connections over $N$, and $\gamma_t \in \Omega^0 \left( \text{ad}P \right)$ is a one-parameter family of sections of the adjoint bundle. 

Via a gauge transformation, we can always choose to set $ \gamma_t = 0$: i.e. for each $t \in I$ take $g_t \in G$ such that $\gamma_t + g_t^{-1}(\partial_t g_t)=0$. We will refer to this choice of gauge as the \textit{temporal gauge}, and the curvature of $A= A_t$ in this gauge is given by $F_A = F_{A_t} - \partial_t A_t \wedge dt $, where $\partial_t A_t \in \Omega^1 \left( \text{ad}P \right)$ denotes the limit as $\epsilon \to 0$ of $\frac{1}{\epsilon} \left( A_t - A_{t+ \epsilon} \right)$. Since the space of connections on a given bundle is affine, $\partial_t A_t$ is a genuine one-parameter family of adjoint-valued one-forms on $N$. 
 
With this said, using \eqref{CYdefine} and the temporal gauge on a tubular neighbourhood of $N$, \eqref{CYmonopole} takes the form: 
\begin{subequations} \label{mono}
\begin{align} 
F_{A_t} \wedge \omega_2 \wedge \eta + \frac{1}{2} d_{A_t} \Phi \wedge \omega_1^2 &=0  \label{mono1}\\
F_{A_t} \wedge \omega_1 \wedge \eta + \frac{1}{2} \partial_t A_t \wedge \omega_1^2 &= 0 \label{mono2} \\ 
F_{A_t} \wedge \omega_3 + \partial_t A_t \wedge \omega_2 \wedge \eta &= d_{A_t} \Phi \wedge \omega_1 \wedge \eta - \frac{1}{2} \partial_t \Phi \omega_1^2 \label{mono3}
\end{align}
\end{subequations}
We refer to the equation \eqref{mono1} as the \textit{static} monopole equation, and \eqref{mono2}, \eqref{mono3}, as the monopole \textit{evolution} equations, where \eqref{mono2} is just the condition \eqref{CYmonopole1}, and the other two arise from \eqref{CYmonopole2}. Furthermore, it is not difficult to compute that the static equation \eqref{mono1} is preserved by the evolution equations. Similarly, in the case $\Phi = 0$, we will refer to the respective equations as the static and dynamic instanton equations.
\begin{remark} A solution of \eqref{mono} with $\Phi=0$ is equivalent to a solution $A_t$ of the $t$-dependent flow:
\begin{align*}
\ast \left( F_{A_t} \wedge \omega_1 \right) = - \partial_t A_t 
\end{align*}
with initial conditions $A_{t=t_0}$ satisfying \eqref{mono1}. By fixing a choice of reference connection $A_0$, this flow can be written as the gradient flow for a Chern-Simons functional $CS_{\omega_1}: \mathcal{A} \times I \rightarrow \mathbb{R} $:  
\begin{align*}
CS_{\omega_1} (A_0 + a, t) := \frac{1}{2} \int_N \mathrm{tr} \left( a \wedge \left( 2 F_{A_0} + d_{A_0} a + \frac{2}{3} a \wedge a \right) \right) \wedge \omega_1(t) 
\end{align*}
\end{remark} 
\subsection{Invariant Monopole and Instanton ODEs} \label{invmonopole}

Away from the singular orbit, the general set-up of \eqref{mono} clearly applies to the co-homogeneity one metrics on $\mathcal{O}\left( -2, -2 \right)$, $T^* S^3$, and $\mathcal{O}\left( -1\right) \oplus \mathcal{O}\left( -1\right)$. We will also suppose that the bundle, connection, and Higgs field are invariant under the $SU(2)^2$-action, so that \eqref{mono} is a system of ODEs for the invariant connection and Higgs field on $SU(2)^2/H$, where the relevant principal isotropy subgroup $H$ is given by $H = K_{2,-2}$ or $H = \triangle U(1)$.     

Recall from \cite{wang1958} that we can write such invariant bundles as $SU(2)^2 \times_H G \ra SU(2)^2/H$ for some compact gauge group $G$ and homomorphism $ \lambda: H \rightarrow G$. These bundles are referred to as $SU(2)^2$\textit{-homogeneous}. Recall also that an invariant connection on this bundle can be written as an $H$-equivariant linear map $ A : \mathfrak{su}(2) \oplus \mathfrak{su}(2) \rightarrow \mathfrak{g}$, such that $\left. A \right|_\mathfrak{h} = d\lambda$. Here, $\mathfrak{g}$, $\mathfrak{h} \cong \mathfrak{u}(1)$ denotes the lie algebra of $G$, $H \subset SU(2)^2$, and $d\lambda$ is the image of the canonical connection on $SU(2)^2 \ra SU(2)^2/H$ under $\lambda$. 
 
If $H = K_{2,-2}$, $G=SU(2)$ then the defining homomorphism $K_{2,-2} \rightarrow SU(2)$ of the homogeneous bundle is classified by a pair $\left(n,j\right)\in \mathbb{Z} \times \mathbb{Z}_2$. Using the isomorphism $K_{2,-2} \cong \triangle U(1) \times \mathbb{Z}_2 \subset U(1)^2$, we can write these as: 
\begin{align} \label{homomorphism}
(e^{i\theta}, e^{i\theta}).(e^{2i\pi}, e^{i\pi}) \mapsto \begin{pmatrix} 
-1 & 0 \\
0 & -1
\end{pmatrix}^j \begin{pmatrix} 
e^{in\theta} & 0 \\
0 & e^{-in\theta} 
\end{pmatrix}  
\end{align}
for some $\left(n,j\right)\in \mathbb{Z} \times \mathbb{Z}_2$, and similarly (with $j=0$) for every homomorphism $\triangle U(1) \rightarrow SU(2)$. We will denote the corresponding homogeneous $SU(2)$-bundles over $SU(2)^2 / H$ as $P_{n,j}$, $P_n$ respectively, although since the action of $\mathbb{Z}_2$ in \eqref{homomorphism} is trivial on the Lie algebra of the gauge group $SU(2)$, for the following section, it will suffice just to consider $P_n$. 
 
The canonical connection on $P_n$ appears as $n E_1 \otimes u^+$, and the space of invariant connections can be identified as an affine space for intertwiners of $\triangle U(1)$-representations given by left-invariant one-forms on $SU(2)^2 / \triangle U(1)$ and the composition of \eqref{homomorphism} with the adjoint action on $\mathfrak{su}(2)$. We summarise the results in the following proposition, and compute curvatures:        
\begin{prop} \label{principalconn} $SU(2)^2$-invariant connections $A$ on $P_n$, and corresponding curvatures $F_A$, are of the following form:
\begin{enumerate}[\normalfont(i)]
\item If $n=0$, then for some $a_1 , a_2, a_3 \in \mathbb{R}$,
 \begin{align} \label{principal0}
A = a_1 E_1 \otimes u^- + a_2 E_2 \otimes u^- + a_3 E_3 \otimes u^-& &F_A = \tfrac{3}{2} (a_1 E_1 + a_2 E_2 + a_3 E_3) \otimes \omega_1^{se}
\end{align} 
\item Otherwise, for some $a_0, a_1, a_2, b_1, b_2 \in \mathbb{R}$, where $a_1 = a_2 = b_1 = b_2 = 0$ if $n \neq 1$:  
 \begin{equation}
 \begin{split}  
A &= a_1 ( E_2 \otimes v^1 + E_3 \otimes w^1 )+ b_1 ( E_3 \otimes v^1 - E_2 \otimes w^1 ) \\ &+ a_2 ( E_2 \otimes v^2 + E_3 \otimes w^2 )+ b_2 ( E_3 \otimes v^2 - E_2 \otimes w^2 ) + a_0 E_1 \otimes u^- + n E_1 \otimes u^+ \\
& \\
 F_A &=   3 ( a_1 a_2 + b_1 b_2 ) E_1  \otimes \omega^{se}_3 + 3( a_1 b_2 - b_1 a_2 )E_1  \otimes \omega^{se}_2  \\ 
 &+ \tfrac{3}{2} \left( a_1^2 + b_1^2 + a_2^2 + b_2^2  - n \right)  E_1 \otimes \omega_0^{se} +  \tfrac{3}{2} \left( - a_1^2 - b_1^2  + a_2^2 + b_2^2 + a_0 \right) E_1 \otimes \omega_1^{se} \\
 &+ \tfrac{3}{2} (a_0-1) \left( a_1 \left( E_2  \otimes w^1  -  E_3  \otimes v^1 \right) +  b_1 \left(  E_2  \otimes v^1  +  E_3  \otimes w^1 \right) \right) \wedge  \eta^{se}  \\
 &+ \tfrac{3}{2} (a_0+1) \left( a_2 \left( E_2  \otimes w^2  -  E_3  \otimes v^2 \right) +  b_2 \left(  E_2  \otimes v^2  +  E_3  \otimes w^2 \right) \right) \wedge  \eta^{se} \\
 \end{split}\label{principal1}
\end{equation}
\end{enumerate}
\end{prop}
\begin{proof} As mentioned previously, the canonical connection appears as $ n E_1 \otimes u^+$, the derivative of the map given by \eqref{homomorphism}. As $\triangle U(1)$-representations, we have the following splitting of $\mathfrak{su}(2)$, the Lie algebra of the gauge group $SU(2)$, and $\mathfrak{m}$, the space of left-invariant 1-forms on $SU(2)^2/ \triangle U(1)$:
\begin{align*}
\mathfrak{su}(2) = \langle E_1 \rangle \oplus \langle E_2, E_3 \rangle \cong \mathbb{R} \oplus \mathbb{C}_{2n}& &\mathfrak{m}= \langle u^- \rangle \oplus \langle v^1, w^1 \rangle \oplus \langle v^2, w^2 \rangle  \cong \mathbb{R} \oplus \mathbb{C}_{2} \oplus \mathbb{C}_{2}  
 \end{align*} 
For any invariant connection $A$, by \cite[Thm.A]{wang1958}, we have that $\left. A \right|_\mathfrak{m} $ is an element in the vector space of $\triangle U(1)$-intertwiners $\mathfrak{m} \rightarrow \mathfrak{su}(2)$. If $n \neq 0,1$, this is space is spanned by $E_1 \otimes u^-$, while if $n=0$ it is spanned by $E_1 \otimes u^-, E_2 \otimes u^-, E_3 \otimes u^-$. For $n=1$, define: 
\begin{gather} \label{tensors}
\begin{aligned} 
I_1 := E_2 \otimes v^1 + E_3 \otimes w^1& &J_1:= E_3 \otimes v^1 - E_2 \otimes w^1 \\
I_2:=E_2 \otimes v^2 + E_3 \otimes w^2& &J_2 := E_3 \otimes v^2 - E_2 \otimes w^2   
\end{aligned}
\end{gather}
where $I_i, J_i$ respectively correspond to the identity map and multiplication by imaginary number $i$ between $\triangle U(1)$-representations $\mathbb{C}_2 \rightarrow \mathbb{C}_2$. Clearly, the space of $\triangle U(1)$-intertwiners is spanned by $E_1 \otimes u^-, I_1, J_1, I_2, J_2$, and curvatures can then be computed from the Maurer-Cartan formula $F_A = dA + \frac{1}{2} \left[ A, A \right]$. Explicitly, we compute the following derivatives: 
\begin{align*} 
d I_1 = 2 J_1 \wedge \left( u^+ + u^- \right) & & d J_1 = -2 I_1 \wedge \left( u^+ + u^- \right)& &d (E_1 \otimes u^+) =  - E_1 \otimes (v^1 \wedge w^1 + v^2 \wedge w^2 )  \\
d I_2 = 2 J_2 \wedge \left( u^+ - u^- \right)& &d J_2 = -2 I_2 \wedge \left( u^+ - u^- \right)& &d (E_1 \otimes u^-) =  - E_1 \otimes (v^1 \wedge w^1 - v^2 \wedge w^2 ) 
\end{align*}
and the following commutators: 
\begin{align*} 
\left[ I_i, I_i \right] = 2 E_1 \otimes v^i \wedge w^i & & \left[ J_i, J_i \right] = 2 E_1 \otimes v^i \wedge w^i & &\left[ I_i, J_i \right] = 0
\end{align*}
for $i = 1,2$. The mixed terms are given by: 
\begin{align*} 
\left[ I_1, I_2 \right] = \left[ J_1, J_2 \right] =  2 E_1 \otimes \left( v^1 \wedge w^2 - w^1 \wedge v^2 \right)& &\left[ J_1, I_2 \right] = -\left[ I_1, J_2 \right] = 2 E_1 \otimes \left( v^1 \wedge v^2 + w^1 \wedge w^2 \right)
\end{align*}
and finally:  
\begin{align*} 
\left[ I_i,  E_1 \otimes u^\pm  \right]= - 2 J_i \wedge u^\pm & &\left[ J_i,  E_1 \otimes u^\pm  \right]= 2 I_i \wedge u^\pm
\end{align*}
For an invariant connection on $P_0$, the expression for the curvatures follows immediately. Otherwise, the connection is of the form $A= a_1 I_1 + b_1 J_1 + a_2 I_2 + b_2 J_2 + n E_1 \otimes u^+ + a_0 E_1 \otimes u^-$, for some $a_0, a_1, a_2, b_1, b_2 \in \mathbb{R}$ where $a_1 = a_2 = b_1 = b_2 = 0$ when $n \neq 1$, so applying the above in the Maurer-Cartan formula and comparing with the expression for the standard Sasaki-Einstein structure \eqref{standardse} on $SU(2)^2 / \triangle U(1)$ gives the result.  
\end{proof}
In a similar way, we can classify $SU(2)^2$-invariant sections of the adjoint bundle: an $SU(2)^2$-invariant section of $\mathrm{Ad}P_n$ appears as an element of the Lie algebra $\mathfrak{su}(2)$ invariant under the $\triangle U(1)$ action. This understood, the following proposition is immediate:
\begin{prop} \label{principalhiggs} $SU(2)^2$-invariant sections of $\mathrm{Ad}P_n$ are of the form $\Phi = \phi_1 E_1 + \phi_2 E_2 + \phi_3 E_3$ for some $\phi_1, \phi_2, \phi_3 \in \mathbb{R}$, where $\phi_2 = \phi_3 = 0 $ if $n \neq 0$.
\end{prop}
There are some useful facts about the bundle data of Propositions \ref{principalconn} and \ref{principalhiggs} that should be noted before continuing: firstly, it is clear that the bundles $P_n$ admit only reducible invariant connections when $n \neq 1$, and when $n=1$, from the explicit expressions for curvature, we see that the invariant connection is reducible iff either $a_0 = 1, a_2 = b_2 =0$, or $a_0 = -1, a_1 = b_1 =0$, or $a_1=b_1=a_2=b_2 = 0$.  

Secondly, on $P_n \rightarrow SU(2)^2 / \triangle U(1)$, there is an invariant gauge transformation generated by the vector field $E_1$ on the fibre, which acts by rotation on the plane spanned by $E_2, E_3$, and leaves $E_1$ fixed. In the notation of Propositions \ref{principalconn}, \ref{principalhiggs} for $n \neq 0$, this acts as a rotation $\left( a_1 + i b_1,  a_2 + i b_2 \right) \mapsto \left( e^{i \theta}(a_1 + i b_1), e^{i \theta} ( a_2 + i b_2 ) \right)$ by some common angle $\theta$, and acts trivially on $\left( a_0, \phi_1 \right)$.
 
Using Propositions \ref{principalconn}, \ref{principalhiggs}, we can now write down \eqref{mono} on $\mathcal{O}\left( -2, -2 \right)$, $T^* S^3$, and $\mathcal{O}\left( -1\right) \oplus \mathcal{O}\left( -1\right)$ as ODE systems for the coefficients appearing in these two propositions. Since we are primarily  interested in finding non-abelian solutions to \eqref{mono}, it will suffice to consider the case for $n=1$\footnote{See \S \ref{abeliansection} for explicit abelian solutions in the case $n=1$: the solutions for $n\neq1$ are similar.}:  
\begin{prop} \label{propgen} On $P_1 \rightarrow \mathbb{R}_{>0} \times SU(2)^2 / \triangle U(1)$ with Calabi-Yau structure \eqref{hypoGEN}, invariant monopoles $\left(A, \Phi \right)$ can be written, up to gauge, as: 
\begin{align*}
A = a_1 ( E_2 \otimes v^1 + E_3 \otimes w^1 ) + a_2 ( E_2 \otimes v^2 + E_3 \otimes w^2 ) + a_0 E_1 \otimes u^- + E_1 \otimes u^+& &\Phi = \phi E_1
\end{align*}
with $\left( a_0, a_1, a_2, \phi \right)$ real-valued functions satisfying the following ODE system: 
\begin{gather}  \label{dynamicODE}
 \begin{aligned} 
\dot{a_0} &= \frac{4 \lambda}{ \mu^2} \left( ( a_1^2 +  a_2^2 - 1 )u_0  - ( a_0 - a_1^2 +  a_2^2 ) u_1 \right)  \\
\dot{a_1} &= \frac{3}{2 \lambda \mu }\left(  (a_0 -1 ) a_1 v_3 -  (a_0+1) a_2 v_0 \right)
 - 2 \frac{u_1-u_0}{\mu} a_2 \phi  \\
 \dot{a}_2 &= \frac{3}{2 \lambda \mu} \left( (a_0-1) a_1 v_0- (a_0+1) a_2 v_3 \right) - 2 \frac{u_1 + u_0}{ \mu } a_1 \phi \\
 \dot{\phi} &= \frac{3}{\mu^2} \left( \left(a_1^2+a_2^2 -1\right) v_0 - 2 a_1 a_2 v_3 \right)  \\
\end{aligned} 
\end{gather}
\end{prop} 
\begin{proof}
We use Propositions \ref{principalconn} and \ref{principalhiggs} in the monopole equations \eqref{mono}: we use the temporal gauge to put the connection into the form $A_t = a_1 I_1 + b_1 J_1 + a_2 I_2 + b_2 J_2 + a_0 E_1 \otimes u^- + n E_1 \otimes u^+ $, where $I_1, I_2, J_1, J_2$ are as in \eqref{tensors}. Then $d_{A_t} \Phi = \left[ A, \Phi \right] = \phi \left[ A, E_1 \right] = 2 \phi \left( - a_1 J_1 + b_1 I_1 - a_2 J_2 + b_2 I_2 \right)$. This implies $d_{A_t} \Phi \wedge \omega_1^2$ vanishes, so the static equation \eqref{mono1} is just the single condition $a_1 b_2 - b_1 a_2 =0$. Equation \eqref{mono2} also only has a single component, giving: 
\begin{align*}
\dot{a_0} &= \frac{4 \lambda}{ \mu^2} \left( ( a_1^2 + b_1^2 +  a_2^2 +  b_2^2 - 1 )u_0  - ( a_0 - a_1^2  - b_1^2 +  a_2^2 +  b_2^2 ) u_1 \right) 
\end{align*}
Splitting \eqref{mono3} into $E_1, E_2, E_3$ components, the $E_1$ component gives: 
\begin{align*}
\dot{\phi} &= \frac{3}{\mu^2} \left( \left(a_1^2+b_1^2+a_2^2 +b_2^2 -1\right) v_0 - 2 \left( a_1 a_2 + b_1 b_2 \right) v_3 \right) 
\end{align*}
Meanwhile the $E_2, E_3$ components together give: 
\begin{align*}
\dot{a_1}&= \frac{3}{2 \lambda \mu }\left(  (a_0 -1 ) a_1 v_3 -  (a_0+1) a_2 v_0 \right)
 - 2 \frac{u_1-u_0}{\mu} a_2 \phi \\
\dot{b_1}&= \frac{3}{2 \lambda \mu }\left(  (a_0 -1 ) b_1 v_3 -  (a_0+1) b_2 v_0 \right) 
 - 2 \frac{u_1-u_0}{\mu} b_2 \phi  \\
 \dot{a}_2&= \frac{3}{2 \lambda \mu} \left( (a_0-1) a_1 v_0- (a_0+1) a_2 v_3 \right) - 2 \frac{u_1 + u_0}{ \mu } a_1 \phi \\ \dot{b}_2 &= \frac{3}{2 \lambda \mu} \left( (a_0-1) b_1 v_0- (a_0+1) b_2 v_3 \right) - 2 \frac{u_1 + u_0}{ \mu } b_1 \phi   
\end{align*}
We can now use the invariant gauge-transformation generated by $E_1$ to simplify this ODE system, which appears as the symmetry of the equations. Using the static condition $a_1 b_2 - a_2 b_1=0$, we will use this symmetry to set $b_1= b_2=0$, thus giving the ODEs in the form stated. 
\end{proof}
We note here that \eqref{dynamicODE} displays some further discrete symmetries: 
\begin{prop} The following involution is a discrete symmetry of \eqref{dynamicODE}:
\begin{align} \label{symmetrygauge}
\left( a_0, a_1, a_2, \phi \right) \mapsto \left( a_0, -a_1, -a_2, \phi \right)
\end{align}
Specialising to the case of \eqref{dynamicODE} with $u_0 = 0$, we have an additional symmetry: 
\begin{align} \label{symmetrymetric}
\left( a_0, a_1, a_2, \phi \right) \mapsto \left( -a_0, a_2, a_1, \phi \right)
\end{align}
\end{prop}
\begin{remark} If one is also free to vary the Calabi-Yau structure, \eqref{symmetrymetric} becomes a symmetry of the full system \eqref{dynamicODE} with $u_0 \mapsto - u_0$.
\end{remark}
\begin{proof}
One can easily check that the symmetries of this proposition are indeed symmetries of the ODE systems in question. We comment instead on the origin of such symmetries: \eqref{symmetrygauge} is a residual symmetry from the invariant gauge transformation that we used to set $b_1=b_2=0$: it is simply the rotation by angle $\pi$ of the plane spanned by $E_2, E_3$. Meanwhile, \eqref{symmetrymetric} is the symmetry arising from interchanging the two factors of $SU(2)$ on the principal orbits: this explains why one must alter the Calabi-Yau structure to see it as a symmetry of \eqref{dynamicODE}. 
\end{proof}

 We also recall that a natural condition on solutions to the monopole equations on asymptotically conical CY 3-folds is to require quadratically decaying curvature. In terms our ODE system \eqref{dynamicODE}, this requirement takes the following form:

\begin{lemma} \label{quadraticdecay} An invariant solution $\left( A, \Phi \right)$ to the monopole equations determined by a solution $\left( a_0, a_1, a_2, \phi \right)$ to \eqref{dynamicODE} has  quadratically decaying curvature if and only if $a_0, a_1, a_2, t a_1 \phi, t a_2 \phi$ are bounded.    
\end{lemma} 
\begin{proof}
Using the expression for curvature $F_{A} = F_{A_t} - \partial_t A_t \wedge dt$ in the temporal gauge, the explicit expressions for $F_{A_t}$, $A_t$, given in \eqref{principal0}, and the scaling of $k$-forms on the cone, it is clear that $t^2 |F_A |$ is bounded if $a_0, a_1, a_2, t\dot{a}_0, t\dot{a}_1, t\dot{a}_2$ are bounded. The converse is clear for $t\dot{a}_0, t\dot{a}_1, t\dot{a}_2$, and note that $t^2 |F_{A_t}|$ is bounded only if $a_1^2 +  a_2^2 - 1$, and $-a_1^2 + a_2^2 + a_0$ are: the first of these implies $a_1, a_2$ must be bounded, and since $|-a_1^2 + a_2^2 + a_0|\geq |a_0| - |a_1^2 - a_2^2|$ this implies $a_0$ must be bounded also.  

Up to terms decaying faster than $O(t^{-1})$, as $t\rightarrow \infty$, the ODE system  \eqref{dynamicODE} is asymptotic to \eqref{dynamicODE} on the conifold: 
\begin{gather} \label{instCone}
\begin{aligned} 
\dot{a}_0 &= - \frac{4}{t} \left(  a_0 - a_1^2 + a_{2}^2 \right) \\
\dot{a}_1 &=  \frac{3}{2t} \left(  a_0 - 1  \right) a_1 - 2 a_2 \phi \\
\dot{a}_2 &= - \frac{3}{2t} \left(  a_0 + 1  \right)  a_2  - 2 a_1 \phi \\
\dot{\phi} &=  - \frac{6}{t^2} a_1 a_2
\end{aligned} 
\end{gather} 
and comparing the expressions for $t\dot{a}_0, t\dot{a}_1, t\dot{a}_2$ gives the statement of the lemma. 
\end{proof}
\begin{remark} One can show that any solution $\left( a_0, a_1, a_2, \phi \right)$ to \eqref{dynamicODE} converging in $C^0$ as $t \rightarrow \infty$ must converge to a $t$-invariant solution of \eqref{instCone}: $\left( 1,1,0,0 \right), \left( -1,0,1,0 \right)$, or $\left( 0,0,0,m \right)$, $m \in \R$, up to gauge i.e. $A_1^\flat$, $A_2^\flat$, or $A^\mathrm{can}$ with parallel Higgs field $\Phi_m = m E_1$. In the following sections, we will also see that any bounded solution to \eqref{dynamicODE} extending over the singular orbit at $t=0$ must converge to one of these solutions as $t \rightarrow \infty$.
\end{remark} 

\subsection{Reducible Solutions} \label{abeliansection}
Before conducting an analysis of the full system \eqref{dynamicODE}, we will briefly say something about the reducible case, i.e. if we consider abelian or flat connections. Firstly, note that the trivial flat connection $A^\flat$ on $P_1 \rightarrow SU(2)^2 / \triangle U(1) \cong S^2 \times S^3$ appears in two distinct $SU(2)^2$-invariant gauge-equivalence classes\footnote{although these represent the same connection up to non-equivariant gauge, at least on $S^2 \times S^3$.}, which can be represented by:
\begin{align}
A_1^\flat:= E_1 \otimes u^1 + E_2 \otimes v^1 + E_3 \otimes w^1 & & A_2^\flat:= E_1 \otimes u^2 + E_2 \otimes v^2 + E_3 \otimes w^2
\end{align}
i.e. in terms of Proposition \ref{principalconn}, we have $a_0 = 1, a_1 =1,  b_1 = a_2 = b_2=0$, or $a_0 = -1, a_2 =1,  b_1 = a_1 = b_2=0$ respectively. These are clearly just lifts of the standard Maurer-Cartan form on $SU(2)$ to $P_1$, and $A_1^\flat, A_2^\flat$ are exchanged via non-equivariant diffeomorphism obtained via exchanging the factors of $SU(2)$ in $SU(2)^2 / \triangle U(1)$. 

Secondly, note that if both $a_1=a_2=0$ then the connection is abelian, and we can solve \eqref{dynamicODE} explicitly on the space of principal orbits: 

\begin{align}  \label{abelian0}
 a_0(t) = \frac{C - 2 u_0 u_1}{\mu^2}& &\phi = -3 I(t) 
\end{align}
where $\dot{I}(t) =  \tfrac{v_0}{\mu^2}$, and $C$ is a constant of integration. 

Using the results of Appendix \ref{sectioncohobundles}, we see that generic solution \eqref{abelian0} can extend over the singular orbits $S^2$, or $S^2 \times S^2$ only if\footnote{the converse will also hold for a suitable choice of bundle extension.} $C=2 u_0 u_1(0)$, and can never extend over the singular orbit $S^3$. 

\begin{remark} \label{remarkabelian} For later reference, we note that the generic abelian solution  \eqref{abelian0} on $\mathcal{O}(-1) \oplus \mathcal{O}(-1) \setminus \mathbb{CP}^1$ is also unbounded near $\mathbb{CP}^1$ unless $C=2 u_0 u_1(0)$. 
\end{remark}
\subsection{Local Solutions} \label{sectionlocalsolutions}
We now consider the full system \eqref{dynamicODE}. Unlike with the reducible case, in general this will not have explicit solutions, and instead, we will analyse the qualitative behaviour of solutions as they move away from the singular orbit. To determine their behaviour near the singular orbit, we will apply the theory of singular initial value problems of the form \cite[Thm.4.7]{FoscolonK}:  $t \dot{y} = M_{-1}(y) + M(t,y)$, where $M(t,y) t^{-1}$, $M_{-1}(y)$ are smooth functions of their arguments. To have unique solution near $t=0$, we require that $M_{-1}(y_0) = 0$ at initial value $y(0)=y_0$, and that the linearisation $d_{y_0}M_{-1}$ has no positive integer eigenvalues. This theory, combined with the boundary conditions found in Appendix \ref{sectioncohobundles}, will allow us to construct local solutions to \eqref{dynamicODE} extending over the singular orbits at $t=0$.

In all cases, we will find that solutions to the monopole equations are in a local two-parameter family for each bundle extending $P_1$ over the singular orbit, with the vanishing of the second parameter corresponding to the vanishing of the Higgs field $\phi$, and thus a local one-parameter family of instantons. 
 
First of all, here is a countable family of bundles $P_{1-l,l}$, $l \in \mathbb{Z}$ extending $P_1$ over the singular orbit $S^2 \times S^2$. However, we can reduce our computations to the case $l>0$ by the diffeomorphism exchanging the factors of $SU(2)$ in the $SU(2)^2$-orbits on the total space of the bundle, since this map sends $P_{1-l,l} \mapsto P_{l,1-l}$. As this map acts on the underlying Calabi-Yau structure by sending the constant $u_0 \mapsto -u_0$, the monopole ODEs \eqref{dynamicODE} also transform, but solutions of the transformed system are equivalent to solutions of \eqref{dynamicODE} under the symmetry \eqref{symmetrymetric}: 
 \begin{prop} \label{localresol4} In a neighbourhood of the singular orbit of $P_{1-l,l} \rightarrow\mathcal{O}(-2,-2)$ local solutions to \eqref{dynamicODE} are in a two-parameter family $\left(Q^l, \Theta^l \right)_{\alpha_l, \beta_l}:= \left( a_0, a_1, a_2, \phi \right)_{\alpha_l, \beta_l}$ for each $l \in \mathbb{Z}$. For $l>0$, these solutions satisfy:
 \begin{align*} 
&a_0 = 1-2l + O(t^2)&  &a_1 = - \frac{1}{l} \beta_l \alpha_l \sqrt{\frac{U_1 - U_0}{U_1 +U_0}} t^l  + O(t^{l+2})& &a_2 = \alpha_l t^{l-1} + O(t^{l+1})& &\phi =  \beta_l + O(t^2)&
 \end{align*} 
\end{prop}
\begin{proof} We write \eqref{dynamicODE} for a Calabi-Yau structure of type $\mathcal{I}$: 
\begin{gather} \label{ivp0} 
 \begin{aligned} 
\dot{a_0} &= \frac{4 \lambda}{ \mu^2} \left( ( a_1^2 +  a_2^2 - 1 )u_0  - ( a_0 - a_1^2 +  a_2^2 ) u_1 \right)  \\
\dot{\phi} &= -\frac{6}{\mu} a_1 a_2   \\
\dot{a}_1 &= \frac{3}{2 \lambda  }  (a_0 -1 ) a_1 - 2 \frac{u_1-u_0}{\mu} a_2 \phi  \\
 \dot{a}_2 &= -\frac{3}{2 \lambda } (a_0+1) a_2 - 2 \frac{u_1 + u_0}{ \mu } a_1 \phi 
\end{aligned} 
\end{gather}
We consider solutions to \eqref{ivp0} with this Calabi-Yau structure given by \eqref{CYstructure}, for any $U_1, U_0$ with $U_1 > | U_0 | \geq 0$. The power-series of $\lambda$, $u_1$, $\mu$ near $t=0$ are given by:
\begin{align*}
\lambda(t)= 3 t + O(t^3)& &u_1 = U_1 + O(t^2)& & \mu = \sqrt{U_1^2 - U_0^2} + O(t^2) 
\end{align*}
Although we cannot apply \cite[Thm.4.7]{FoscolonK} directly, we can use the boundary conditions for extending to the singular orbit to re-write this system in the correct form:
 
First, assume $l>0$. Using Proposition \ref{propconnboundary4}, we can define smooth functions $X_1, X_2$ such that $a_1 = t^l X_1$, $a_2 = t^{l-1} X_2$, and \eqref{ivp0} becomes: 
\begin{align*} 
\dot{a_0} &= O(t) \\
\dot{\phi} &= O(t^{2l-1})   \\
\dot{X}_1 &= \frac{1}{t}  \left( \frac{1}{2} \left( a_0 -1 - 2l \right)X_1 - 2 X_2 \phi  \sqrt{\frac{U_1-U_0}{U_1+U_0}} \right) + O(t)   \\
 \dot{X}_2 &= - \frac{1}{2t} \left( a_0 -1 + 2l \right)X_2 + O(t)
\end{align*}
Since the extension conditions also require $a_0(0) = 1-2l$, once we fix $\alpha_l := X_2(0)$, $ \delta_l:= \phi(0)$ such that $lX_1 (0) + X_2 (0) \phi(0)\sqrt{\frac{U_1-U_0}{U_1+U_0}} = 0$, then $y(t)=\left( a_0,X_1, X_2,\phi \right)$ satisfies a singular initial value problem with linearisation
\begin{align*}
d_{y_0} M_{-1} =  \begin{pmatrix} 
0 & 0 & 0 & 0 \\
0 & 0 & 0 &  0 \\
\frac{1}{2} X_1(0) & -2\sqrt{\frac{U_1-U_0}{U_1+U_0}} \alpha_l & -2l & -2\sqrt{\frac{U_1-U_0}{U_1+U_0}} \delta_l  \\
 -\frac{1}{2} \alpha_l & 0 & 0 & 0 
\end{pmatrix}
\end{align*}
at initial value $y_0 = \left( 1-2l, - \tfrac{1}{l} \alpha_l \beta_l \sqrt{\tfrac{U_1-U_0}{U_1+U_0}}, \alpha_l, \beta_l \right)$. This has a unique solution once we fix $y_0$, since $\det \left( k \mathrm{Id} - d_{y_0} M_{-1} \right)= \left(k+2l \right) k^3>0$ for $k>0$.

To recover the local solutions extending for $l\leq 0$ from these solutions, we can send $U_0 \mapsto -U_0$ and apply the transformation \eqref{symmetrymetric}. It is easy to verify from Proposition \ref{propconnboundary4} that these solutions extend to $P_{l,1-l}$.   
\end{proof}
\begin{remark} By setting $\beta_l =0$ in $\left(Q^l, \Theta^l \right)_{\alpha_l, \beta_l}$, we obtain a local one-parameter family of instantons i.e. solutions to \eqref{dynamicODE} with $\phi =0$, and for $l>0$ these solutions have: 
\begin{align} \label{localresol5}
a_0 - a^{\mathrm{ab}}_0 = -\frac{6 \alpha_l^2 }{l (U_0 + U_1)}t^{2l} + O(t^{2l+2})
\end{align}
where $a^{\mathrm{ab}}_0$ denotes the abelian solution to \eqref{dynamicODE} extending over the singular orbit of $P_{1-l,l}$. 

Moreover, when $l=1$, these solutions have: 
 \begin{align} \label{localresol3} 
a_0 = - 1 - \frac{6}{U_1+U_0} (\alpha_1^2 -1) t^2 + O(t^4)& &a_2 = \alpha_1 + \frac{3}{2 (U_1+U_0)} \alpha_1 ( \alpha_1^2 -1) t^2 + O(t^4)
\end{align}
\end{remark}
As their proofs are similar, we will state the results for solutions extending over singular orbits $S^2$ and $S^3$ without proof. The correct re-parametrisations, corresponding initial values $y_0$, and  linearisations $d_{y_0} M_{-1}$ as in \cite[Thm.4.7]{FoscolonK} can be found in Appendix \ref{singIVPs}. 
  
The bundle $P_1$ extends uniquely over $S^3$, and we denote this extension $P_\mathrm{Id}$: 
\begin{prop} \label{localmonopolesmoothing} In the neighbourhood of the singular orbit, solutions to \eqref{dynamicODE} on $P_\mathrm{Id} \rightarrow T^* S^3$ are in a two-parameter family $\left( S, \Phi \right)_{\xi,\chi} := \left(a_0,a_+,a_-,\phi \right)_{\xi,\chi}$,  where $a_+ = a_1 + a_2$, $a_- = a_1 - a_2$. These solutions satisfy: 
 \begin{align*}  
&a_0 = \xi + O(t^2)&  &a_+ = 1 + \left( \frac{9}{8} (\xi^2 - 1) - \chi \right) t^2  + O(t^4)& &a_- = \xi + O(t^2)& &\phi =  \chi t + O(t^3)&
 \end{align*}
\end{prop}

\begin{remark} By setting $\chi =0$ in $\left( S, \Phi \right)_{\xi,\chi}$, we obtain a local one-parameter family of instantons i.e. solutions to \eqref{dynamicODE} with $\phi =0$, and we give some additional terms in the resulting power-series: 
\begin{align} \label{localsmoothing} 
a_0 = \xi + \frac{9}{10} \xi(-1 + \xi^2) t^2 + O(t^4)& 
&a_+ = 1 + \frac{9}{8} (- 1 + \xi^2) t^2  + O(t^4)& 
&a_- = \xi + \frac{27}{40} \xi ( - 1 + \xi^2 ) t^2 + O(t^4)
 \end{align}
\end{remark} 
Finally, the bundle $P_1$ extends in exactly two ways over $S^2 = SU(2)^2/ U(1) \times SU(2)$, and we denote these possible extensions $P_{0,\mathrm{Id}}$ and $P_{1,\mathbf{0}}$:
\begin{prop} \label{localresol1}  In the neighbourhood of the singular orbit of $P_{0,\mathrm{Id}}\rightarrow\mathcal{O}(-1) \oplus \mathcal{O}(-1)$, solutions to \eqref{dynamicODE} are in a two-parameter family $\left(R, \Psi \right)_{\epsilon, \delta}:=\left( a_0, a_1, a_2, \phi\right)_{\epsilon, \delta}$, with: 
\begin{align*} 
& a_0 = - 1 + \epsilon t^2 + O (t^4) &  & a_1 =  - \frac{\delta}{\sqrt{3}} t^2 + O (t^4) & & a_2 = 1 - \frac{1}{2} \epsilon t^2 + O (t^4) & & \phi = \delta t^2 + O (t^4) & 
\end{align*}
\end{prop} 
\begin{prop} \label{localresol6}  In a neighbourhood of the singular orbit of $P_{1, \mathbf{0}}\rightarrow \mathcal{O}(-1) \oplus \mathcal{O}(-1)$, solutions to \eqref{dynamicODE} are in a two-parameter family $\left(R', \Psi' \right)_{\epsilon', \delta'}:=\left( a_0, a_1, a_2, \phi\right)_{\epsilon', \delta'} $, with: 
\begin{align*}
& a_0 = 1 + O (t^2) &  & a_1 = \epsilon' + O (t^2) & & a_2 = O (t^2) & & \phi = \delta' + O (t^2) & 
\end{align*}
\end{prop} 
\begin{remark} On $P_{1, \mathbf{0}}$, we have $a_2 = - \frac{\sqrt{3}}{4} \epsilon' \delta' t^2 + O(t^4)$ once we fix $\epsilon'$, $\delta'$. 
\end{remark}

Having computed these two-parameter families of local solutions to the monopole equations \eqref{dynamicODE}, by uniqueness, we see that the following one-parameter families are the local solutions to the instanton equations, i.e. \eqref{dynamicODE} with $\phi =0$: 
\begin{align*}
S_\xi := \left( S, \Phi \right)_{\xi,0} &  & R_\epsilon := \left(R, \Psi \right)_{\epsilon, 0}& & {R'}_{\epsilon'} := \left(R' , \Psi' \right)_{\epsilon', 0}& &Q^l_{\alpha_l} :=\left(Q^l, \Theta^l \right)_{\alpha_l, 0} 
\end{align*}
For later reference, we have already computed some additional terms in the power-series of $S_\xi$, $Q^l_{\alpha_l}$, in \eqref{localsmoothing}, \eqref{localresol5}, \eqref{localresol3}. For the analysis of the family $R'_{\epsilon'}$, it will be more useful to first apply the transformation \eqref{symmetrymetric}, and then compute higher-order terms with respect to \eqref{dynamicODE} with $u_0 \mapsto -u_0$. To explain why, observe that the instanton equations for a hypo-structure of type $\mathcal{I}$, i.e. \eqref{dynamicODE} with $\phi, v_0$ vanishing, has at least one of $a_1$ or $a_2$ vanishing identically, and if both vanish we have the abelian solution. From the boundary conditions of Propositions \ref{propconnboundary}, \ref{propconnboundary3}, which of $a_1$ or $a_2$ must necessarily vanish will depend on how we extend the bundle $P_1$ to the singular orbit: we have $a_1$ vanishing for $P_{0,\mathrm{Id}}$ and $P_{1-l,l}$ for $l>0$, while $a_2$ vanishes for $P_{1,\mathbf{0}}$ and $P_{1-l,l}$ for $l \leq 0$. 

However, we can always reduce our analysis to a single ODE system with, say, $a_1$ vanishing identically by applying \eqref{symmetrymetric} to \eqref{dynamicODE} and mapping $u_0 \mapsto -u_0$: this is the same as pulling back these equations by the diffeomorphism exchanging the factors of $SU(2)$ in the $SU(2)^2$-orbits on the total space of the bundle. This has been previously explained for the solutions $Q^l_{\alpha_l}$, and we can apply the same reasoning to the family $R'_{\epsilon'}$: the caveat here is  that if we exchange the factors on the singular orbit $S^2$, then the bundle $P_1$ and the Calabi-Yau structure on the principal orbits now extends over $S^2 = SU(2)^2/SU(2) \times U(1)$ rather than our convention $S^2 =SU(2)^2/ U(1) \times SU(2)$. 
 
With this explained, let us denote $P_{\mathbf{0},1}$ the bundle obtained from $P_{1,\mathbf{0}}$ by exchanging the factors of $SU(2)^2$, and pull back the local one-parameter family of invariant instantons $R'_{\epsilon'}$ on $P_{1,\mathbf{0}}$ to a local one-parameter family of invariant instantons on $P_{\mathbf{0},1}$. Corollary \ref{propconnboundary4} ensures that these solutions actually extend to $SU(2)^2/SU(2) \times U(1)$, and for later reference, we compute some higher order terms in the power-series: 
\begin{lemma}
In a neighbourhood of the singular orbit, solutions to \eqref{dynamicODE} on $P_{\mathbf{0}, 1}$ with $\phi = 0$ are in a one-parameter-family, pulled back via \eqref{symmetrymetric} from the one-parameter family $R'_{\epsilon'}$. 
 \begin{align} \label{localresol2}
a_0 = - 1 - \frac{3}{4} \left(\epsilon'^2-1\right) t^2 + O(t^4)& &a_2 = \epsilon' + \frac{3}{8} \epsilon' \left(\epsilon'^2-1\right) t^2 + O(t^4) 
 \end{align}
 \end{lemma}

\section{ODE Analysis} \label{odessection}
\subsection{Solutions to the Instanton Equations} \label{odessection1}
Using the description of solutions to \eqref{dynamicODE} near the singular orbit, we will now describe the qualitative behaviour of the solutions as we move away from this orbit. We will focus first on the case $\phi$ vanishes i.e. instantons: in this case, the requirement of quadratic curvature decay in Lemma \ref{quadraticdecay} is equivalent to considering bounded solutions. 
 
We will start with the smoothing $T^* S^3$. A single explicit solution to \eqref{dynamicODE} on the smoothing was found in \cite[Theorem 2]{Oliveira2015}: 
 \begin{align} \label{explicit}
a_0= \phi= 0& &a_1 = a_2 = \frac{1}{2}\sqrt{\frac{4}{3 \lambda (v_3 - v_0)}}
 \end{align}
given locally by the power-series $S_\xi$ in \eqref{localsmoothing} with $\xi = 0$. We now show that this instanton actually lies in a one-parameter family: 
\begin{theorem} \label{stenzelthm} Invariant instantons with quadratic curvature decay on $P_\mathrm{Id} \rightarrow T^* S^3$ are in a one-parameter family $S_\xi$, $-1 \leq \xi \leq 1$, up to gauge. Moreover: 
\begin{enumerate}[\normalfont(i)]
\item The isometry exchanging the factors of $SU(2)$ on the principal orbits of $T^*S^3$ sends $S_\xi \mapsto S_{-\xi}$, with explicit fixed point $S_0$ given by \eqref{explicit}.
\item $S_{1} = A_1^b$, $S_{-1} = A_2^b$, and $S_{\xi}$, $-1< \xi <1$ are irreducible with $\lim_{t \to \infty} S_{\xi} (t) = A^{\mathrm{can}}$.
\end{enumerate}
\end{theorem}
\begin{proof}[Proof of Theorem \ref{stenzelthm}]
We will prove that the local solutions $S_\xi$ given by the power-series \eqref{localsmoothing} near the singular orbit exist for all time if $|\xi | \leq 1 $ and are otherwise unbounded.
 
 First, we formulate \eqref{dynamicODE} with $\phi=0$ in terms of $a_+= a_1 + a_2$, $a_- = a_1 - a_2$:
 \begin{align} \label{instB2}
\dot{a}_0 = f_0 ( a_+ a_- - a_0 )&  &\dot{a}_+ = f_+(  a_0 a_- - a_+ )& &\dot{a}_- = f_-( a_0 a_+ - a_-) 
 \end{align} 
where we define: 
\begin{align*}
&f_0 := \frac{4\lambda}{\mu}& &f_+ := \frac{3 (v_3 + v_0)}{2 \lambda \mu}& &f_- := \frac{3 (v_3 - v_0)}{2 \lambda \mu}&
\end{align*}   
As the functions $f_0$, $f_+$, $f_-$ are all strictly positive on $\left(0,\infty \right)$, the following lemma is immediate: 
\begin{lemma} \label{criticalpoints} Critical points of \eqref{instB2} for $t \in \left(0, \infty \right)$ are given by the following triples $\left(a_0, a_+, a_- \right)$:
\begin{align*}
\left(1,1,1 \right)& &\left(1,-1,-1 \right)& &\left(-1,1,-1 \right)& &\left(-1,-1,1 \right)& &\left(0,0,0 \right)
\end{align*}
\end{lemma}
\begin{proof}
This follows by a simple computation: note that these critical points are just the canonical connection $A^\mathrm{can}$ and the flat connections $A_1^\flat$, $A_2^\flat$ under the symmetries \eqref{symmetrygauge} and \eqref{symmetrymetric}. 
\end{proof}
We will define a subset $\mathcal{S} \subset \mathbb{R}^n$ to be \textit{forward-invariant} for an ODE system $\dot{\mathbf{x}}=\mathbf{F}\left(\mathbf{x},t\right)$ if a solution $\mathbf{x}(t)$ contained in $\mathcal{S}$ at some non-singular initial time $t^*$, must remain in $\mathcal{S}$ for all forward time $t \geq t^*$ for which the solution exists.  
\begin{lemma} The following sets in $\mathbb{R}^3$ are forward-invariant for \eqref{instB2}:
\begin{align*}
\left(0,\infty\right)^3& &\left(0,1\right)^3& &\left(1,\infty\right)^3
\end{align*}
\end{lemma} 
\begin{proof} 
\begin{enumerate}[(i)]
\item We bound a solutions $\left( a_0, a_+, a_- \right) $ lying in the quadrant $\left(0,\infty\right)^3$ with boundary $a_0=0$, $a_+=0$, and $a_- = 0$. We can exclude the axes at intersections of these planes by local uniqueness to ODEs, since \eqref{instB2} has three families of solutions given by setting any two of $\left( a_0, a_+, a_- \right)$ to be identically zero. 

At $a_0 = 0, a_+ \geq 0, a_- \geq 0$, $\dot{a}_0 = f_0 a_+ a_- \geq 0$, with equality iff $a_+ =0$ or $a_- = 0$. Since a solution cannot hit any of the axes, this implies both are zero if $\dot{a}_0 = 0$, but since $\left(0,0,0\right)$ is a critical point, by uniqueness one cannot have this situation either, and hence the inequality is strict. This implies a solution with $a_0 > 0, a_+ \geq 0, a_- \geq 0$ for some non-zero time cannot leave this region at $a_0 = 0, a_+ \geq 0, a_- \geq 0$. 

One obtains the same result for $a_+$ and $a_-$ by repeating the proof with permuted subscripts $ 0, + , - $.     
\item We show that the boundary of the unit cube also bounds solutions lying inside it. By the symmetry of permuting $ 0, + , - $, and the previous result, it will be enough to show this for the top face of the cube: i.e. prove that a solution with $1>a_0 > 0, 1 \geq a_+ > 0, 1 \geq a_- > 0$ cannot leave this region via $a_0 = 1, 1 \geq a_+ > 0, 1 \geq a_- > 0$. We have, at $a_0 = 1,  1 \geq a_+ > 0, 1 \geq a_- > 0$, $\dot{a}_0 = f_0 \left ( a_+ a_- - 1 \right) \leq 0$ , with equality iff both $a_+ = a_- = 1$. However since $\left(1,1,1\right)$ is a critical point for \eqref{instB2}, this cannot be the case, hence the inequality is strict, and we cannot have a solution with $1>a_0 > 0, 1 \geq a_+ > 0, 1 \geq a_- > 0$ leaving this region at $a_0 = 1, 1 \geq a_+ > 0, 1 \geq a_- > 0$, arguing as before. 
\item The proof that the quadrant $\left(1,\infty\right)^3$ bounded by the planes $a_0=1$, $a_+=1$, and $a_- = 1$ goes almost exactly as for the previous part of the lemma: at $a_0 = 1,   a_+ \geq 1 , a_- \geq 1 $, $\dot{a}_0 = f_0 \left ( a_+ a_- - 1 \right) \geq 0$ with equality iff both $a_+ = a_- = 1$, hence the inequality is strict, and we cannot have a solution leaving this region via $a_0 = 1,   a_+ \geq 1 , a_- \geq 1 $.  
\end{enumerate} 
\end{proof}
Having established these results, we can immediately see from the local solutions $S_\xi$ in  \eqref{localsmoothing} for some sufficiently small non-zero time $\left(a_0, a_+, a_-\right)_\xi \in (0,1)^3$ for $0<\xi<1$, and $\left(a_0, a_+, a_-\right)_\xi \in (1,\infty )^3$ for $1<\xi$, so we have a rough bound on the behaviour of our solutions as $t\rightarrow \infty$. However, we can use the previous lemma to show an improved statement: 
\begin{lemma} The following sets in $\mathbb{R}^3$ are forward-invariant  for \eqref{instB2}:
\begin{enumerate}[\normalfont(i)]
\item $\mathcal{S}_\infty := \lbrace \left(a_0, a_+, a_-\right) \in \mathbb{R}^3 \mid a_+ a_- > a_0 >1, a_0 a_- > a_+ >1, a_0 a_+ > a_- > 1 \rbrace$
\item  $\mathcal{S}_0 := \lbrace \left(a_0, a_+, a_-\right) \in \mathbb{R}^3 \mid  0 < a_+ a_- < a_0 <1, 0 <  a_0 a_- < a_+ < 1, 0 < a_0 a_+ < a_- < 1 \rbrace$.
\end{enumerate} 
\end{lemma}
\begin{proof}
Given an ODE system $\dot{\mathbf{x}}=\mathbf{F}\left(\mathbf{x},t\right)$ in $\mathbb{R}^n$, if one has a hypersurface $h(\mathbf{x})=0$ such that $\nabla h \cdot \mathbf{F} \left(\mathbf{x},t\right) > 0$, where $\nabla$ is the gradient of $h$, and $``\cdot"$ denotes the standard dot product on $\mathbb{R}^n$, then for all time for which a smooth solution $\mathbf{x}(t)$ exists, it can only cross hypersurface $h(\mathbf{x})=0$ in the same direction as $\nabla h$. 
   
In the case of \eqref{instB2}, we use the hypersurfaces $\lbrace \left(a_0, a_+, a_-\right) \in \mathbb{R}^3 \mid a_0 = a_+ a_- \rbrace$, $\lbrace  \left(a_0, a_+, a_-\right) \in \mathbb{R}^3 \mid  a_+ = a_0 a_- \rbrace$, and $\lbrace  \left(a_0, a_+, a_-\right) \in \mathbb{R}^3 \mid  a_- = a_+ a_0 \rbrace$: 
\begin{enumerate}[(i)]
\item $\mathcal{S}_\infty$ is the region in $\left(1,\infty\right)^3 $ bounded by these three paraboloids, with triple intersection at $\left( 1, 1, 1 \right)$, and intersecting pairwise along three line segments in $\mathbb{R}^3$. We can exclude the intersections: note that $\lbrace \left(a_0, a_+, a_-\right) \in \left[ 1,\infty\right)^3 \mid a_+ = a_0 a_-,  a_- = a_0 a_+ \rbrace = \lbrace \left(a_0, a_+, a_-\right) \in \left[ 1,\infty\right)^3 \mid a_- = a_+, a_0 =1 \rbrace$ which lies in the boundary of $\left(1 , \infty \right)^3$ so using the previous lemma, and the symmetry of permuting $0,+,-$, it will be enough to prove that a solution contained in $\mathcal{S}_\infty$, at some initial time, cannot leave via $\lbrace \left(a_0, a_+, a_-\right) \in \left(0,\infty\right)^3 \mid a_0 = a_+ a_- \rbrace$. We calculate for $h= a_+ a_- - a_0$, with $a_0 > 1, a_+ > 1, a_- >1$: 
\begin{align*}
\left. \nabla h \cdot \left( \dot{a}_0, \dot{a}_+, \dot{a}_- \right) \right|_{h=0} &= \left( -1, a_-, a_+ \right) \cdot \left. \left( f_0 ( a_+ a_- - a_0 ) , f_+ ( a_0 a_- - a_+ ) , f_- ( a_0 a_+ - a_- ) \right)\right|_{a_0 = a_+ a_-} \\ 
&= a_+ a_- \left( f_+ ( {a_-}^2 -1 ) + f_- ( {a_+}^2 -1 ) \right) > 0 
\end{align*}
Repeating the proof with indices $0,+,-$ permuted gives the result for surfaces defined by $a_0 a_- - a_+ =0$ and $a_0 a_+ - a_- =0$ respectively.
\item $\mathcal{S}_0$ is also bounded by these three paraboloids, but in $\left(0,1\right)^3$, $\nabla h$ (as we have defined it) points outward. As for the intersections, we can again exclude them, as before for $\left(a_0, a_+, a_-\right) \in \left(0,\infty\right)^3$, but also for $\lbrace \left(a_0, a_+, a_-\right) \in \left[0,1\right]^3 \mid a_+ = a_0 a_-,  a_- = a_0 a_+ \rbrace = \lbrace \left(a_0, a_+, a_-\right) \in \left[0,1\right]^3 \mid a_+ = a_-, a_0 = 1 \rbrace \cup \lbrace \left(a_0, a_+, a_-\right) \in \left[0,1\right]^3 \mid a_+ = a_- = 0 \rbrace $, which lies in the boundary of the unit cube. Now the calculation is exactly the same as the previous part of the lemma, with $0<a_0 < 1, 0< a_+ < 1, 0 <a_- <1$ and $h= a_+ a_- - a_0$, giving $\left. \nabla h \cdot \left( \dot{a}_0, \dot{a}_+, \dot{a}_- \right) \right|_{h=0} < 0$.
\end{enumerate} 
\end{proof}
Note that solutions $ \left(a_0, a_+, a_-\right)$ to \eqref{instB2} lying inside $\mathcal{S}_0$, $\mathcal{S}_\infty$ have $a_0, a_+, a_-$ monotonic in $t$. We can then use this fact to determine their  asymptotic behaviour: 
\begin{lemma} A solution $ \left(a_0, a_+, a_-\right)$ to \eqref{instB2} lying inside $\mathcal{S}_0$ at some time $t^*>0$, exists for all forward time $t\geq t^*$, and is asymptotic as $t \rightarrow \infty$ to $\left( 0,0,0\right)$. A solution  $ \left(a_0, a_+, a_-\right)$ lying inside $\mathcal{S}_\infty$ at some time $t^*$ cannot be bounded for all $t\geq t^*$.    
\end{lemma}
\begin{proof}
We begin by looking at solutions lying in $\mathcal{S}_0$. Forward-time existence and boundedness of these solutions follows from the boundedness of $\mathcal{S}_0$, and since $a_0, a_+, a_-$ are all (strictly) monotonically decreasing in $\mathcal{S}_0$, the solution $ \left(a_0, a_+, a_-\right)$ must have a limit lying in the closure. To determine that limit, we reparameterize \eqref{instB2} in terms of the variable $s$, as in the explicit solutions given by \eqref{hypoBsol}: 
 \begin{gather} \label{instB3}
 \begin{aligned} 
\dot{a}_0 &= \frac{4\lambda^2}{\mu} ( a_+ a_- - a_0 )  \\
\dot{a}_+ &= \frac{3 (v_3 + v_0)}{2 \mu}(  a_0 a_- - a_+ ) \\ 
\dot{a}_- &= \frac{3 (v_3 - v_0)}{2 \mu} ( a_0 a_+ - a_-) 
 \end{aligned}
 \end{gather}
In particular, by using \eqref{hypoBsol}, one can check that $\lambda f_{0}\rightarrow C_{0}>0$ as $s\rightarrow \infty$ for some strictly positive constant $C_{0}$, and similarly $\lambda f_{\pm}\rightarrow C_{\pm}>0$. If a solution $ \left(a_0, a_+, a_-\right) $ to \eqref{instB3} lying in $\mathcal{S}_0$ does not have $a_+ a_- - a_0 \rightarrow 0$ as $s \rightarrow \infty$, then we get a contradiction: otherwise for $s$ sufficiently large we can bound $\dot{a}_0$ above, away from $0$. Said more explicitly, if we do not have $a_+ a_- - a_0 \rightarrow 0$, then we do not have $\dot{a}_0 \rightarrow 0$, so for some constant $C^*_0 <0$, there exists $s^* \gg 0$ such that $\dot{a}_0(s) < C^*_0$ for all $s\geq s^*$. Integrating this inequality would give the contradiction $a_0 \rightarrow - \infty $ as $s\rightarrow \infty$, thus we must have $a_+ a_- - a_0 \rightarrow 0$ as $s\rightarrow \infty$.   

One then repeats this argument for $a_\pm (s)$, to obtain that a solution in $\mathcal{S}_0$ must tend to a critical point of this system in the closure of $\mathcal{S}_0$ as $s\rightarrow \infty$: either $(0,0,0)$, or $(1,1,1)$ by Lemma \ref{criticalpoints}. Since $a_0, a_+, a_-$ are all strictly decreasing, we must have $\left(a_0, a_+, a_-\right) \rightarrow \left( 0,0,0\right)$. 
 
Now we deal with solutions $ \left(a_0, a_+, a_-\right)$ to \eqref{instB2} lying in $\mathcal{S}_\infty$. These have $a_0, a_+, a_-$ strictly increasing as long as the solution exists, so again, if a solution is bounded and exists for all time, it must have limit lying in the closure of $\mathcal{S}_\infty$.  Let us assume this is the case and derive a contradiction: since the right-hand side of \eqref{instB3} has a limit as $s \rightarrow \infty$, this implies that $\left( \dot{a}_0, \dot{a}_+, \dot{a}_- \right)$ must also have a limit. Since $\lambda f_{0}\rightarrow C_{0}>0$ we have, for a fixed constant $C_0^*>0$, some $S> 0$ such that for all $s>S$: 
 \begin{align*} 
\dot{a}_0 > C^*_0 ( a_+ a_- - a_0 ) 
\end{align*}
and likewise for $\dot{a}_\pm$. As such, a bounded solution existing for all time cannot have simultaneously $\dot{a}_0, \dot{a}_+, \dot{a}_-  \rightarrow 0$ as $s \rightarrow \infty$, since this would require $ \left(a_0, a_+, a_-\right) \rightarrow (1,1,1)$, which is impossible by the monotonicity of  $a_0, a_+, a_-$. Therefore, at least one of $\dot{a}_0, \dot{a}_+, \dot{a}_- $ must be bounded below away from $0$ for $s$ sufficiently large, and hence the corresponding $a_0, a_+, a_- $ must be unbounded above as $s\rightarrow \infty$.  
\end{proof}
We can now conclude the proof of Theorem \ref{stenzelthm}: the first point is clear by applying the symmetry outlined in \eqref{symmetrymetric} to the local power-series of $\left(a_0, a_+, a_-\right)_\xi$, i.e. \eqref{localsmoothing}, and noting that the fixed point $\xi =0$ is the explicit solution \eqref{explicit}. For the rest, by using \eqref{localsmoothing}, one finds the flat connection $\left(a_0, a_+, a_-\right)_1 = \left(1, 1, 1\right)$ is a critical point, and:
\begin{align*}
a_0 - a_- a_+ &= -\frac{9}{10} \left( \xi^2 -1 \right) \xi t^2 + O(t^4) \\
a_+ - a_0 a_- &= 1 - \xi^2  - \frac{45-63 \xi^2}{40} \left( \xi^2 -1 \right) t^2 + O(t^4) \\
a_- - a_+ a_0 &= -\frac{27}{20} \left( \xi^2 -1 \right) \xi t^2 + O(t^4) 
\end{align*}
In particular, for non-zero $t$ sufficiently small, and $0<\xi<1$, we have $\left(a_0, a_+, a_-\right)_\xi (t) \in \mathcal{S}_0$, while for $1<\xi$ we have $\left(a_0, a_+, a_-\right)_\xi (t) \in \mathcal{S}_\infty$. Using the symmetry \eqref{symmetrymetric} for $\xi<0$, Theorem \ref{stenzelthm} follows. 
\end{proof}
On $\mathcal{O}(-1) \oplus \mathcal{O}(-1)$ and $\mathcal{O}(-2,-2)$ there are multiple ways of extending the invariant bundle $P_1$ to the singular orbit. The local solutions on each extension exhibit a slightly different behaviour:
\begin{theorem} \label{candelasthm1} Invariant instantons with quadratic curvature decay on $P_{0,\mathrm{Id}} \rightarrow \mathcal{O}(-1) \oplus \mathcal{O}(-1)$ are in a one-parameter family $R_\epsilon$, $\epsilon \geq 0$, up to gauge. Moreover:
\begin{enumerate}[\normalfont(i)]
\item  $R_{0} = A_2^\flat$, and $R_\epsilon$ are irreducible for $\epsilon > 0$. 
\item  $\lim_{t \to \infty} R_\epsilon (t) = A^\mathrm{can}$ for $\epsilon > 0$.
\end{enumerate}
\end{theorem}
\begin{theorem} \label{candelasthm2} Invariant instantons with quadratic curvature decay on $P_{1, \mathbf{0}} \rightarrow \mathcal{O}(-1) \oplus \mathcal{O}(-1)$ are in a one-parameter family $R'_{\epsilon'}$, $0\leq \epsilon' \leq 1$, up to gauge. Moreover:
\begin{enumerate}[\normalfont(i)]
\item  $R'_{0}$ is abelian, $R'_{1}=A_1^\flat$, and $R'_{\epsilon'}$ are irreducible for $0 < \epsilon' < 1$. 
\item  $\lim_{t \to \infty} R'_{\epsilon'}(t) = A^{\mathrm{can}}$ for $0 \leq \epsilon' < 1$.
\end{enumerate}
\end{theorem}
For insantons over $\mathcal{O}(-2,-2)$, we also split the statement of the theorem into two cases. The first case is similar to the situation of Theorem \ref{candelasthm2}:
\begin{theorem} \label{pandozayasthm1} Invariant instantons with quadratic curvature decay on $P_{1-l,l} \rightarrow \mathcal{O}(-2,-2)$ with $l= 0, 1$, are in a one-parameter family $Q^l_{\alpha_l}$,$0 \leq \alpha_l  \leq 1$, up to gauge. Moreover:
\begin{enumerate}[\normalfont(i)]
\item  $Q^l_{0}$ is abelian, $Q^0_{1}=A_1^\flat$, $Q^1_{1}=A_2^\flat$, and $Q^l_{\alpha_l}$ are irreducible for $0<\alpha_l<1$.  
\item  $\lim_{t \to \infty} Q^l_{\alpha_l}(t) = A^{\mathrm{can}}$ for $0 \leq \alpha_l  < 1$. 
\end{enumerate}
\end{theorem} 
The second case exhibits a new phenomenon: now, the instantons appearing at the boundary of the moduli-space of instantons with asymptotics $A^\mathrm{can}$ are not themselves flat, but are asymptotic to the flat connection:  
\begin{theorem} \label{pandozayasthm2} Invariant instantons with quadratic curvature decay on $P_{1-l,l}\rightarrow \mathcal{O}(-2,-2)$ with $l \neq 0, 1$, are in a one-parameter family $Q^l_{\alpha_l}$, $0 \leq \alpha_l  \leq \alpha^\mathrm{crit}_l$ for some $\alpha^\mathrm{crit}_l>0$, up to gauge. Moreover:
\begin{enumerate}[\normalfont(i)]
\item  $Q^l_{0}$ are abelian, and $Q^l_{\alpha_l}$ are irreducible for $0 < \alpha_l  \leq \alpha^\mathrm{crit}_l$.  
\item  $\lim_{t \to \infty} Q^l_{\alpha_l}(t) = A^{\mathrm{can}}$ for $0 \leq \alpha_l  < \alpha^\mathrm{crit}_l$, $\lim_{t \to \infty} Q^l_{\alpha^\mathrm{crit}_l}(t) = A_1^\flat$ for $l<0$, and $\lim_{t \to \infty} Q^l_{\alpha^\mathrm{crit}_l}(t) = A_2^\flat$ for $l>1$. 
\end{enumerate}
\end{theorem} 
\begin{proof}[Proof of Theorems \ref{candelasthm1}, \ref{candelasthm2}, \ref{pandozayasthm1}] Most of what is required to prove these theorems boils down to studying the qualitative behaviour of a single ODE system. We study solutions to \eqref{dynamicODE} with $\phi=0$, $a_1=0$: 
 \begin{align} \label{instA2} 
\dot{a}_0 =  - \frac{4 \lambda}{\mu^2} \left( a_2^2 ( u_1 - u_0) +  a_0 u_1 + u_0 \right)& &\dot{a}_2 = - \frac{3}{2\lambda} a_2 \left(a_0+1 \right) 
 \end{align}
where we have a generic family of Calabi-Yau structures defined by hypo-structures of type $\mathcal{I}$, so that $u_0, u_1, \mu, \lambda$ are non-degenerate solutions to hypo-evolution equations \eqref{hypoAevolution}. We consider forward-invariant sets for this system, see also Fig. \ref{instA2figure} below:
 \begin{lemma} \label{absorbinglemmaA} The following sets in $\mathbb{R}^2$ are forward-invariant under \eqref{instA2}:
 \begin{enumerate}[\normalfont(i)]
 \item Half-planes $\lbrace \left( a_0, a_2 \right) \in \mathbb{R}^2 \mid \pm a_2 > 0 \rbrace$
 \item $\mathcal{R}_\infty := \lbrace \left( a_0, a_2 \right) \in \mathbb{R}^2 \mid  a_0 <-1, 1 < a_2  \rbrace$
 \item $\mathcal{R}_0 := \lbrace \left( a_0, a_2 \right) \in \mathbb{R}^2 \mid  -1< a_0 <1, 0 < a_2 < 1  \rbrace$
 \end{enumerate}
\end{lemma}
 \begin{proof}
 \begin{enumerate}[(i)] 
 \item Since there is always a non-trivial abelian solution $\left(a_0, 0 \right)$ to \eqref{instA2}, by uniqueness a solution hitting $a_2 = 0$ at some time $t^*>0$ must be there for all time $t>0$. Furthermore, since the symmetry \eqref{symmetrygauge} exchanges the upper/lower-half planes, we can reduce to the case of $a_2>0$ in what follows. 
 \item In the following, we will split the upper-half plane into four quadrants centred around the critical point $\left(-1,1\right)$, and look at the sign of $\dot{a}_0$ along $a_0 = -1$ and $\dot{a}_2$ along $a_2 = 1$.  
 
Since $\lambda > 0$ for all $t>0$, and $a_2>0$ by assumption, the sign of $\dot{a}_2$ is the same as that of $-(a_0+1)$, and the sign of $\dot{a}_0$ is the same as that of $-\left( ( a_2^2 -1 )(u_1 - u_0) + (a_0 + 1) u_1 \right) $. Then $\dot{a}_2 > 0$ for all $a_0 < -1$. Since $\lambda > 0$, solutions to the hypo equations \eqref{hypoAevolution} must have $u_1 \pm u_0$ strictly increasing. In addition we must have $\mu = \sqrt{u_1^2 - u_0^2}>0$ for all time $t>0$, so $u_1 \pm u_0$ must be strictly positive for all time $t>0$, and hence also $u_1$. Thus at $a_0 = -1$, we have that $\dot{a}_0 <0$ iff $a_2 > 1$. Thus a solution in $\mathcal{R}_\infty$ at some initial time $t^*>0$, cannot leave via either of its boundaries $a_0 =  -1$ or $a_2=1$, and since the intersection $\left(-1,1\right)$ is a critical point, the solution must remain in $\mathcal{R}_\infty$ for all time $t>t^*$.  
 \item As shown in the first part of the lemma, no solution can hit $a_2=0$, the bottom of $\mathcal{R}_0$, unless it is contained in $a_2 = 0$ for all time. From the proof of the second part of the lemma, we see that a solution in $\mathcal{R}_0$ cannot exit $\mathcal{R}_0$ via. the top $a_2=1, a_0>-1$ , or the side $a_0 = -1, a_2 \leq 1$. All that remains to show is that the side $a_0=1$, $1>a_2 > 0$ is bounding. This follows from the fact that $u_1 \pm u_0$ must be strictly positive, since at $a_0 = 1$, $\dot{a}_0 = - \tfrac{4 \lambda}{\mu^2} (( a_2^2 ( u_1 - u_0) + u_1 + u_0) < 0$. 
 \end{enumerate}
 \end{proof}
These sets determine the behaviour of solutions lying inside them: 
 \begin{lemma} A solution $ \left(a_0, a_2 \right)$ to \eqref{instA2} lying inside $\mathcal{R}_0$ at initial time $t^*>0$ exists for all forward time $t\geq t^*$, and is asymptotic as $t \rightarrow \infty$ to $\left( 0,0\right)$. A solution lying inside $\mathcal{R}_\infty$ at initial time $t^*>0$ cannot be bounded for all $t\geq t^*$.     
\end{lemma}
\begin{proof}
For the bounded set $\mathcal{R}_0$, it is clear that solutions exist for all time, but it remains to prove their asymptotic behaviour. Since $\dot{a}_2 < 0 $ in $\mathcal{R}_0$, $a_2$ is strictly decreasing, and as it is bounded below, $a_2$ must have a limit $\hat{a}_2 \in \left[0, 1 \right)$ as $t \rightarrow \infty$. To get a limit for $a_0$, notice that the first equation for the ODEs \eqref{instA2}, together with hypo evolution equations \eqref{hypoAevolution}, gives:
\begin{align} \label{differentialformula}
\frac{d}{dt} \left( a_0 \mu^2 \right) = - 4 \lambda \left( a_2^2 (u_1 - u_0) + u_0 \right) 
\end{align}
Written in integral form on the interval $t\geq t^*$, this is the equation: 
 \begin{align} \label{integralformula}
 a_0 (t) = -\frac{1}{\mu^2} \left( \left( \int^t_{t^*} 4 \lambda \left( a_2^2 (u_1 - u_0) + u_0 \right) \right) + a_0 (t^*) \mu^2 (t^*) \right) 
\end{align}
Since the hypo-structure $\lambda, u_1, u_0, \mu$ is asymptotically conical as a function of $t$ and $a_0$ bounded, as  $t\rightarrow \infty$ \eqref{integralformula} gives:
\begin{align*}
 a_0(t) \sim - \frac{1}{t^4}  \int_{T}^t 4 t \left( \hat{a}_2^2 t^2 + (\hat{a}_2^2 - 1 ) u_0 \right)  \sim  - \hat{a}_2^2  - \frac{2 u_0 \left( \hat{a}_2^2 -1 \right)}{t^2} + O(t^{-4}) \sim  - \hat{a}_2^2  
\end{align*}
for some $T \geq t^*$ sufficiently large. Hence we also have a limit $a_0 \rightarrow - \hat{a}_2^2$ as $t \rightarrow \infty$. Since $a_2>0$, integrating the second equation of \eqref{instA2} gives us, as $t \rightarrow \infty$:
\begin{align*}
a_2(t) = a_2(T) \exp \left( - \int_{T}^t \frac{3}{2\lambda} ( a_0 + 1 ) \right) \sim a_2(T) \exp \left( ( \hat{a}_2^2 - 1 ) \int_{T}^t \frac{3}{2 t} \right) = C t^{ \frac{3}{2}( \hat{a}_2^2 - 1 )}   
\end{align*}
where $C$ is some constant of integration. As $\hat{a}_2 < 1$, this implies $a_2 \rightarrow 0$, and thus also $a_0 \rightarrow 0$.  
 
Now we come to solutions lying in $\mathcal{R}_\infty$. Since $\mathcal{R}_\infty$ is forward-invariant, and a solution lying in $\mathcal{R}_\infty$ has $\dot{a}_2 > 0$ for all finite $t$, the statement for finite $t$ follows directly from the previous lemmas. All that is left is to prove that if a solution exists for all time in $\mathcal{R}_\infty$, then it cannot be bounded. We will assume that it is, and derive a contradiction: 

If a solution is bounded, then since $a_2$ is strictly increasing in $\mathcal{R}_\infty$, $a_2$ must have a limit as $t \rightarrow \infty$, and as before, the integral formula \eqref{integralformula}, and the boundedness of $a_0$ gives that $\left(a_0, a_2 \right)$ must have a limit lying on the curve $a_0 = - a_2^2$. Since $a_2$ is strictly increasing, we can bound $a_2$ away from $1$, thus for some $t$ large enough, we can also bound $a_0$ away from $-1$. Call this bound $C$, i.e. there exists $T$, such that for $t>T$ we have $a_0 < C < -1$. Then we also have that: 
\begin{align*}
\dot{a}_2 > - \frac{3}{2\lambda} a_2 (C+1)
\end{align*}
So by integrating this inequality, we get: 
\begin{align*}
a_2 (t) \geq a_2(T) \exp \left( - \frac{3}{2\lambda}(C+1) \int_T^t \frac{1}{\lambda}  \right)
\end{align*}
but the right-hand side grows to $O(t^{- \frac{3 (C+1)}{2}})$ as $t\rightarrow \infty$, hence we have a contradiction. 
 \end{proof}
We now conclude the proof of Theorem \ref{candelasthm1}, by applying our analysis above to the local power-series $R_\epsilon$ of Proposition \ref{localresol1} with $\delta =0$, so that $\left(a_0, a_2 \right)_\epsilon \in \mathcal{R}_0$ for $\epsilon > 0 $ while $\left(a_0, a_2 \right)_\epsilon \in \mathcal{R}_\infty$ for $\epsilon < 0 $ at sufficiently small non-zero time. Taking $\epsilon =0$ gives the flat connection $\left(a_0, a_2 \right)_0 =\left( -1,1\right)$, which is a critical point of \eqref{instA2}. 
 
Theorems \ref{candelasthm2}, and \ref{pandozayasthm1}, also follow from what has been said. In the first case, in order to apply the results of the previous lemmas, one must first pull-back the Calabi-Yau structure via the involution $u_0 \mapsto -u_0$ by exchanging the factors of $SU(2)$ on the principal orbits, which pulls back the local solutions to solutions of the form \eqref{localresol2}. These invariant instantons extend on the singular orbit $SU(2)^2 / SU(2) \times U(1)$ rather than $SU(2)^2 / U(1) \times SU(2)$ as is our convention, but one can fix this by again applying the involution lifted to the total space of the principal bundle i.e. \eqref{symmetrymetric}. Similarly for the latter case, to consider $l=0$, one considers the local solutions on $P_{0,1}$ for the original Calabi-Yau structure pulled-back via the involution, and then applies the involution again on the total space of $P_{0,1}$ to get the result on $P_{1,0}$.  

With this in mind, we can apply our analysis to the local power-series \eqref{localresol2} and \eqref{localresol3}. We see that these situations are basically the same in terms of the gauge theory: up to invariant gauge transformation \eqref{symmetrygauge}, for some sufficiently small $t^* > 0$, for $1 >\epsilon' > 0$ (resp. $1> \alpha_1 >0$) we have $\left(a_0, a_2 \right)_{\epsilon'}(t^*) \in \mathcal{R}_0$ (resp. $\left(a_0, a_2 \right)_{\alpha_1}$), while for $\epsilon'>1$, we have $\left(a_0, a_2 \right)_{\epsilon'}(t^*) \in \mathcal{R}_\infty$ (resp. $\left(a_0, a_2 \right)_{\alpha_1}(t^*)$). We also see that $\left(a_0, a_2 \right)_{0}(0)= \left(-1,0 \right)$, hence by uniqueness $\left(a_0, a_2 \right)_0$ must correspond to the abelian solution to \eqref{instA2}, and $\left(a_0, a_2 \right)_{1}(0)=\left( -1,1\right)$. 
\end{proof}  
The proof for the remaining case of Theorem \ref{pandozayasthm2} requires slightly more care: 
\begin{proof}[Proof of Theorem \ref{pandozayasthm2}] We are again studying solutions to the ODE \eqref{instA2}. Looking at the local power-series solutions in Proposition \ref{localresol4}, we see that they do not initially lie in the sets $\mathcal{R}_0$ or $\mathcal{R}_\infty$ covered in our previous analysis. However, we will show that the only possibilities are that such solutions either enter $\mathcal{R}_0$ or $\mathcal{R}_\infty$ in finite time, or are otherwise asymptotic to the flat connection $A_2^\flat$: 
\begin{lemma} \label{r0lemma} Let $\mathcal{R}_1 := \lbrace \left( a_0, a_2 \right) \in \mathbb{R}^2 \mid 1 > a_2 > 0, a_0 < -1 \rbrace$. A solution $\left( a_0, a_2 \right)$ to \eqref{instA2} lying in $\mathcal{R}_1$ at initial time $t^*>0$ can remain in $\mathcal{R}_1$ for all forward time $t\geq t^*$ only if it is asymptotic to $\left( -1, 1 \right)$ as $t\rightarrow \infty$, and must otherwise enter one of $\mathcal{R}_0$, $\mathcal{R}_\infty$ in finite time.   
\end{lemma}
\begin{proof} We have that, in $\mathcal{R}_1$: 
\begin{align*}
\dot{a}_0 =  - \frac{4 \lambda}{\mu^2} \left( (a_2^2 -1)( u_1 - u_0) +  (a_0 +1) u_1 \right)>0& &\dot{a}_2 = - \frac{3}{2\lambda} a_2 \left(a_0+1 \right) > 0
\end{align*}
Hence, a solution lying in $\mathcal{R}_1$ can only leave in finite time via the boundaries $\lbrace a_2 = 1, a_0 <-1 \rbrace$ or $ \lbrace a_0 = - 1, 0<a_2<-1 \rbrace$, since $\left( -1, 1 \right)$ is a critical point for \eqref{instA2}. Since $\dot{a}_2$ is strictly positive on the first boundary, and $\dot{a}_0$ is strictly positive on the second, this proves that if a solution leaves $\mathcal{R}_1$ in finite time, it must actually leave the boundary and end up in the regions $\mathcal{R}_\infty$, $\mathcal{R}_0$ respectively.

If a solution remains in $\mathcal{R}_1$ for all forward-time, then by monotonicity $\left(a_0, a_2\right) $ has a limit lying in the closure. The existence of a limit, combined with the integral formula \eqref{integralformula}, gives that $\left(a_0, a_2\right) $ must also tend to a point lying on the curve $a_0 = - a_2^2$, which only intersects the closure of $\mathcal{R}_1$ at $\left( -1, 1 \right)$. 
\end{proof}
\begin{figure}[h]
\centering
\begin{tikzpicture}[scale=1.3]
\draw[help lines, color=gray!30, dashed] (-1.9,-0.9) grid (1.9,1.9);
\draw[] (-2,0)--(2,0) node[below right]{$a_0$};
\draw[] (0,-1)--(0,2) node[above right]{$a_2$};
\draw[] (-1,0) -- (-1,-0.1) node[below]{\small{-1}};
\draw[] (1,0) -- (1,-0.1) node[below]{\small{1}};
\fill [fill=orange, opacity=0.3] (-2,0) rectangle (-1,1);

\fill [fill=red, opacity=0.3] (-2,1) rectangle (-1,2);

\fill [fill=teal, opacity=0.3] (-1,0) rectangle (1,1);

\draw[thick] (-2,1)  -- (-1,1) node[above right]{$A^\flat_2$} -- (1,1);
\draw[thick] (-2,0) -- (0,0) node[below right]{$A^\mathrm{can}$} -- (1,0);
\draw[thick] (-1,0) -- (-1,2);
\draw[thick] (1,0) -- (1,1);
\fill[fill=black] (-1,1) circle (1.2pt); 
\fill[fill=black] (0,0) circle (1.2pt); 
\node at (-0.5,0.5) {$\mathcal{R}_0$};
\node at (-1.5,0.5) {$\mathcal{R}_1$};
\node at (-1.5,1.5) {$\mathcal{R}_{\infty}$};
\end{tikzpicture}
\caption{Distinguished sets for \eqref{instA2}, and possible asymptotics: the flat connection $A^\flat_2$ at $\left(a_0, a_2 \right)=\left( -1, 1 \right)$ and $A^\mathrm{can}$ at $\left(a_0, a_2 \right)= \left( 0, 0 \right)$.}
\label{instA2figure}
\end{figure}
We must also prove a comparison lemma for two solutions to \eqref{instA2}, which will allow us to compare our power-series solutions away from the singular orbit at $t=0$:
\begin{lemma}[Forward-Comparison] \label{forwardcomparison} Let $\left( a_0, a_2 \right)$, $\left( \hat{a}_0,  \hat{a}_2 \right)$ be two solutions to \eqref{instA2}. If $ a_0( t^* ) < \hat{a}_0 ( t^* )$, $a_2( t^* ) > \hat{a}_2 ( t^* ) \geq 0$, at initial time $t^*>0$, then $a_0( t ) < \hat{a}_0 ( t )$, $a_2( t ) > \hat{a}_2 ( t )\geq 0$, for all forward time $t \geq t^*$ for which these solutions exist. 
\end{lemma}
\begin{proof}
Let $t>t^*>0$ be the first time for which the condition $ a_0 < \hat{a}_0$, $a_2 > \hat{a}_2 $ fails. By uniqueness of solutions to ODEs, we cannot have both $a_0( t ) = \hat{a}_0 ( t )$ and $a_2( t ) = \hat{a}_2 ( t )$, hence we must have exactly one of these. In the first case, at $t$: 
\begin{align*}
\dot{a}_0-\dot{\hat{a}}_0 = - \frac{4 \lambda}{\mu^2} ((  a_2^2  - \hat{a}^2_2 )  (u_1 - u_0) )<0 
\end{align*}
but this implies $a_0( t^{**} )-\hat{a}_0 ( t^{**} ) > 0 $ for some $t^*<t^{**}<t$, which contradicts $t$ being the first time the condition fails. In the second case, at $t$:  
\begin{align*}
\dot{a}_2 - \dot{\hat{a}}_2 = - \frac{3 }{2 \lambda} ((  a_0-  \hat{a}_0 ) a_2  >0 
\end{align*}
but this implies $a_2( t^{**} )-\hat{a}_2 ( t^{**} ) <0 $ for some $t^*<t^{**}<t$, which is again a contradiction. 
\end{proof}
Another ingredient we will need is a slight improvement on the comparison lemma, restricted to solutions lying in $\mathcal{R}_1$:
\begin{lemma}[Improved Comparison] Let $\left( a_0, a_2 \right)$, $\left( \hat{a}_0,  \hat{a}_2 \right)$ be two solutions to \eqref{instA2}, with $ a_0( t^* ) < \hat{a}_0 ( t^* )$, $a_2( t^* ) > \hat{a}_2 ( t^* )\geq 0$, at some initial time $t^*>0$. Then $a_2 - \hat{a}_2$ is strictly increasing $\forall t \geq t^*$ for which $\left( a_0, a_2 \right)(t) \in \mathcal{R}_1$.  
\end{lemma}
\begin{proof}
By the forward-comparison lemma, $ a_0 < \hat{a}_0$, $a_2> \hat{a}_2 \geq 0$ for all time  $t \geq t^*$, and by definition $\left( a_0 +1 \right)<0$ for all time $ t \geq t^*$ such that $\left( a_0, a_2 \right)(t) \in \mathcal{R}_1$. Rewriting $\dot{a}_2 - \dot{\hat{a}}_2$ using \eqref{instA2}: 
\begin{align*}
\dot{a}_2 - \dot{\hat{a}}_2 = \frac{3}{2\lambda} \left( \hat{a}_2 \left( \hat{a}_0 - a_0 \right) + \left( \hat{a}_2 - a_2 \right)\left( a_0 +1 \right) \right)>0   
\end{align*}
for all such $t$, and hence $a_2 - \hat{a}_2$ is strictly increasing in $t$ as claimed. 
\end{proof} 
With these out the way, we are almost ready to prove the theorem. First of all, it is clear that the one-parameter family $Q^l_{\alpha_l}$ of local solutions to the ODEs \eqref{instA2} given by Proposition \ref{localresol4} with $\beta_l = 0, \alpha_l>0$, are all contained in $\mathcal{R}_1$ for some $t^* >0$ sufficiently small, and up to gauge transformation \eqref{symmetrygauge} we can assume this one-parameter family has $\alpha_l\geq 0$. The local solution with $\alpha_l =0$ is clearly the abelian solution by uniqueness. 

If $\left(a_0, a_2 \right)_{\alpha} = \left( a_0^\alpha, a_2^\alpha \right)$ and $\left(a_0, a_2 \right)_{\alpha'}= ( a_0^{\alpha'}, a_2^{\alpha'} )$ are any two of these solutions, then near the singular orbit: 
 \begin{gather}
 \begin{aligned} 
a_0^\alpha - a_0^{\alpha'} &= -\frac{6}{l (U_1 + U_0)}\left( \alpha^2-\alpha'^2 \right)t^{2l} + O(t^{2l+2})  \\
a_2^\alpha - a_2^{\alpha'} &=\left( \alpha-\alpha' \right)t^{l-1} + O(t^{l+1}) 
 \end{aligned}
 \end{gather}
So, by the forward-comparison lemma, if $\left( a_0, a_2 \right)_{\alpha_l}$ hits the boundary of $\mathcal{R}_0$ in finite time (and thus enters it if $\alpha_l>0$) then so does $\left( a_0, a_2 \right)_{\alpha'_l}$ for all $0 \leq \alpha'_l \leq \alpha_l$. Similarly, if $\left( a_0, a_2 \right)_{\alpha_l}$ hits the boundary of $\mathcal{R}_\infty$ in finite time (and thus enters it), then so does $\left( a_0, a_2 \right)_{\alpha'_l}$, for all $ \alpha'_l \geq \alpha_l$. By continuous dependence on initial conditions for singular initial-value problems, these sets are disjoint open intervals in $\mathbb{R}_{\geq 0}$. Clearly, the set $\alpha_l \in \mathbb{R}_{\geq 0}$ for which $\left( a_0, a_2 \right)_{\alpha_l}$ hits the boundary of $\mathcal{R}_0$ in finite time is non-empty since it contains $0$, so to complete the theorem, we must prove:
\begin{enumerate}
\item There exists $\alpha_l>0$ such that $\left( a_0, a_2 \right)_{\alpha_l}$ enters $\mathcal{R}_\infty$ in finite time. 
\item There is at most one $\alpha_l$ such that $\left( a_0, a_2 \right)_{\alpha_l}$ remains in $\mathcal{R}_1$ for all time. 
\end{enumerate}
The latter statement follows directly from our improved forward-comparison lemma: if $\alpha > \alpha'$ then $a_0^\alpha ( t ) < a_0^{\alpha'} ( t ), a_2^\alpha( t ) > a_2^{\alpha'} ( t )$ for all $t>0$, and $\dot{a_2}^\alpha ( t ) - \dot{a}_2^{\alpha'}(t)>0$ when the solutions are in $\mathcal{R}_1$. However, if two distinct solutions lie in $\mathcal{R}_1$ for all time $t>0$, they must both be asymptotic as $t\rightarrow \infty$ to $\left( -1, 1 \right)$ which would be a contradiction as $a_2^\alpha - a_2^{\alpha'}$ can be bounded below away from $0$. 

The former statement can be proved via a rescaling argument, which we state below as a proposition:
\begin{prop} \label{propscaling} Fix $l>1$. Then $\exists \alpha_l >0, t^* >0$  such that $\left(a_0, a_2 \right)_{\alpha_l}(t^*) \in \mathcal{R}_\infty$, where $\left(a_0, a_2 \right)_{\alpha_l}$ is the one-parameter family of solutions $Q^l_{\alpha_l}$ to \eqref{instA2} near $t=0$ given in Proposition \ref{localresol4}.
\end{prop}
\begin{proof}
As has been said previously, every local solution $\left(a_0, a_2 \right)_{\alpha_l}$ is contained in  the region $\mathcal{R}_1$ for small $t>0$. Since these solutions can only fail to exist for all forward time if they leave in finite time via $\mathcal{R}_\infty$, it suffices to consider the case that, for all $\alpha_l$, these solutions exist for all time. 

We start by rescaling the Calabi-Yau structure along the fibre $\mathbb{C}_{2,-2}$ of $\mathcal{O}(-2,-2)$, by defining, for some $\delta>0$:
\begin{align}
\lambda_\delta (t):= \frac{\lambda ( \delta t)}{\delta}& &\left( u_1 \right)_\delta ( t ) :=  u_1( \delta t ) 
\end{align}
Near the singular orbit $S^2 \times S^2$, we have the power-series expansions $\lambda = 3t + O(t^3)$, $u_1 = U_1 + O(t^2)$ for some fixed Calabi-Yau structure, and hence, for fixed $t$, $\lambda_\delta (t) \rightarrow 3t$, $\left( u_1 \right)_\delta ( t ) \rightarrow U_1$ as $\delta \rightarrow 0$\footnote{if we consider the the metric on $\mathcal{O}(-2,-2)$, rescaled by the diffeomorphism $t\mapsto \delta t$, this rescaling is the adiabatic limit as $\delta \rightarrow 0$ of the product of the rescaled metric on the fibres and the two copies of $\mathbb{CP}^1$ of fixed volume. See \S \ref{bubbling} for a similar discussion.}. 
 
In terms of the rescaled Calabi-Yau structure, the instanton equations \eqref{instA2} for $\left(a_0^\delta, a_2^\delta \right) (t) = \left(a_0, a_2 \right)(\delta t) $ become the family of ODEs, parametrized by $\delta$: 
 \begin{align}\label{instAscale1} 
\dot{a}^\delta_0 =  - \frac{4 \delta^2 \lambda_\delta}{\left( u_1 \right)_\delta^2 - u_0^2} \left( a_2^2 ( \left(u_1\right)_\delta - u_0) +  a_0 \left(u_1\right)_\delta + u_0 \right)& &\dot{a}^\delta_2 = - \frac{3}{2\lambda_\delta} a_2 \left(a_0+1 \right)
 \end{align}
One can always rescale the one-parameter family of (local) solutions $\left(a_0, a_2 \right)_{\alpha_l}$ to \eqref{instA2} to obtain solutions to \eqref{instAscale1} for fixed $\delta>0$, but one can show that there is a one-parameter family of (local) solutions extending to the singular orbit for any $\delta \geq 0$. To verify this claim, we apply the boundary conditions \ref{propconnboundary} for extending an invariant connection to the singular orbit, which allows us to write $a^\delta_2 = t^{l-1} X_2^\delta$ for some smooth $X_2^\delta$. The ODEs \eqref{instAscale1} can now be written as the singular IVP: 
 \begin{align} \label{instAscale2} 
\dot{a}^\delta_0 =  O(t)& &\dot{X}^\delta_2 = - \frac{a_0 +2l -1}{2t} X^\delta_2 +O(t)  
 \end{align}
which, for every $\delta \geq 0$, has a one-parameter family of solutions by fixing $X^\delta_2(0)$ as some constant $\kappa_l$. These solutions are determined by the local power-series: 
 \begin{align}
a^{\delta}_0 = 1 -2l + O( t^2)& &a^{\delta}_2 = \kappa_l t^{l-1} + O(t^{l+1})
\end{align}
and by comparing the two power-series, it is clear that the rescaled solutions $\left(a_0, a_2 \right)_{\alpha_l}(\delta t)$ to \eqref{instAscale1} for any $\delta>0$ have $\kappa_l = \alpha_l \delta^{l-1}$. 

Meanwhile, for $\delta=0$, \eqref{instAscale1} can be solved explicitly: 
\begin{align}
a^{0}_0 = 1 -2l& &a^{0}_2 = \kappa_l t^{l-1}
\end{align}
We can always fix $\kappa_l =1$ for this solution by a further rescaling of $t$, so as $\delta \rightarrow 0$, a solution $\left(a_0^\delta, a_2^\delta \right)$ to \eqref{instAscale1} has $\left(a_0^\delta, a_2^\delta \right) (t) \rightarrow \left( 1-2l, t^{l-1} \right)$. By assumption, for all $\delta$ these solutions exist for all time, and therefore we can always find $T>1$, $\delta \ll 1$, such that $\left(a_0^\delta, a_2^\delta \right) (T) \in \mathcal{R}_\infty$. If we set $\delta (\alpha_l)$ such that $\delta^{1-l} = \alpha_l$, and take $t^*=T \delta$, then the solution $\left(a_0, a_2 \right)_{\alpha_l}$ to the instanton equations \eqref{instA2} can be rescaled to a solution of \eqref{instAscale1}, so it must satisfy $\left(a_0, a_2 \right)_{\alpha_l} (t^*) \in    \mathcal{R}_\infty$ for some $\alpha_l$ sufficiently large.
\end{proof}
This concludes the proof of Theorem \ref{pandozayasthm2}, in the case $l>1$. As before, one can consider the case $l<0$ in the same way, by first considering solutions for the pulled-back Calabi-Yau structure by exchanging the factors of $SU(2)$ on the underlying manifold, and then applying this diffeomorphism again on the total space of the principal bundle.   
\end{proof}
\subsection{Bubbling} \label{bubbling}

Having described the one-parameter family $R_\epsilon$ of solutions to \eqref{dynamicODE} on $\mathcal{O}(-1) \oplus \mathcal{O}(-1)$ in Theorem \ref{candelasthm1}, a natural question would be to ask about the behaviour of solutions as $\epsilon \rightarrow \infty $. We will show that there is a familiar bubbling phenomenon in this setting: after a suitable rescaling of the metric, the one-parameter family of Calabi-Yau  instantons $R_\epsilon$ converge as $\epsilon \rightarrow \infty$ to an anti-self-dual connection along the co-dimension four calibrated singular orbit $S^2= \mathbb{CP}^1 \subset \mathcal{O}(-1) \oplus \mathcal{O}(-1)$. We use this result to obtain the expected removable-singularity statement, which says that as $R_\epsilon$ bubbles off this anti-self-dual connection, if we do not perform this rescaling, it will uniformly converge on compact subsets of $\mathcal{O}(-1) \oplus \mathcal{O}(-1)\setminus \mathbb{CP}^1$ to the instanton $R'_0$ of Theorem \ref{candelasthm2}, which extends smoothly over $\mathbb{CP}^1$. Recall that the abelian instanton $R'_0$ is determined by the unique solution to the ODE \eqref{instA2} on $\left[ 0, \infty \right)$ with $a_2=0$, which has explicit form  \eqref{abelian0} with $C= -2 $, $u_0 = -1, u_1(0) = 1$.
 
Let us first discuss this rescaling in detail: as $\mathcal{O}(-1) \oplus \mathcal{O}(-1)$ has the structure of a vector bundle, fibre-wise multiplication equips it with an natural $SU(2)^2$-equivariant action of $\mathbb{R}_{>0}$. Let $s_\delta$ denote the corresponding $\mathbb{R}$-action for some $\delta>0$, i.e. the map fixing the singular orbit and sending $t \mapsto \delta t$ on the space of principal orbits. Pulling back the Riemannian metric $g$ on $\mathcal{O}(-1) \oplus \mathcal{O}(-1)$ as given in \eqref{hypoAmetric} by $s_\delta$:
\begin{align} \label{adiabaticmetric}
s^*_\delta g = \delta^2  \left( dt^2 +  \lambda_\delta^2 \left( \eta^{se} \right)^2 + \tfrac{2}{3} \left( u_1 + u_0 \right)_\delta \left( (v^2)^2 + (w^2)^2 \right) \right) + \tfrac{2}{3} \left( u_1 - u_0 \right)_\delta \left( (v^1)^2 + (w^1)^2 \right)
\end{align} 
where $\lambda_\delta$, $\left( u_1 + u_0 \right)_\delta$, and $\left( u_1 - u_0 \right)_\delta$ are defined as:
\begin{align} \label{hyporescaled}
\lambda_\delta (t):= \frac{\lambda ( \delta t)}{\delta}& &\left( u_1 + u_0 \right)_\delta ( t ) := \frac{\left( u_1 + u_0 \right) ( \delta t )}{ \delta^2 }& &\left( u_1 - u_0 \right)_\delta ( t ) := \left( u_1 - u_0 \right) ( \delta t )
\end{align}
We will refer to the limit $\delta \rightarrow 0$ as the \textit{adiabatic limit}, and recall that here, $u_0 = -1$, and that $\lambda(t)= \frac{3}{2}t + O(t^3)$, $u_1 = 1 + \frac{3}{2}t^2 + O(t^4) $ near $t=0$, so $\lambda_\delta$, $\left( u_1 \pm u_0 \right)_\delta$ have well-defined point-wise limits as $\delta \rightarrow 0$.
 
Near the adiabatic limit, restricted to any finite distance from the singular orbit, $s^*_\delta g$ is approximated by the metric $\delta^2 g_{F} + g_{B}$ for some $\delta$ sufficiently small, where $g_{F}$ denotes a lift of the Euclidian metric on the fibre $\mathbb{R}^4$ and $g_{B}$ denotes a round metric on the base $S^2$. Here, the lift of the Euclidian metric on $\mathbb{R}^4$ to the fibres identifies $\tfrac{3}{2} \eta^{se} =   u^1 - u^2 $, $v^2$, $w^2$ with the standard orthonormal basis of one-forms on $S^3 \subset \mathbb{R}^4$, and $v^1, w^1$ as an orthonormal basis of one-forms for the singular orbit $S^2$, viewed upstairs on $SU(2)^2 \rightarrow SU(2)^2/ U(1) \times SU(2)$. 

One can always obtain a solution to the Calabi-Yau instanton equation \eqref{CYinstantonI} on the flat Calabi-Yau 3-fold $\mathbb{C}^3$ by pulling-back any anti-self-dual connection on $\mathbb{C}^2$ to $\mathbb{C}^3 = \mathbb{C}^2 \times \mathbb{C}$, so at least at the level of tangent spaces, if we pull back some Calabi-Yau instanton by $s_{\delta}$ on some sufficiently large neighbourhood of $\mathbb{CP}^1 \subset \mathcal{O}(-1) \oplus \mathcal{O}(-1)$ in the adiabatic limit, there should appear an anti-self-dual connection pulled back from the fibre. However, the fibre bundle $\mathcal{O}(-1) \oplus \mathcal{O}(-1)$ is non-trivial, so to make this a global statement, one must first choose a connection on this bundle: we will see in the next lemma that this connection will be forced onto us by the symmetries of the problem. 

   \begin{lemma} Up to gauge and rescaling, there is a unique non-flat $SU(2)^2$-invariant anti-self-dual connection ${A}^\mathrm{asd}$ on $\mathbb{R}^4$. ${A}^\mathrm{asd}$ has a unique lift $\bar{A}^\mathrm{asd}$ to the fibres of $\mathcal{O}(-1) \oplus \mathcal{O}(-1)$ as an $SU(2)^2$-invariant connection on $P_{0,\mathrm{Id}}$: 
\begin{align}
\bar{A}^\mathrm{asd} := \frac{1}{1+ t^2} \left( E_1  \otimes(u^2 -u^1 ) +  E_2 \otimes v^2 + E_3 \otimes w^2 \right) +  E_1 \otimes u^1 
\end{align}  
\end{lemma} 
\begin{proof} We first explain how to view the $SU(2)^2$-invariant bundle $P_{0,\mathrm{Id}}$ over $\mathcal{O}(-1) \oplus \mathcal{O}(-1)$ as a bundle over the fibres: there is an obvious $SU(2)^2$-equivariant $U(1)$-action on $S^3 \times \mathbb{R}^4$, viewed as $SU(2) \times \mathbb{H}$, where $SU(2)^2$ acts on the left and $U(1)$ on the right, and this $U(1)$-action induces the quotient map $q: S^3 \times \mathbb{R}^4 \rightarrow \mathcal{O}(-1) \oplus \mathcal{O}(-1)$. By definition of $P_{0,\mathrm{Id}}$, its pull-back via $q$ is also the pull-back of an $SU(2)$-invariant bundle over $\mathbb{R}^4$, via the projection $\pi: S^3 \times \mathbb{R}^4 \rightarrow \mathbb{R}^4$ onto the second factor. Here, we view $\mathbb{R}^4$ as a co-homogeneity one manifold with group diagram $\lbrace 1 \rbrace \subset SU(2) \subseteq SU(2)$, and define the $SU(2)$-invariant $SU(2)$-bundle over $\mathbb{R}^4$ by the homomorphism $\mathrm{Id}: SU(2) \rightarrow SU(2)$, i.e. the singular isotropy group $SU(2)$ acts via the identity homomorphism on the fibre $SU(2)$. 

The canonical connection of $P_{0,\mathrm{Id}}$ over the singular orbit $S^2 = SU(2)^2/U(1) \times SU(2)$ is just the flat Maurer-Cartan form $A_2^\flat$, and this clearly pulls back via $q$ over $S^3 \times \mathbb{R}^4$ as the canonical connection  (pulled-back via $\pi$) on the singular orbit $\lbrace 0 \rbrace = SU(2)/SU(2)$ of the $SU(2)$-invariant bundle over $\mathbb{R}^4$. Using this choice of reference connection, a connection defined on $q^*P_{0,\mathrm{Id}}$ over $S^3 \times \mathbb{R}^4$ descends to $\mathcal{O}(-1) \oplus \mathcal{O}(-1)$ if and only if the corresponding adjoint-valued one-form is basic with respect to the $U(1)$-action. Furthermore, the one-form $u^1$ is the unique $SU(2)$-invariant connection on the principal $U(1)$-bundle $S^3 \rightarrow S^2$, and this induces connection on the associated vector bundle $\mathcal{O}(-1) \oplus \mathcal{O}(-1) \rightarrow S^2$. We can use this connection to project any (adjoint-valued) one-form on $S^3 \times \mathbb{R}^4$ to its semi-basic component, and thus uniquely lift any $U(1) \times SU(2)$-invariant $SU(2)$-connection over $\mathbb{R}^4$ to a connection over $\mathcal{O}(-1) \oplus \mathcal{O}(-1)$. 
 
With this understood, we now describe the invariant anti-self-dual connection ${A}^\mathrm{asd}$ over $\mathbb{R}^4$. Since the $SU(2)$-invariant bundle is trivial restricted to the principal orbit $S^3 \subset \mathbb{R}^4$, up to gauge transformation, we can always put any $SU(2)$-invariant connection $A$ on this bundle into the form: 
\begin{align*}
A = \alpha_1 E_1 \otimes u^2 + \alpha_2  E_2 \otimes v^2 + \alpha_3  E_3  \otimes w^2
\end{align*}
for some $\alpha_i(t)$ satisfying $\alpha_i(0)=1$, so that this connection extends to the singular orbit as the canonical connection.
 
The basis of self-dual forms (up to cyclically permuting  $ u^2,  v^2,  w^2$) for the Euclidean metric\footnote{the slightly non-standard orientation convention here arises from the disparity between orientation conventions for \eqref{adiabaticmetric} and the $SU(2)$-action.}  can be written as $t dt\wedge u^2 - t^2 v^2 \wedge w^2$, so the $SU(2)$-invariant anti-self-dual equations can be written as the ODE system: 
\begin{align*}
t \dot{\alpha_i} = -2 \alpha_i + 2 \alpha_j \alpha_k  
\end{align*}
where $ijk$ are cyclic permutations of $\left( 1 \, 2 \, 3 \right)$.  Imposing the additional $SU(2)$-symmetry\footnote{one can show, however, that the additional $SU(2)$-symmetry arises a-posteriori as a consequence of the general solution to the $SU(2)$-invariant anti-self-dual equations extending to the singular orbit.}, we set $\alpha_1 = \alpha_2 = \alpha_3$, giving: 
\begin{align}  \label{ASD3}
t \dot{\alpha} = -2 \alpha + 2 \alpha^2  
\end{align} 
which has the explicit solution $\alpha= \left(1+\kappa t^2\right)^{-1}$ for any $\kappa \in \mathbb{R}$, and it is not hard to see that these are the only solutions to \eqref{ASD3} extending over the singular orbit. Clearly, this exists for all $t\geq 0$ if and only if $\kappa \geq 0$, and when $\kappa>0$ we can always fix the solution to have $\kappa=1$ by rescaling $t$. This defines the connection:
\begin{align} \label{asd}
A^\mathrm{asd} := \frac{1}{1+ t^2} \left( E_1 \otimes u^2 +  E_2 \otimes v^2 + E_3 \otimes w^2 \right) 
\end{align}  
over $\mathbb{R}^4$. The extra symmetry ensures, implicitly, that $A^\mathrm{asd}$ is $U(1)\times SU(2)$-invariant, so it can be uniquely lifted to the connection $\bar{A}^\mathrm{asd}= A^\mathrm{asd} - \left( \frac{1}{1+ t^2}  - 1\right) E_1 \otimes u^1$ as previously explained.  
\end{proof}     
With these preliminaries out of the way, we can now state the main theorem of this section: 
\begin{theorem} \label{candelasthm1instanton} Let $\delta( \epsilon) = \sqrt{2 \epsilon^{-1}}$. Then, as $\epsilon \rightarrow \infty$:  
\begin{enumerate}[\normalfont(i)]
\item $s^*_\delta R_\epsilon (t) \rightarrow \bar{A}^{\mathrm{asd}} (t)$.  
\item $R_\epsilon(t) \rightarrow R'_{0}(t)$ uniformly on compact subsets of $\left( 0,\infty \right)$. 
\end{enumerate}  
\end{theorem}
\begin{proof}[Proof of Theorem \ref{candelasthm1instanton} (i)] 
We start by rewriting the instanton equations in terms of the rescaling \eqref{adiabaticmetric}, and the lift from an invariant connection on the fibre as in the previous lemma. If we define $\left( a^{\delta}_0,a_2^\delta \right)(t)= \left( \frac{1 - a_0}{2}, a_2 \right) ( \delta t ) $, and consider some invariant connection $A$ defined by $\left( a_0,a_2\right)(t)$, as in Proposition \ref{principalconn}, then:
\begin{align*}
s^*_\delta A (t) = a_2^{\delta}(t) \left( E_2 \otimes v^2 + E_3 \otimes w^2 \right) + a_0^{\delta} (t) E_1 \otimes (u^2 -u^1 ) +  E_1 \otimes u^1  
\end{align*}
 Written in this way, the instanton equations \eqref{instA2} for $\left( a_0,a_2\right)(t)$ becomes the following one-parameter family of ODEs for $\left( a^{\delta}_0,a_2^\delta \right)(t)$: 
\begin{align} \label{instA3} 
\dot{a}^\delta_0 =   \frac{2 \lambda_\delta}{\left( u_1 + u_0 \right)_\delta} \left( \left( a_2^{\delta} \right)^2 -  a_0^\delta \right) + \frac{2 \delta^2 \lambda_\delta}{\left( u_1 - u_0 \right)_\delta } \left(( 1 - {a_0}^\delta ) \right)& &{\dot{a}^\delta}_2 = - \frac{3}{\lambda_\delta} \left ( 1 - {a_0}^\delta \right) a_2^\delta
 \end{align}
\eqref{instA3} has a one-parameter family of solutions for each $\delta > 0$ by rescaling the family of solutions $R_\epsilon$ to \eqref{instA2}, and we now show that one still obtains a one-parameter family as $\delta \rightarrow 0$. Considering the boundary conditions Proposition \ref{propconnboundary3} for extending $\left( a^{\delta}_0, a_2^\delta \right)$ to $t=0$, we can write $a^{\delta}_0= 1 - t^2 X_0$, $a^{\delta}_2= 1 - t^2 X_2$ for some smooth functions $ X_0, X_2$, so that, in a neighbourhood of $t=0$, \eqref{instA3} becomes the well-defined initial-value problem:   
 \begin{align}
 t \dot{X}_0 =   2 \left( X_0 -  X_2 \right) + O(t^2)&
& t \dot{X}_2 =   2 \left( X_0 -  X_2 \right) + O(t^2) 
 \end{align}
hence, once we fix the parameter $\kappa:= X_0 (0) = X_2(0)$, the continuous dependence of \eqref{instA3} on $\delta$ gives existence of a sufficiently small open neighbourhood of $t=0$ such that, for each $\delta \geq 0$, solutions to \eqref{instA3} are in a one-parameter family. These are determined by the power-series:
\begin{align}
a^{\delta}_0 = 1 - \kappa t^2 + O(t^4)& &a^{\delta}_2 =1 - \kappa t^2 + O(t^4) 
\end{align}  
Comparing with the power-series in Proposition  \ref{localresol1} for $R_\epsilon$, we see that the rescaled solutions $\left( \frac{1 - a_0}{2}, a_2 \right)_\epsilon ( \delta t )$ for $\delta>0$ have $\kappa = \delta^2 \epsilon /2$, and so the family of solutions $\left( a^{\delta}_0, a^{\delta}_2\right)$ exist for all time if $\kappa \geq 0$, $\delta>0$ by Theorem \ref{candelasthm1}. 

By continuity, the solutions must also exist for all time for $\kappa \geq 0$ as $\delta \rightarrow 0$, but as the resulting equations:  
\begin{align} 
t \dot{a}^0_0 =   2 \left( \left( a_2^{0} \right)^2 -  a_0^0 \right)& &t {\dot{a}^0}_2 = 2 \left ( {a_0}^0 - 1 \right) a_2^0
\end{align}
have the explicit solutions $a^{0}_0 = a^{0}_2 =  \left(1 +  \kappa t^2\right)^{-1}$ as in \eqref{ASD3}, this is already guaranteed.  

We now set $\delta \left(\epsilon \right) = \sqrt{2 \kappa \epsilon^{-1}}$ for some given $\kappa>0$, which we can always fix to be $1$ by a further rescaling. By rescaling the family $R_\epsilon = \left(  a_0, a_2 \right)_\epsilon$ to the instanton equations \eqref{instA2}, we get a solution $ \left(  a_0^\delta, a_2^\delta \right)(t) = \left( \frac{1 - a_0}{2}, a_2 \right)_\epsilon ( \delta t )$ to \eqref{instA3}. As we have just shown, the solution has $a_0^\delta (t)$, $a_2^\delta(t) \rightarrow \left(1 + t^2\right)^{-1}$ as $\delta \rightarrow 0$, hence $s^*_{\delta} R_\epsilon (t) \rightarrow \bar{A}^\mathrm{asd} ( t )$.  
\end{proof}
\begin{remark} Since $\left( 1 +  \kappa t^2 \right)^{-1}$ blows up in finite time if $\kappa < 0$,  by keeping the freedom to vary $\kappa$, this rescaling argument can be used to show that local solutions of Theorem \ref{candelasthm1} with $\epsilon < 0$ are not only unbounded but blow up in finite time.
\end{remark}      
\begin{proof}[Proof of Theorem \ref{candelasthm1instanton} (ii)] Let  $\left(a_0, a_2\right)_\epsilon(t) = \left(a_0^\epsilon, a_2^\epsilon\right)(t)$ denote the one-parameter family of solutions to \eqref{instA2} corresponding to $R_\epsilon$. Using the local power-series in Proposition \ref{localresol1}, it follows that $a_0^\epsilon (t^*)$, $a_2^\epsilon (t^*)$ are both strictly monotonic (increasing and decreasing, respectively) functions of $\epsilon>0$ for any fixed $t^*>0$, since we can take some $0<t^{**} \leq t^*$ sufficiently small to compare the power-series for $\left(a_0, a_2\right)_\epsilon$ and $\left(a_0, a_2\right)_{\epsilon'}$ at $t^{**}$ for any pair $\epsilon, \epsilon'$ by their first few terms, and then use the forward-comparison lemma to recover monotonicity at $t^*$. Since $\left(a_0, a_2\right)_\epsilon(t)$ lies in the bounded set $\mathcal{R}_0$ for all $t>0$, $\epsilon>0$, this implies $\left(a_0, a_2\right)_\epsilon(t)$ converges point-wise on $\left( 0,\infty \right)$ as $\epsilon \rightarrow \infty$. 

Since $a_2^\epsilon(t)$ is strictly decreasing in $t$ for all $t>0$, $\epsilon>0$, if we assume that the pointwise-limit $\inf_{\epsilon} a_2^\epsilon (t^*)$ for any fixed $t^*>0$ is non-zero, then we can use the inequality $a_2^\epsilon (t) \geq a_2^\epsilon (t^*) \geq \inf_{\epsilon} a_2^\epsilon (t^*) \geq 0$ for all $\epsilon>0$, $t \leq t^*$ to uniformly bound $a_2^\epsilon$ away from zero on $\left( 0 , t^* \right)$ and derive a contradiction with the first part of the main theorem. Explicitly, by part (i), for any $\varepsilon>0$, $T > 0$, $\exists \epsilon(\varepsilon, T)$ such that $\forall \epsilon \geq \epsilon(\varepsilon, T)$, the rescaled solution $\left(a_0, a_2\right)_\epsilon (\sqrt{2\epsilon^{-1}} T )$ satisfies $| a_2^\epsilon (\sqrt{2\epsilon^{-1}} T ) - (1+T^2)^{-1} | < \varepsilon$. By the assumption $L:=\inf_{\epsilon} a_2^\epsilon (t^*)>0$, we can pick $\varepsilon$, $T$ such that $0<\varepsilon< L - (1+T^2)^{-1}$, and then apply our inequality to any $\epsilon \geq \max \lbrace \epsilon(\varepsilon,T), 2\left( \tfrac{T}{t^*} \right)^2 \rbrace $:
\begin{align*}
 \varepsilon < L - (1+T^2)^{-1} \leq | a_2^\epsilon (\sqrt{2\epsilon^{-1}} T )| - |(1+T^2)^{-1}|  \leq | a_2^\epsilon (\sqrt{2\epsilon^{-1}} T ) - (1+T^2)^{-1} |
\end{align*}
since $\sqrt{2\epsilon^{-1}} T \leq t^*$. However, this demonstrates the existence of an $\epsilon \geq \epsilon(\varepsilon,T)$ such that the inequality $| a_2^\epsilon (\sqrt{2\epsilon^{-1}} T ) - (1+T^2)^{-1} | < \varepsilon $ fails, and hence we have a contradiction.
 
This previous discussion implies that $a_2^\epsilon$ converges uniformly to zero on any compact interval contained in $\left(0, \infty \right)$, and by using \eqref{differentialformula} to express the derivative $\dot{a}_0^\epsilon$ purely in terms of $a_2^\epsilon$ (up to some fixed functions of $t$), we also get the uniform convergence of $a_0^\epsilon$ on this interval. Now consider the initial value problem defined by the ODE \eqref{instA2} with $\left( a_0, a_2 \right)(t^*) = \left( a_0, a_2 \right)_\epsilon(t^*)$ for some fixed initial time $t^*>0$: this has unique solution $\left( a_0, a_2 \right)_\epsilon(t)$ on $\left(0, \infty \right)$, and continuous dependence on initial conditions guarantees that the limit as $\epsilon \rightarrow \infty$ is the unique solution to \eqref{instA2} with $a_2 =0$, and $a_0 (t^*)=\sup_{\epsilon} a_0^\epsilon (t^*)$. Since this solution must be contained in the closure of $\mathcal{R}_0$, by Remark \ref{remarkabelian}, this must be identified with $R'_0$ on $\left(0, \infty \right)$, the unique solution to \eqref{instA2} bounded on $\left(0, \infty \right)$ with $a_2 =0$.      
\end{proof}
\subsubsection*{Bubbling on $\mathcal{O}(-2,-2)$}
In the rest of this section, we will discuss bubbling phenomena for instantons on $\mathcal{O}(-2,-2)$: we will see that the one-parameter families of Calabi-Yau instantons $Q^l_{\alpha_l}$ on $\mathcal{O}(-2,-2)$ can be understood, in an appropriate adiabatic limit, as anti-self-dual instantons for the Eguchi-Hanson metric on the cotangent bundle of $ \mathbb{CP}^1$, fibred along a co-dimension four calibrated sub-manifold $\mathbb{CP}^1 \subset \mathcal{O}(-2,-2)$. Although this discussion is unnecessary for understanding the main results of this article, we include a brief sketch of the details, omitting any explicit proofs.
  
First, recall that the Eguchi-Hanson metric is $SU(2) \times U(1)$-invariant, and can be written on $T^* \mathbb{CP}^1 \setminus \mathbb{CP}^1 \cong  \mathbb{R}_{>0} \times SU(2) / \, \mathbb{Z}_2$, up to scale, as:
\begin{align} \label{eh}
g_{\mathrm{eh}} = dt^2 + \varphi^2 \left( (1 - \varphi^{-4}) u^2 + v^2 + w^2 \right) 
\end{align}    
where $t \geq 0$ is the radial geodesic arc-length extending over $\mathbb{CP}^1$ at $t=0$, $u, v, w $ are the basis of left-invariant one-forms dual to $E_1$, $E_2$, $E_3$ on $SU(2)$, and $\varphi(t)$ is the unique solution to $\dot{\varphi}^2 = 1 - \varphi^{-4}$ on $\left[ 0 , \infty \right)$ with $\varphi(0) = 1$. 

One can show that, up to gauge, $SU(2)$-invariant anti-self-dual connections for the Eguchi-Hanson metric are in $l \in \mathbb{Z}_{>0}$ one-parameter families $A^{\mathrm{eh}}_{\kappa,l}$, $0 \leq \kappa \leq 1 $, and these are additionally $U(1)$-invariant. Written in the temporal gauge:
\begin{align} \label{connectioneh0}
A^{\mathrm{eh}}_{\kappa,l} := \alpha_0 E_1 \otimes u +  \alpha_2 \left( E_2 \otimes v +  E_3  \otimes w \right)
\end{align}
with coefficients $\alpha_0(t), \alpha_2(t)$, explicit up to a smooth change of variable $t\mapsto \varphi$:
\begin{align} \label{connectioneh}
\alpha_0 =  \frac{l}{ \varphi^2} \left( 1 + \kappa \left( \frac{\varphi^2 -1 }{ \varphi^2 +1 } \right)^l \right)\left( 1 - \kappa \left( \frac{\varphi^2 -1 }{ \varphi^2 +1 } \right)^l  \right)^{-1}& & \alpha_2 = \frac{2l}{ 1 - \kappa \left( \frac{\varphi^2 -1 }{ \varphi^2 +1 } \right)^l}  \sqrt{\frac{\kappa \left( \frac{\varphi^2 -1 }{ \varphi^2 +1 } \right)^l}{\varphi^4 -1 }}
\end{align}  
By considering the orbit of the one-parameter family $A^{\mathrm{eh}}_{\kappa,l}$, $0 \leq \kappa < 1$ under constant ($t$-independent) invariant gauge transformations, these invariant families can be understood in terms of the moduli-spaces of anti-self-dual connections for the Eguchi-Hanson metric constructed by Nakajima in \cite{nakajima}: he shows, under certain assumptions, that connected components of the moduli-space are themselves a copy of Eguchi-Hanson. In our invariant set-up, the parameter $\kappa$ corresponds to (a parametrisation of) the radial parameter $t$, and these gauge orbits correspond to the orbits of the co-homogeneity one action of $SU(2)$ on $T^* \mathbb{CP}^1$.  
 
Observe that the moduli-space of invariant anti-self-dual connections $A^{\mathrm{eh}}_{\kappa ,l}$ behaves similarly to the moduli-space of invariant Calabi-Yau instantons for the metrics on $\mathcal{O}(-2,-2)$ in Theorems \ref{pandozayasthm1}, \ref{pandozayasthm2}: $\kappa=0$ is the unique abelian solution for each $l$, and at $ \kappa = 1$ there is a transition in the asymptotic behaviour of $A^{\mathrm{eh}}_{\kappa ,l}$ from the canonical connection $\left( \alpha_0, \alpha_2 \right) = \left(0,0\right)$ to the flat connection $\left( \alpha_0, \alpha_2 \right) = \left(1,1\right)$. Any solutions with $\kappa$ greater than this critical value become unbounded (at least here, in finite time), and only when $l=1$ are these critical solutions $A^{\mathrm{eh}}_{1 ,l}$ themselves flat.

If we pick overall scale for the family of metrics on $\mathcal{O}(-2,-2)$ such that one of the copies of $\mathbb{CP}^1 \subset \mathcal{O}(-2,-2)$ has fixed volume $U_1 - U_0  = 1$, and consider the rescaling $s_\delta$ as in \eqref{adiabaticmetric}, \eqref{hyporescaled} with $\delta^2 = \tfrac{2}{3}\left(U_1 + U_0 \right)$, then one can show that by varying in this one-parameter family of metrics, $\tfrac{4}{3} \lambda_\delta \rightarrow \varphi \sqrt{ 1- \varphi^{-4}}$, $\tfrac{2}{3} \left( u_1 + u_0 \right)_\delta \rightarrow \varphi^2$, and $ \left( u_1 - u_0 \right)_\delta \rightarrow 1$ in the adiabatic limit as $\delta \rightarrow 0$. In other words, near this adiabatic limit, the metric $s^*_\delta g$ is approximated by a fibration $\delta^2 g_\mathrm{eh} + g_{B}$ of a rescaled Eguchi-Hanson metric over a round metric on the base $\mathbb{CP}^1$, for some $\delta$ sufficiently small.

 With this established, the similarity between the moduli-spaces of instantons on $\mathcal{O}(-2,-2)$ and anti-self-dual connections on Eguchi-Hanson is somewhat less mysterious: if one considers the one-parameter family of instantons $Q^l_{\alpha_l}$ for $\alpha_l = l \sqrt{\kappa} \delta^{1-l}$ for some fixed, $l>0$, $\kappa\geq 0$, and pulls back via this rescaling, it can furthermore be shown that $s^*_\delta Q^l_{\alpha_l}(t) \rightarrow A_\kappa^{\mathrm{eh}}(t)$ as $\delta \rightarrow 0$ along the fixed-volume copy of $\mathbb{CP}^1 \subset \mathcal{O}(-2,-2)$.  

\begin{remark} Since the explicit solution \eqref{connectioneh} blows up in finite time if $\kappa>1$, up to exchanging the copies of $\mathbb{CP}^1 \subset \mathcal{O}(-2,-2)$, one can use this rescaling argument to show that the instantons $Q^l_{\alpha_l}$ for metrics on $\mathcal{O}(-2,-2)$ with $U_1 \pm U_0$ sufficiently close to $0$ must also blow up in finite time if $\alpha_l$ is sufficiently large. 
\end{remark}
\subsection{Solutions to the Monopole Equations} \label{odessection2}

In this section, we analyse the qualitative behaviour of solutions to the monopole equations  \eqref{dynamicODE} with non-zero Higgs field $\Phi$ away from the singular orbit. Assuming that the connection is not an instanton, we first show that there are no solutions for $\mathcal{O}(-2,-2)$, or $\mathcal{O}(-1) \oplus \mathcal{O}(-1)$ with quadratic curvature decay:   
\begin{prop} \label{thmmonopoleglobalA} There are no irreducible invariant monopoles on $\mathcal{O}(-2,-2)$ or $\mathcal{O}(-1) \oplus \mathcal{O}(-1)$ with quadratic curvature decay.     
\end{prop}
\begin{proof}
We look at solutions to the monopole equations \eqref{dynamicODE} with $v_0 = 0, v_3 = \mu$, i.e. for a hypo-structure of type $\mathcal{I}$ given by \eqref{hypoA}: 
 \begin{gather}  \label{dynamicODEmonopoleA}
 \begin{aligned} 
\dot{a_0} &= \frac{4 \lambda}{ \mu^2} \left( ( a_1^2 +  a_2^2 - 1 )u_0  - ( a_0 - a_1^2 +  a_2^2 ) u_1 \right)  \\
\dot{\phi} &= -\frac{6}{\mu} a_1 a_2   \\
\dot{a}_1 &= \frac{3}{2 \lambda  }  (a_0 -1 ) a_1 - 2 \frac{u_1-u_0}{\mu} a_2 \phi  \\
 \dot{a}_2 &= -\frac{3}{2 \lambda } (a_0+1) a_2 - 2 \frac{u_1 + u_0}{ \mu } a_1 \phi 
\end{aligned} 
\end{gather}
Recall from Lemma \ref{quadraticdecay} that a weak condition for solutions $\left( a_0, a_1, a_2, \phi \right)$ was to assume at least boundedness of $a_0, a_1, a_2$. We will show that there are no such solutions to \eqref{dynamicODEmonopoleA} existing for all time: in particular, we show that if a solution exists for all time with $a_0$ bounded and $\phi$ non-zero, then $a_1 - a_2$ must have at least exponential growth at infinity, provided certain initial conditions are satisfied. These initial conditions will be satisfied by the local solutions extending to $t=0$ obtained in the previous propositions, up to certain easily-verified symmetries:   
\begin{lemma} The following involutions are symmetries of \eqref{dynamicODEmonopoleA}: 
\begin{align} \label{symmetrygaugemonopole}
\left( a_0, a_1, a_2, \phi \right) \mapsto \left( a_0,  -a_1, -a_2, \phi \right)
\end{align}
\begin{align} \label{symmetrymonopole}
\left( a_0, a_1, a_2, \phi \right) \mapsto \left( a_0,  -a_1, a_2 , - \phi \right)
\end{align}
\end{lemma} 
We now find the set in which our solutions remain for all time:
\begin{lemma} The set $\mathcal{R}^+_\infty : = \lbrace \left(a_0, a_1, a_2,  \phi \right) \in \mathbb{R}^4 \mid a_1 > 0 > a_2, \phi > 0 \rbrace $ is forward-invariant under \eqref{dynamicODEmonopoleA}. 
\end{lemma}
\begin{proof}
We let $t$ be the first time a solution  $\left(a_0, a_1, a_2, \phi \right)$ leaves $\mathcal{R}^+_\infty$. However, none of the possibilities for a solution to leave $\mathcal{R}^+_\infty$ can hold at $t$: 
\begin{enumerate}[(i)]
\item $a_1 =0, a_2<0, \phi>0$, since it implies $\dot{a}_1>0$. 
\item $a_2 =0, a_1>0, \phi>0$, since it implies $\dot{a}_2<0$.
\item $a_1 = a_2 = 0 $, $\phi \geq 0$ coincides with the reducible solution $a_1 \equiv a_2 \equiv 0$, which by uniqueness of solutions implies this solution coincides with the reducible solution for all time. 
\item Since $\phi(t_0)>0$ for some $t_0<t$, if we assume $\phi(t)=0$, then by the mean value theorem $\dot{\phi}(t_1)<0$ for some $t_0<t_1<t$ which implies $a_1 (t_1) a_2 (t_1) > 0$. 
\end{enumerate} 
\end{proof}
\begin{remark} \label{remarkmaximumprinciple} Although we have shown explicitly from the ODEs that solutions preserve $|\phi|>0$, it also follows more generally, since the function $\vert \Phi \vert^2: M \rightarrow \mathbb{R}$ is sub-harmonic for any monopole $\left(A, \Phi\right)$ over an arbitrary Calabi-Yau 3-fold $M$. In the invariant co-homogeneity one setting, this implies that if $\vert \Phi (t) \vert^2 = 0$ for some $t>0$, then it vanishes for all $t$ by the maximum principle.  
\end{remark}
We also prove that solutions lying in this set, if they exist for all time, are (exponentially) unbounded as $t \rightarrow \infty$: 
\begin{lemma} If $\left(a_0, a_1, a_2, \phi\right)$ be a solution to \eqref{dynamicODEmonopoleA} existing for all time $t\geq t^*$ with $a_0$ bounded, lying in $\mathcal{R}^+_\infty$ at initial time $t^*$, then $a_1(t)- a_2 (t)$ is unbounded as $t\rightarrow \infty$.     
\end{lemma}
\begin{proof} Since the CY-structure is AC, $\lambda \sim t$ as $t\rightarrow \infty$, hence $\frac{3}{2 \lambda  }  (a_0 \pm 1 ) \rightarrow 0$ as $t \rightarrow \infty$ by assumption. We also have that $\frac{u_1 \pm u_0}{\mu } \rightarrow 1$ in the same limit. Since $a_1>0> a_2 $, then for every $\epsilon>0$, $\exists T^* \gg 0 $ such that $\forall t> T^*$ the following inequalities hold: 
\begin{gather} 
 \begin{aligned} 
& \frac{3}{2 \lambda  }  (a_0 -1 ) a_1 > - \epsilon a_1& &
- \frac{3}{2 \lambda  }  (a_0 + 1 ) a_2 <  - \epsilon a_2 &   \\
&- \frac{u_1-u_0}{\mu} a_2  >  -(1-\epsilon) a_2& & 
 - \frac{u_1-u_0}{\mu} a_1 < -(1-\epsilon) a_1 &  
\end{aligned}
\end{gather}
Let $T:= \max \lbrace t^*, T^* \rbrace$ for our fixed initial time $t^*$, and $\bar{\phi} := \phi(t^*) >0$. Since $\phi$ is strictly increasing in $\mathcal{R}^+_\infty$, we have $\phi(t)> \bar{\phi}$ for $t>t^*$. Putting all our inequalities together on $t>T$, we obtain the following: 
 \begin{align*} 
\dot{a}_1 -\dot{a}_2 > ( 2 \bar{\phi} - \epsilon ( 2 \bar{\phi}+1 )) ( a_1 - a_2)    
\end{align*}
and if we choose $\epsilon < \dfrac{\bar{\phi}}{ (2\bar{\phi}+1)}$, then by integrating:  
\begin{align*} 
a_1(t) -a_2(t) \geq ( a_1(T) - a_2 (T)) \exp ( (t -T) \bar{\phi} )  
\end{align*} 
\end{proof} 
This completes the proof the proposition, since in all cases, using the symmetries \eqref{symmetrygaugemonopole}, \eqref{symmetrymonopole}, for the power series solutions near the singular orbit, one can reduce to the case of the monopole lying in $\mathcal{R}^+_\infty$ for some small initial time:
\begin{enumerate}
\item For local solutions $\left(R', \Psi' \right)_{\epsilon', \delta'}$ defined by Proposition \ref{localresol6}, since $\epsilon',\delta' \neq 0$ by assumption i.e. we do not have an instanton, then up to symmetry one can assume $\epsilon',\delta'>0$. Hence $\left(a_0, a_1, a_2, \phi \right)_{\epsilon', \delta'}$ lies in $\mathcal{R}^+_\infty$. 
\item For local solutions $\left(Q^l, \Theta^l \right)_{\alpha_l, \beta_l}$ defined by Proposition \ref{localresol4}, since $\alpha_l,\delta_l \neq 0$ by assumption, up to symmetry one can assume $\alpha_l<0,\beta_l>0$. Hence $\left(a_0, a_1, a_2, \phi \right)_{\alpha_l, \beta_l}$ lies in $\mathcal{R}^+_\infty$. This also covers the case $l \leq 0$, by exchanging the factors of $SU(2)$ on the principal orbits, and considering the Calabi-Yau structure on $\mathcal{O}(-2,-2)$ pulled-back via this diffeomorphism. 
\item  For local solutions $\left(R, \Psi \right)_{\epsilon, \delta}$ defined by Proposition \ref{localresol1}, since $\delta \neq 0$ by assumption, then up to \eqref{symmetrymonopole}, one can also assume $\delta>0$. While the image of a solution under \eqref{symmetrygaugemonopole} may not extend to the singular orbit, existence of a bounded solution extending to the singular orbit would imply existence of a bounded solution in $\mathcal{R}^+_\infty$ under the symmetry.   
\end{enumerate}  
\end{proof}
The existence of invariant monopoles on $T^*S^3$ was shown in \cite{Oliveira2015}: by restricting the monopole equations to the $\mathbb{R}^3$-fibre over a point in $S^3$, i.e. solving \eqref{dynamicODE} with $a_0 \equiv 0$,  $a_1 \equiv a_2$, Oliveira constructed a one-parameter family of invariant monopoles for $T^*S^3$, first by considering the local solutions $\left(S, \Phi \right)_{\xi, \chi}$ with $\xi=0$ of this system, and then applying PDE methods for invariant monopoles in $\mathbb{R}^3$. Due to a computational error in \cite[Lemma 6, Appendix A]{Oliveira2015}, Oliveira did not consider local solutions $\left(S, \Phi \right)_{\xi, \chi}$ with $\xi$ non-zero, but we fix the resulting gap in the proof of the main theorem \cite[Theorem 1]{Oliveira2015} by imposing quadratic curvature decay:
\begin{prop}\label{thmmonopoleglobalB} Invariant monopoles with quadratic curvature decay on $P_\mathrm{Id} \rightarrow T^* S^3$ are in a one-parameter family $\left(S, \Phi \right)_{\chi}:=\left(S, \Phi \right)_{0, \chi}$, $\chi>0$, up to gauge. Moreover: 
\begin{enumerate}[\normalfont(i)]
\item $\lim_{t \to \infty} \left(S, \Phi \right)_{\chi} (t) = \left( A^\mathrm{can}, \Phi_\chi \right)$, where $\Phi_\chi$ is a constant non-trivial Higgs field.
\item $\left(S, \Phi \right)_{0,0} = S_0$ where $S_0$ is the instanton of Theorem \ref{stenzelthm}, with a trivial Higgs field. 
\end{enumerate} 
\end{prop}
\begin{proof} 
We rewrite the monopole equations \eqref{dynamicODE} with $a_\pm := a_1 \pm a_2$: 
\begin{gather} \label{dynamicODEmonopoleB} 
 \begin{aligned} 
\dot{a_0} &= \frac{4 \lambda}{ \mu} \left( a_+ a_- - a_0  \right)  \\
\dot{a}_+ &= \frac{3 (v_3 + v_0)}{2 \lambda \mu}(  a_0 a_- - a_+ ) - 2 a_+ \phi \\ 
\dot{a}_- &= \frac{3 (v_3 - v_0)}{2 \lambda \mu} ( a_0 a_+ - a_-) + 2 a_- \phi  \\
\dot{\phi} &= \frac{3}{\mu^2} \left( \left( \frac{1}{2} (a_+^2+a_-^2) -1\right) v_0 - \frac{1}{2} ( a_+^2 - a_-^2 )v_3 \right) 
\end{aligned} 
\end{gather}
for $\mu, \lambda, v_3, v_0$ explicit solutions to the hypo-equations given in \eqref{hypoBsol}, and recall from Lemma \ref{quadraticdecay}, that we are interested in solutions with $a_0, a_+, a_-, t\phi a_+, t\phi a_-$ bounded. There are three parts to the proof: 
\begin{enumerate}
\item Solutions to \eqref{dynamicODEmonopoleB} extending over the singular orbit with $a_0, a_- \not \equiv 0$ are unbounded.  
\item Solutions to \eqref{dynamicODEmonopoleB} with $a_0, a_- \equiv 0$, which have local power-series $\left(S, \Phi \right)_{\xi, \chi}$ given in Proposition \ref{localmonopolesmoothing} for $\xi =0$, are bounded iff $\chi \geq 0$. 
\item In this case, solutions with $\chi>0$ have $a_+ \rightarrow 0$, $\phi \rightarrow \phi_\chi$ as $t\rightarrow \infty$ for some constant $\phi_\chi>0$, and $t a_+ \phi$ bounded. The solution with $\xi, \chi=0$ is the explicit instanton \eqref{explicit} found in \cite{Oliveira2015}.  
\end{enumerate}
To prove the first part, we will recall the symmetries \eqref{symmetrygauge} and \eqref{symmetrymetric} of the problem:
 \begin{lemma} The following involutions are symmetries of \eqref{dynamicODEmonopoleB}: 
\begin{align} \label{symmetrygaugemonopoleB}
\left( a_0, a_+, a_- , \phi \right) \mapsto \left( a_0,  -a_+, -a_-, \phi \right)
\end{align}
\begin{align} \label{symmetrymonopoleB}
\left( a_0, a_+, a_-, \phi \right) \mapsto \left( -a_0, a_+, - a_-, \phi \right)
\end{align}
\end{lemma}
We can also prove a strict monotonicity condition for $\phi$:
\begin{lemma}[Monotonicity] \label{maximumlemma} A solution $\left( a_0, a_+, a_- , \phi \right)$ to \eqref{dynamicODEmonopoleB} with $\pm \phi(t^*) > 0, \pm \dot{\phi}(t^*)>0$ at some initial time $t^*>0$, has $\pm \phi(t) > 0, \pm \dot{\phi}(t)>0$ for all $t\geq t^*$. 
\end{lemma} 
\begin{proof}
We calculate:
\begin{align*}
\left.\ddot{\phi} \right|_{\dot{\phi}=0}=\frac{6}{\mu^2}\left( a_+^2 \left( v_3 - v_0 \right) + a_-^2 \left( v_3 + v_0 \right) \right) \phi  
\end{align*}
In particular, since $\dot{\phi}\neq 0$ for $a_+ = a_- = 0$, and $\left( v_3 \pm v_0 \right)>0$ for $t \neq 0$, we have $\ddot{\phi}> 0 $ at $\dot{\phi} = 0$  iff $ \phi > 0$. Hence any critical point for $\phi$ must be minimum, and since $\dot{\phi}(t^*)>0$ , we must have $\dot{\phi}(t)>0, \phi(t)>0$ for all $t>t^*$. The proof for $ \phi(t) < 0,  \dot{\phi}(t)<0$ is similar. 
\end{proof}
Using this, we find a set that contains our solutions for all time: 
\begin{lemma} A solution $\left( a_0, a_+, a_- , \phi \right)$ to \eqref{dynamicODEmonopoleB} lying in $\mathcal{S}^{\pm}_\infty : = \lbrace \left( a_0, a_+, a_- , \phi \right) \in \mathbb{R}^4 \mid a_0 > 0, a_+ > 0, a_->0, \pm \phi > 0\rbrace $ at some initial time $t^*$ with $ \pm \dot{\phi}(t^*)>0$, remains there for all forward time $t\geq t^*$. 
\end{lemma}
\begin{proof}
Let $t$ be the first time a solution $\left( a_0, a_+, a_- , \phi \right)$ leaves $\mathcal{S}^{\pm}_\infty $. However, none of the possibilities for a solution to leave $\mathcal{S}^{\pm}_\infty$ can hold at $t$: 
\begin{enumerate}[(i)]
\item $a_{0} =0, a_->0, a_+>0$, since it implies $\dot{a}_0>0$. The same is true if we permute indices $0,+,-$.  
\item $a_+ =0, a_-= 0,a_0>0$, since then $\dot{a}_+ = \dot{a}_- = 0 $. By local uniqueness and existence for ODEs, from \eqref{dynamicODEmonopoleB}, one sees that the solution must have  $a_+ \equiv 0, a_- \equiv 0$, at least for some small interval $\left( t- \epsilon, t+\epsilon\right) $. Again one obtains similar results by permuting indices.   
\item $a_0 = a_+ = a_- = 0 $ coincides with solution $\left( 0 , 0, 0 , -3 I \right)$, where $\dot{I} = \frac{v_0}{\mu^2} $, which is a solution to \eqref{dynamicODEmonopoleB} for any choice of initial condition $\phi(t)$ for $t>0$.  
\item $\phi=0$ is impossible by monotonicity. 
\end{enumerate} 
\end{proof} 
We can now use monotonicity to bound $\phi$ away from zero, which will show that solutions in $\mathcal{S}^{\pm}_\infty$ must be unbounded as $t\rightarrow \infty$:  
\begin{lemma} \label{propblowup} A solution $\left( a_0, a_+, a_- , \phi \right)$ to \eqref{dynamicODEmonopoleB} lying in $\mathcal{S}^{\pm}_\infty$ with $\pm \dot{\phi}>0$ at some initial time $t^*>0$ cannot have $a_\mp$ bounded for all forward-time $t\geq t^*$.
\end{lemma}
\begin{proof}
We start with the case $\mathcal{S}^{+}_\infty$. Since $a_+ a_0 > 0$, we have the following inequality: 
 \begin{align*} 
\dot{a}_- > \left( 2\phi - \frac{3 (v_3 - v_0)}{2 \lambda \mu} \right)  a_-   
\end{align*}
Since $\frac{3 (v_3 - v_0)}{2 \lambda \mu} \rightarrow 0$ and $\phi$ strictly increasing, then for fixed $t^*$, $\exists T>t^*$ such that $\forall t> T$: 
 \begin{align*} 
\bar{\phi}:= \phi (t^*) > \frac{3 (v_3 - v_0)}{2 \lambda \mu} (t)
\end{align*}
Then, since $a_->0$, integrating the inequality for  $\dot{a}_-$ gives: 
 \begin{align*} 
a_- (t)  \geq a_- (T) \exp ((t-T)  \bar \phi )   
\end{align*}
The proof for $\mathcal{S}^{-}_\infty$ is almost identical, since now: 
 \begin{align} \label{eqnblowup}
\dot{a}_+ > \left( - 2\phi - \frac{3 (v_3 + v_0)}{2 \lambda \mu} \right)  a_+   
\end{align} 
with $\phi<0$ monotonically decreasing and $\frac{3 (v_3 + v_0)}{2 \lambda \mu} \rightarrow 0$.  
\end{proof}
To complete the proof of the first part of the theorem, one only need apply this lemma to the power-series solution $\left(S, \Phi \right)_{\xi, \chi}$ of Proposition \ref{localmonopolesmoothing}. Up to symmetry, we can take $\chi,\xi>0$, so for some $0<t^*$ sufficiently small, the solution $\left( a_0, a_+, a_- , \phi \right)_{\xi,\chi}(t^*) $ lies in $\mathcal{S}^{+}_\infty$ with $\dot{\phi}(t^*)>0$, and hence we obtain that these solutions are unbounded.   

To prove the second and third parts of the theorem (cf. \cite[Theorem 1]{Oliveira2015}), we need to prove local solutions $\left(S, \Phi \right)_{0, \chi}$  with  $\xi=0$, i.e. solutions to the ODE:
\begin{align} \label{dynamicODEmonopoleB2}   
\dot{a}_+ = -   a_+ \left( \frac{3 (v_3 + v_0)}{2 \lambda \mu} + 2\phi \right)& &\dot{\phi} = - \frac{3}{\mu^2} \left(\frac{1}{2} a_+^2 (v_3 - v_0) + v_0 \right)
\end{align}
have fixed asymptotics $a_+ \rightarrow 0$, $\phi \rightarrow \phi_\chi >0$ only in the case $\chi>0$,   and if $\chi<0$ are these solutions are unbounded as $t\rightarrow \infty$. By uniqueness, the local solution with $\chi = 0$, $\xi=0$ is the instanton \eqref{explicit} with $\phi \equiv 0$.
  
We first note that the sign of $a_+$ is preserved by \eqref{dynamicODEmonopoleB2}, hence by using the gauge symmetry \eqref{symmetrygaugemonopoleB} we can always reduce to the case $a_+ >0$ in the following. Assuming this, we can prove the existence of a set in which solutions become unbounded:   
\begin{lemma} \label{blowuplemma} Solutions to \eqref{dynamicODEmonopoleB2} with $a_+>0,\phi<0, \dot{\phi}<0 $ at some initial time $t^*>0$, cannot have $a_+$ bounded for all forward-time $t \geq t^*$. 
\end{lemma}
\begin{proof} This proceeds almost identically to the proof of Proposition \ref{propblowup}, only now we have the inequality \eqref{eqnblowup} is an equality. Again we have $\phi<0$ monotonically decreasing by Lemma \ref{maximumlemma}, and integrating  the inequality for $\dot{a}_+$ in terms of $\phi(t^*)$, we have that there exists $T>t^*$, such that for all $t\geq T$: 
 \begin{align*} 
a_+ (t)  \geq a_+ (T) \exp (-(t-T) \phi(t^*) )   
\end{align*} 
\end{proof}
We also prove the existence of a set in which solutions are bounded for all time, and have the desired asymptotics: 
\begin{lemma} \label{fixedlemma} Solutions to \eqref{dynamicODEmonopoleB2} with $a_+>0,\phi>0, \dot{\phi}>0 $ at some initial time $t^*$, are bounded for all $t \geq t^*$, and have $\left( a_+, \phi \right) \rightarrow \left( 0 , \phi_\chi \right)$ as $t\rightarrow \infty$, for some constant $\phi_\chi>0$. Moreover, $t a_+ \phi$ is bounded for all $t \geq t^*$.  
\end{lemma}
\begin{proof} We already have lower bounds for $a_+$ and $\phi$. We now prove an upper bound for $\phi$: we have the inequality $\dot{\phi} < -\frac{3v_0}{\mu^2}$, and hence by integrating $\phi$ must be bounded above. Since $\phi$ is also strictly increasing, this implies the existence of a limit $\phi \rightarrow \phi_\chi>0$ as $t\rightarrow \infty$.

For $a_+$, since $ \frac{3 (v_3 + v_0)}{2 \lambda \mu}>0$, and $\phi>0$ strictly increasing, we have the inequality $\dot{a}_+ \leq - 2 a_+ \phi(t^*)$. Integrating this, we get:
 \begin{align*} 
0 < a_+ (t)  \leq a_+ (t^*) \exp (- 2(t-t^*) \phi(t^*) )   
\end{align*}
giving us the required asymptotics for $a_+$, $t \phi a_+ $.      
\end{proof}
The final two parts of the proof of the Theorem \ref{thmmonopoleglobalB} are now immediate, since local solution $\left(S, \Phi \right)_{0, \chi}$ to \eqref{dynamicODEmonopoleB2} given by Proposition \ref{localmonopolesmoothing} with $\xi=0$, $\chi<0$ satisfies the conditions of Lemma \ref{blowuplemma}, while for $\chi>0$ they satisfy the conditions of Lemma \ref{fixedlemma}. 
\end{proof}
\appendix

\section{Extending invariant bundle data to the singular orbit} \label{sectioncohobundles}

By considering $SU(2)^2$-invariant instantons and monopoles on the space of principal orbits, we obtained ordinary differential equations depending on geodesic parameter $t\in \mathbb{R}_{>0}$. In this appendix, we check the boundary conditions for these data to extend smoothly to the singular orbits at $t=0$. 
 
First, in a general setting for extending homogeneous bundles over the principal orbits of co-homogeneity one manifolds to the singular orbits, we let $M$ denote a co-homogeneity one manifold with group diagram $H \subset H' \subseteq K$, and $H'$-representation $V$, and let the structure group be denoted $G$. Any $K$-invariant $G$-bundle $P$ over $M$ must be of the form $P_{\lambda} = K \times_{H'} \left( V \times G \right)$ for some homomorphism $\lambda: H' \rightarrow G$, which we denote the \textit{singular} homomorphism. It is not difficult to see that $P_{\lambda}$ restricted to a principal orbit $K/H$ is a homogeneous bundle $K \times_{H} G$, where $H$ acts on $G$ via group homomorphism $\left. \lambda \right|_{H} $.
 
Focusing on the case of $SU(2)^2$-invariant $SU(2)$-bundles over the co-homogeneity one manifolds $T^* S^3$, $\mathcal{O}(-1) \oplus \mathcal{O}(-1)$, and $\mathcal{O} \left(-2,-2\right)$, let us describe and classify the possible extensions of homogeneous bundles in this way. Recall from \S \ref{invmonopole} that homogeneous bundles $P_n$ over the principal orbits are classified by the integer $n$, and an additional $j \in \mathbb{Z}_2$ in the case of homogeneous bundles $P_{n,j}$ over the principal orbit of $\mathcal{O} \left(-2,-2\right)$: 
\begin{prop} \label{bundles} Up to equivariant isomorphism, the $SU(2)^2$-invariant $SU(2)$-bundles extending $P_n$, $P_{n,j}$ to the singular orbit are given by:
\begin{enumerate}[\normalfont(i)]
\item Extending over $S^3 = SU(2)^2 / \triangle SU(2)$: $P_1$, $P_0$ extend to $P_\mathrm{Id}$, $P_\mathbf{0}$ defined by singular homomorphisms $\mathrm{Id}, \mathbf{0}$ respectively. 
\item Extending over $S^2 = SU(2)^2 / U(1) \times SU(2)$: $P_n$ extends to $P_{n, \mathbf{0}}$ defined by singular homomorphism $\iota^n \times \mathbf{0}$ for all $n$, and $P_1$ also extends to $P_{0, \mathrm{Id}}$ defined by singular homomorphism  $\iota^0 \times \mathrm{Id}$. 
\item Extending over $S^2 = SU(2)^2 / SU(2) \times U(1)$: $P_n$ extends to $P_{\mathbf{0},n}$ defined by singular homomorphism $\mathbf{0} \times \iota^n$ for all $n$, and $P_1$ also extends to $P_{\mathrm{Id},0}$ defined by singular homomorphism  $\mathrm{Id} \times \iota^0$.  
\item Extending over $S^2 \times S^2 = SU(2)^2 / U(1)^2$: $P_{n,j}$ extends to $P_{l,m}$ defined by singular homomorphism $\mathbb{\iota}^l\times \mathbb{\iota}^m$, where $l+m =n$, and either $j=m \, \mathrm{mod} 2$ or $j=l \, \mathrm{mod} 2$. 
\end{enumerate}
where $\mathrm{Id}$, $\mathbf{0}: SU(2) \rightarrow SU(2)$ denote the identity and the trivial homomorphism respectively, and $\mathbb{\iota}^n$ denotes the $n^\text{th}$-power of the diagonal embedding $\iota: U(1) \hookrightarrow SU(2)$. 
\end{prop}   
\begin{proof} The first two parts of the proposition follow directly from the previous discussion, and the group diagrams $\triangle U(1) \subset \triangle SU(2) \subset SU(2)^2$, and $\triangle U(1) \subset U(1) \times SU(2) \subset SU(2)^2$ for $T^* S^3$ and $\mathcal{O}(-1) \oplus \mathcal{O}(-1)$, respectively. The third part follows via exchanging the factors of $SU(2)^2$ for $\mathcal{O}(-1) \oplus \mathcal{O}(-1)$, i.e. writing the group diagram as $\triangle U(1) \subset SU(2) \times U(1) \subset SU(2)^2$.  

The group diagram for $\mathcal{O} \left(-2,-2\right)$ is given by $K_{2,-2} \subset U(1)^2 \subset SU(2)^2$, so the singular homomorphisms are classified by a pair of integers $(l,m)$:
\begin{gather}\label{homU1squared}
\begin{aligned} 
&(e^{i\theta_1}, e^{i\theta_2}) \mapsto \begin{pmatrix} 
e^{il\theta_1 + im \theta_2} & 0 \\
0 & e^{-il\theta_1 - im \theta_2}
\end{pmatrix}
\end{aligned} 
\end{gather}
where principal isotropy group $K_{2,-2}$ is uniquely defined as the kernel of \eqref{homU1squared} with $\left(l,m\right) = \left(2,-2\right)$.  One can realise the isomorphism $K_{2,-2} \cong \triangle U(1) \times \mathbb{Z}_2 \subset U(1)^2$ in exactly two ways, either with $\mathbb{Z}_2 \subset U(1)^2$ defined as the subgroup generated by $(e^{2i \pi}, e^{i\pi})$ or $\mathbb{Z}_2 \subset U(1)^2$ defined as the subgroup generated by $(e^{i \pi}, e^{2i\pi})$, equivalent up to the automorphism exchanging the factors of $U(1) \subset U(1)^2$. The first of these isomorphisms is given by $K_{2,-2} \ni (e^{i\theta_1}, e^{i\theta_2}) \mapsto (e^{i\theta_1}, e^{i\theta_1}).(e^{2i \pi}, e^{i (\theta_2 -\theta_1)}) \in  \triangle U(1)\times \mathbb{Z}_2$, and if we re-write \eqref{homU1squared} as: 
\begin{align*} 
(e^{i\theta_1}, e^{i\theta_2} ) \mapsto \begin{pmatrix} 
e^{i\theta_2-i\theta_1} & 0 \\
0 & e^{-i\theta_2 + i\theta_1} 
\end{pmatrix}^m \begin{pmatrix} 
e^{i\left(l+m \right)\theta_1} & 0 \\
0 & e^{-i\left(l+m \right)\theta_1} 
\end{pmatrix}  
\end{align*}
and fix the $\mathbb{Z}_2$-generator $(e^{2i \pi}, e^{i\pi})$, then \eqref{homU1squared} restricts to $\triangle U(1) \times \mathbb{Z}_2 \subset U(1)^2$ as the homomorphism \eqref{homomorphism} with $j = m \, \mathrm{mod} 2$ and $l+m =n$. By exchanging the factors of $U(1) \subset U(1)^2$, which also exchanges $\left(l,m\right)$ in \eqref{homU1squared}, we get the homomorphism \eqref{homomorphism} with $j = l \, \mathrm{mod} 2$ and $l+m =n$.
 \end{proof}
With a little extra work, the following proposition can also be seen from the previous discussion: 
\begin{proposition} \label{so3prop} Any $SU(2)^2$-invariant $SO(3)$-bundle over $\mathcal{O}(-1) \oplus \mathcal{O}(-1)$, $T^* S^3$ or $\mathcal{O}(-2,-2)$ admitting irreducible invariant connections has a lift to an $SU(2)^2$-invariant $SU(2)$-bundle.  
\end{proposition}
\begin{proof}  
By definition, an invariant $SO(3)$-bundle lifts if the singular homomorphism lifts. On the other hand, to admit irreducible invariant connections, the invariant $SO(3)$-bundle restricted to the space of principal orbits must lift to the invariant $SU(2)$-bundle $P_1$ i.e. if we denote the principal isotropy group by $H$, the homomorphism $H \rightarrow SO(3)$ lifts to the homomorphism $H \rightarrow SU(2)$ given by \eqref{homomorphism} with $n=1$. The statement for $T^* S^3$ and $\mathcal{O}\left(-1 \right) \oplus \mathcal{O}\left(-1 \right)$ is then immediate from Proposition  \ref{bundles}.
 
As for $\mathcal{O}\left(-2, -2 \right)$, the $SO(3)$-bundles are classified by the singular homomorphisms $U(1)^2 \rightarrow SO(2) \subset SO(3)$:
\begin{gather}
\begin{aligned} 
&(e^{i\theta_1}, e^{i\theta_2}) \mapsto \begin{pmatrix} 
1 & 0 & 0  \\
0 & \cos\left(l\theta_1 + m \theta_2 \right) & \sin\left(l\theta_1 + m \theta_2 \right)  \\
0 & - \sin\left(l\theta_1 + m \theta_2 \right) & \cos\left(l\theta_1 + m \theta_2 \right) 
\end{pmatrix}
\end{aligned} 
\end{gather}
which lift to the $SU(2)$-homomorphism \eqref{homU1squared} when $l,m$ are both even. By the assumption of irreducibility, we require $l+m =2$, so it suffices to consider the case where $l,m$ are also both odd. Restricted to $K_{2,-2} \subset U(1)^2$, this gives:   
\begin{gather} \label{so3action}
\begin{aligned} 
(e^{i\theta_1}, e^{i \theta_2}) \mapsto \begin{pmatrix} 
1 & 0 & 0  \\
0 & -1 & 0  \\
0 & 0 & -1 
\end{pmatrix} \begin{pmatrix} 
1 & 0 & 0  \\
0 & \cos\left(2 \theta_1 \right) & \sin\left(2 \theta_1 \right)  \\
0 & - \sin\left(2\theta_1 \right) & \cos\left(2 \theta_1 \right) 
\end{pmatrix}
\end{aligned} 
\end{gather}
up to the automorphism exchanging the factors of $U(1)\subset U(1)^2$. Recall from Remark \ref{remarkz2} that both $(e^{i\pi}, e^{2i\pi}), (e^{2i\pi}, e^{i\pi}) \in K_{2,-2}$ act trivially on the tangent space of the principal orbits of $\mathcal{O}\left(-2, -2 \right)$, but one of $(e^{i\pi}, e^{2i\pi}), (e^{2i\pi}, e^{i\pi})$ acts non-trivially on $\mathfrak{so}(3)$ by \eqref{so3action}, and so every invariant $\mathfrak{so}(3)$-valued connection one-form on the principal orbit can only take values in the set of fixed-points $\mathfrak{u}(1)\subset \mathfrak{so}(3)$ and must therefore be reducible. 
\end{proof}
 If there is a lift of the invariant $SO(3)$-bundle, then clearly any invariant connection or Higgs field can also be lifted. Consequentially, in the irreducible case, the invariant gauge theory can also be lifted. 
 
Returning now to the invariant $SU(2)$-bundles $P_{\lambda}$ in Proposition \ref{bundles}, we seek to describe the conditions for extending $SU(2)^2$-invariant connections on $P_{\lambda}$ and sections of the adjoint bundle $\Omega^0 \left( \mathrm{Ad} P_{\lambda}\right)$ over the singular orbit. For any $SU(2)^2$-invariant connection $A$, it will suffice to describe the conditions for extending $SU(2)^2$-invariant adjoint-valued one-forms $\Omega^1 (\text{Ad}P_{\lambda} )$, since we can use the canonical connection $ d \lambda$ of $P_{\lambda}$ over the singular orbit as an $SU(2)^2$-invariant reference connection to write $A- d \lambda \in \Omega^1 (\text{Ad}P_{\lambda} )$. 

Observe that an $SU(2)^2$-invariant section of $\text{Ad}P_{\lambda}$ can be identified with an $H'$-equivariant map $V \rightarrow\mathfrak{su}(2)$, where $\left(H', V \right)$ are the singular isotropy groups and their representations as above.  Similarly, an $SU(2)^2$-invariant element of $\Omega^1 (\text{Ad}P_{\lambda} )$ is determined by an $H'$-equivariant map $L: V \rightarrow \mathfrak{su}(2) \otimes \left( V^* \oplus \mathfrak{p}^* \right)$, where $\mathfrak{p}^*$ is the space of left-invariant one-forms on the singular orbit $SU(2)^2 / H'$, and $H'$, $\lambda \left( H' \right)$ act via the adjoint on $\mathfrak{p}^*$, $\mathfrak{su}(2)$ respectively. 

In order to calculate these extension conditions, we apply a similar analysis as in \cite[Lemma 1.1]{Eschenburg2000} applied to $SU(2)^2$-invariant adjoint-valued forms, c.f. \cite{goncalog2}. The general procedure is as follows: by evaluating at some non-zero $v_0 \in V$, $H'$-equivariant homogeneous polynomial maps $L: V \rightarrow \mathfrak{su}(2) \otimes \left( V^* \oplus \mathfrak{p}^* \right)$ (respectively $L: V \rightarrow \mathfrak{su}(2)$ for zero-forms) give splitting of the vector-space of $\triangle U(1)$-invariant adjoint-valued forms, indexed by the associated polynomial degree. Any $SU(2)^2$-invariant adjoint-valued form, away from $0 \in V$, can be written as the sum of $\triangle U(1)$-invariant adjoint-valued forms
 
We now summarise the extension conditions for $SU(2)^2$-invariant elements of $\Omega^0 (\text{Ad}P_{\lambda} )$, and $SU(2)^2$-invariant connections in the temporal gauge, as obtained in Proposition \ref{principalconn}, focusing on the case that the bundle $P_{\lambda}$ restricts to the homogeneous bundle $P_1$ over the principal orbit. We adopt the notation of writing invariant adjoint-valued one forms $I_1, J_1, I_2, J_2$ as in \eqref{tensors}:
\begin{align*} 
I_1 := E_2 \otimes v^1 + E_3 \otimes w^1& &J_1:= E_3 \otimes v^1 - E_2 \otimes w^1 \\
I_2:=E_2 \otimes v^2 + E_3 \otimes w^2& &J_2 := E_3 \otimes v^2 - E_2 \otimes w^2   
\end{align*}
\begin{prop} \label{propconnboundary} An invariant connection $A$ on $P_{l,m}$ with $l+m =1$, extends to the singular orbit $S^2 \times S^2 = SU(2)^2/ U(1)^2$ if and only if $A = a_1 I_1 + b_1 J_1 + a_2 I_2+ b_2 J_2 + a_0 E_1 \otimes u^- + (l+m)E_1 \otimes u^+$, with $a_0(0)= l-m$,  $a_0$ even, and the following cases: 
\begin{enumerate}[\normalfont(i)]
\item If $l \geq 1$, then $a_1, b_1$ must be of degree $l-1$ and $a_2, b_2$ of degree $l$
\item If $m \geq 1$, then $a_1, b_1$ must be of degree $m$ and $a_2, b_2$ of degree $m-1$.
\end{enumerate}
\end{prop}
\begin{prop} \label{propmonboundary} An invariant section $\Phi$ of $\mathrm{Ad}P_{l,m}$ with $l+m =1$ extends to the singular orbit $S^2 \times S^2 = SU(2)^2/ U(1)^2$ if and only if $\Phi = \phi E_1$ with $\phi$ even. 
\end{prop}
\begin{prop} \label{propconnboundary2} An invariant connection $A$ on $P_\mathrm{Id}$ extends to the singular orbit $S^3 = SU(2)^2 / \triangle SU(2)$ if an only if $A =a_1 I_1 + b_1 J_1 + a_2 I_2+ b_2 J_2 + a_0 E_1 \otimes u^- + E_1 \otimes u^+$ with $a_1$, $a_2$, $a_0$ even, $b_1, b_2$ odd, $b_1' (0)= - b_2'(0)$, $a_1 (0) - a_2(0) = a_0 (0)$, and $a_1 (0) + a_2(0) =  1$.
\end{prop} 
\begin{prop} \label{propmonboundary2} An invariant section $\Phi$ of $\mathrm{Ad}P_\mathrm{Id}$ extends to the singular orbit $S^3 = SU(2)^2 / \triangle SU(2)$ if and only if $\Phi = \phi E_1$ with $\phi$ odd.
\end{prop}
\begin{prop}
\label{propconnboundary3} An invariant connection $A$ on $P_{0, \mathrm{Id}}$, $P_{1, \mathbf{0}}$ extends to the singular orbit $S^2 = SU(2)^2 / U(1) \times SU(2)$ iff:
\begin{enumerate}[\normalfont(1)]   
\item On $P_{0,\mathrm{Id}}$, $A= a_1 I_1 + b_1 J_1 + a_2 I_2+ b_2 J_2 + a_0 E_1 \otimes u^- + E_1 \otimes u^+$ with $a_1$, $a_2$, $a_0$ $b_1$, $b_2$ even, $a_1(0)=b_1(0)= b_2(0) =0 $, $a_2 (0) = -a_0 (0) = 1$, and  $a''_0 (0) + 2 a''_2(0) =  b''_2 (0) = 0$. 
\item On $P_{1, \mathbf{0}}$, $A = a_1 I_1 + b_1 J_1 + a_2 I_2+ b_2 J_2  + a_0 E_1 \otimes u^- + n E_1 \otimes u^+$ with $a_1$, $b_1$, $a_2$, $b_2$, $a_0$ even, $a_2 (0)= b_2 (0) = 0$, $a_0(0) = 1$.
\end{enumerate}
\end{prop}
By exchanging the factors of $SU(2) \subset SU(2)^2$ in Proposition \ref{propconnboundary3}, we also obtain the following corollary:
\begin{corollary} \label{propconnboundary4} An invariant connection $A$ on $P_{\mathrm{Id},0}$ and $P_{\mathbf{0},1}$ extends to the singular orbit $S^2 = SU(2)^2 / SU(2) \times U(1)$ if and only if: 
\begin{enumerate}[\normalfont(1)]  
\item On $P_{\mathrm{Id},0}$, $A = a_1 I_1 + b_1 J_1 + a_2 I_2+ b_2 J_2 + a_0 E_1 \otimes u^- +E_1 \otimes u^+$ with $a_1$, $a_2$, $a_0$ $b_1$, $b_2$ even, $a_2(0)=b_2(0)= b_1(0) =0 $, $a_1 (0) = a_0 (0) = 1$, and  $a''_0 (0) - 2 a''_1(0) =  b''_1 (0) = 0$.
\item On $P_{\mathbf{0}, 1}$, $A = a_1 I_1 + b_1 J_1 + a_2 I_2+ b_2 J_2 + a_0 E_1 \otimes u^- +E_1 \otimes u^+$ with $a_1$, $b_1$, $a_2$, $b_2$, $a_0$ even, $a_1 (0)= b_1 (0) = 0$, and $a_0(0) = -1$.
\end{enumerate}
\end{corollary}
\begin{prop}
\label{propadboundary3} An invariant section $\Phi$ of $\mathrm{Ad} P_{0,\mathrm{Id}}$, $\mathrm{Ad} P_{1, \mathbf{0}}$ extends to the singular orbit $S^2 = SU(2)^2 / U(1) \times SU(2)$ if and only if:
\begin{enumerate}[\normalfont(1)]   
\item On $P_{0,\mathrm{Id}}$, $\Phi = \phi E_1$ with $\phi$ even, and $\phi(0) = 0$. 
\item On $P_{1, \mathbf{0}}$, $\Phi = \phi E_1$ with $\phi$ even. 
\end{enumerate}
\end{prop}
 
In the remainder of this section, we will explicitly prove \ref{propconnboundary} and \ref{propmonboundary}: we omit details of the others, since these are proved similarly. 
\begin{proof}[Proof of \ref{propconnboundary}]
We first calculate the boundary extension conditions for invariant sections of $\Omega^1 \left( \mathrm{Ad} P_{l,m} \right)$. Here, denote $\mathfrak{g} = \mathfrak{su}(2)$, $\mathfrak{p} = \langle V^1, W^1, V^2, W^2\rangle$, and $V$ the tangent space of the fibre $\mathbb{C}_{2,-2}$, spanned by the Cartesian coordinate vector fields $\tfrac{\partial}{\partial x^0}$, $\tfrac{\partial}{\partial x^1}$. Clearly $\mathfrak{p} \oplus  V $ is a $U(1)^2$-invariant splitting of the tangent space of $\mathcal{O}\left( -2,-2 \right)$, and as $U(1)^2$-representations: 
\begin{align} \label{reps0}
\mathfrak{g} = \langle E_1 \rangle \oplus \langle E_2, E_3 \rangle  \cong \mathbb{R} \oplus \mathbb{C}_{2l,2m}& &\mathfrak{p} = \langle V^1, W^1 \rangle \oplus \langle V^2, W^2 \rangle \cong \mathbb{C}_{2,0} \oplus \mathbb{C}_{0,2}& &V \cong \mathbb{C}_{2,-2}
\end{align}
while as $\triangle U(1) \times \mathbb{Z}_2$-representations\footnote{note the factor of $\mathbb{Z}_2$ in $K_{2-2} \cong \triangle U(1) \times \mathbb{Z}_2$ does not appear in the representation theory, as it always acts trivially on the tangent space and the Lie algebra of the gauge group.}: 
\begin{align}
\mathfrak{g} = \langle E_1 \rangle \oplus \langle E_2, E_3 \rangle  \cong \mathbb{R} \oplus \mathbb{C}_{2(l+m)}& &\mathfrak{p} = \langle V^1, W^1 \rangle \oplus \langle V^2, W^2 \rangle \cong \mathbb{C}_{2} \oplus \mathbb{C}_{2}& &V \cong \mathbb{R}^2
\end{align}
Recall that, since $l+m=1$, the of $\triangle U(1) \times \mathbb{Z}_2$-invariant adjoint-valued one-forms in $\mathfrak{g} \otimes \left( V^* \oplus \mathfrak{p}^* \right)$  is spanned by the real and imaginary parts of $E_1 \otimes \left( dx^0 + i dx^1 \right)$, $\left( E_2 + iE_3 \right) \otimes \left( v^1 - iw^1 \right), \left( E_2 + iE_3 \right) \otimes \left( v^2 - iw^2 \right)$. To apply the power-series analysis of \cite[Lemma 1.1]{Eschenburg2000}, we use \eqref{reps0} to look for a basis in terms of $U(1)^2$-equivariant homogeneous polynomials $p: V \rightarrow \mathfrak{g} \otimes \left( V^* \oplus \mathfrak{p}^* \right)$, evaluated at $1 \in V = \mathbb{C}$.

First assume $l>0$. By making the identification of the fibre Cartesian coordinate one-forms $dx^0=dt$ and $dx^1 =3t\eta^{se} =4tu^-$ along $\gamma(t) = (1,1,t) \in SU(2)^2 \times \mathbb{C}$, and by taking real and imaginary parts, we obtain the following splitting: 
\begin{center}
\begin{tabular}{ c|c|c } 

 degree & polynomial $p(z)$ & evaluation at $z=1$ \\ 
 \hline
 $1$ & $z E_1 \otimes \left( dx^0+ i dx^1 \right)$ & $E_1 \otimes dt, E_1 \otimes 4 t u^-$ \\ 
 $l-1$ & $z^{l-1} \left( E_2 + i E_3 \right)  \otimes \left( v^1 - iw^1 \right)$ & $E_2 \otimes v^1 + E_3 \otimes w^1, - E_2 \otimes w^1 + E_3 \otimes v^1$\\ 
 $l$ & $z^{l} \left( E_2 + i E_3 \right)  \otimes \left( v^2 - iw^2 \right)$ & $E_2 \otimes v^2 + E_3 \otimes w^2, - E_2 \otimes w^2 + E_3 \otimes v^2$ \\
\end{tabular}
\end{center}
We can recover the case $l\leq0$ by exchanging $P_{l,m} \mapsto P_{m,l}$, since clearly, the polynomials of degree $l-1$ and $l$ are exchanged by map.
 
We now apply this calculation to invariant connection $A$ of the proposition: the canonical connection $d\lambda$ on $P_{l,m}$ is given by $d\lambda = l E_1 \otimes u^1 +  m E_1 \otimes u^2$, so writing the $SU(2)^2$-invariant connection $A$ as an invariant section $A - d \lambda \in \Omega^1 \left(\text{Ad}P_{l,m} \right)$, we get: 
\begin{align*} 
A - d \lambda &= a_1 ( E_2 \otimes v^1 + E_3 \otimes w^1 )+ b_1 ( E_3 \otimes v^1 - E_2 \otimes w^1 ) \\ &+ a_2 ( E_2 \otimes v^2 + E_3 \otimes w^2 )+ b_2 ( E_3 \otimes v^2 - E_2 \otimes w^2 ) + (a_0 -2l +1) E_1 \otimes u^-
\end{align*}
So if $l>0$, we require $a_0(0) = 2l -1$, $a_0$ be even, $a_1, b_1$ to have degree $l-1$ and $a_2, b_2$ to have degree $l$ to extend $A$. Again, one gets the corresponding claim for $l \leq 0$ by exchanging the factors of $SU(2)$. 
\end{proof}
\begin{proof}[Proof of \ref{propmonboundary}]
The degree of a function appearing as the coefficient of an $SU(2)^2$-invariant element in $\Omega^0 \left( \mathrm{Ad} P_{l,m} \right)$ on the principal orbits is determined by a $U(1)^2$-equivariant homogeneous polynomial from $V = \mathbb{C}_{2,-2}$ to $\mathfrak{g} = \langle E_1 \rangle \oplus \langle E_2, E_3 \rangle  \cong \mathbb{R} \oplus \mathbb{C}_{2l,2m}$. 

When $l+m=1$, there is a single degree zero polynomial given by the constant map $E_1$, so the invariant section $\Phi = \phi_1 E_1$ must have $\phi_1$ even.   
\end{proof}

\section{Singular Initial Value Problems} \label{singIVPs}
For use in the following proof, we note that the Calabi-Yau structure on $T^* S^3$ is given by \eqref{hypoBsol}, and we compute the power-series near $t=0$ of the following expressions:
\begin{align*}
&\frac{4 \lambda}{ \mu} = \frac{2}{t} + O(t)& &\frac{3 \left( v_3 - v_0 \right)}{ 2 \lambda \mu} = \frac{1}{t} + O(t)& &\frac{3 \left( v_3 + v_0 \right)}{ 2 \lambda \mu} = \frac{9}{4}t + O(t^3)& \\ & \frac{3 v_0}{ \mu^2} = - \frac{1}{2 t^2} + \frac{3}{2} + O (t^2)& &\frac{3 v_3}{ \mu^2} = \frac{1}{2 t^2} + \frac{3}{4} + O (t^2) 
\end{align*}
\begin{proof}[Proof of Proposition \ref{localmonopolesmoothing}] 
Let $\left(a_0, a_1, a_2, \phi\right)$ be a solution to \eqref{dynamicODE} on $T^* S^3$. Using Prop. \ref{propconnboundary2}, define smooth functions $a_-, A_+, \psi$ such that $a_1 -a_2 = a_-$, $a_1 + a_2 = 1 + t^2 A_+$,  $\phi = t \psi$. Then $y(t) = \left( a_0, \psi, A_+, a_- \right)$ must satisfy
a singular initial value problem with linearisation:
\begin{align*}
d_{y_0} M_{-1} = \begin{pmatrix} 
-1 & 0 & 0 & 1 \\
0 & -1 & -1 &  \frac{9}{4} \xi \\
\frac{9}{4} \xi & -2 & -2 & \frac{9}{4} \xi \\
 2 & 0 & 0 & -2 
\end{pmatrix}
\end{align*}
at initial value $y_0 = \left( \xi, \frac{9}{8} (\xi^2 -1) - \chi, \xi, \chi \right)$ for some $\xi, \chi \in \mathbb{R}$. This initial value problem has a unique solution once we fix $y_0$, since $\det ( k \mathrm{Id} - d_{y_0} M_{-1} ) = \left(k+3\right)^2 k^2$. 
\end{proof}
For use in the following proofs, we note that the Calabi-Yau structure on $\mathcal{O}(-1)\oplus \mathcal{O}(-1)$ is given by  \eqref{CYstructure} with $U_1 = - U_0 = - u_0 = 1$, and the power-series of $\lambda$, $u_1$, $\mu$ near $t=0$ satisfy:
\begin{align*} 
\lambda(t)= \frac{3}{2}t + O(t^3)& &u_1 = 1 + \frac{3}{2}t^2 + O(t^4)& & \mu = \sqrt{3} t + O(t^3) 
\end{align*}
\begin{proof}[Proof of Propositions \ref{localresol6}]
Let $\left(a_0, a_1, a_2, \phi\right)$ be a solution to \eqref{dynamicODE} on $\mathcal{O}(-1)\oplus \mathcal{O}(-1)$. Using Prop. \ref{propconnboundary3} for extending on $P_{1, \mathbf{0}}$, $y(t) = \left(a_0, a_1, a_2, \phi\right)$ satisfies a singular initial value problem with linearisation:
\begin{align*}
d_{y_0} M_{-1} =  \begin{pmatrix} 
-2 & 0 & 0 & 0 \\
0 & 0 & 0 &  - 2 \sqrt{3} \epsilon' \\
\epsilon' & 0 & 0 & -\frac{4}{\sqrt{3}} \delta' \\
 0 & 0 & 0 & -2 
\end{pmatrix}
\end{align*}
at initial value $y_0 = \left( 1, \epsilon' , 0, \delta' \right)$ for some $\epsilon', \delta' \in \mathbb{R}$. This initial value problem has a unique solution once we fix $y_0$, since $\det ( k \mathrm{Id} - d_{y_0} M_{-1} ) = \left(k+2\right)^2 k^2$. 
\end{proof}
\begin{proof}[Proof of Proposition \ref{localresol1}]
Let $\left(a_0, a_1, a_2, \phi\right)$ be a solution to \eqref{dynamicODE} on $\mathcal{O}(-1)\oplus \mathcal{O}(-1)$. Using Prop. \ref{propconnboundary2} for extending on $P_{0,\mathrm{Id}}$, we define smooth functions $X_0, X_1, X_2, \psi$ such that $a_0 = -1 + t^2 X_0$, $a_1 = t^2 X_1$, $a_2 = 1 + t^2 X_2$, and $\phi = t^2 \psi$. Then $y(t) = \left(X_0, X_1, X_2, \psi \right)$ satisfies a singular initial value problem with linearisation: 
\begin{align*}
d_{y_0} M_{-1} = \begin{pmatrix} -4 & 0 & 0 & -8 \\
0 & -2 & -2\sqrt{3} & 0 \\
0 & -\frac{4}{\sqrt{3}} & -4 & 0 \\
 -1 & 0 & 0 & -2 
\end{pmatrix}
\end{align*}
at initial value $y_0 = \left( \epsilon, -\tfrac{1}{\sqrt{3}}\delta, - \tfrac{1}{2} \epsilon, \delta \right)$. This initial value problem has a unique solution once we fix $y_0$, since $\det ( k \mathrm{Id} - d_{y_0} M_{-1} ) =  \left(k + 6 \right)^2 k^2$. 
\end{proof}
\bibliographystyle{alpha}
\bibliography{diffgeo1} 
\end{document}